\documentclass[10pt,reqno]{amsart}
\usepackage{amsmath,amssymb,latexsym,esint,cite,mathrsfs}
\usepackage{verbatim,wasysym,cite,relsize,color}
\usepackage[left=2.5cm,right=2.5cm,top=2.5cm,bottom=2.5cm]{geometry}
\usepackage[latin1]{inputenc}
\usepackage{microtype}
\usepackage{color,enumitem,graphicx}
\usepackage[colorlinks=true,urlcolor=blue, citecolor=red,linkcolor=blue,
linktocpage,pdfpagelabels, bookmarksnumbered,bookmarksopen]{hyperref}
\usepackage[hyperpageref]{backref}
\usepackage[english]{babel}
\usepackage{tikz}
\usepackage[font=small,skip=5pt]{caption}
\numberwithin{equation}{section}
\newtheorem{theorem}{Theorem}[section]
\newtheorem{lemma}[theorem]{Lemma}
\newtheorem{proposition}[theorem]{Proposition}
\newtheorem{corollary}[theorem]{Corollary}

\theoremstyle{definition}
\newtheorem{definition}[theorem]{Definition}
\newtheorem{remark}[theorem]{Remark}

\newcommand{\R}{\mathbb R}
\newcommand{\C}{\mathbb C}

\newcommand{\Fc}{\mathcal F}
\newcommand{\Lc}{\mathcal L}
\newcommand{\Mcal}{\mathcal M}
\newcommand{\Vc}{\mathcal V}
\newcommand{\varep}{\varepsilon}

\newcommand{\betc}{\beta_{\text{c}}}

\newcommand{\gamc}{\gamma_{\text{c}}}

\newcommand{\ret}[1]{\mbox{Re}  \left(#1\right)} 
\newcommand{\imt}[1]{\mbox{Im}  \left(#1\right)} 
\newcommand{\ree}[1]{\mbox{\emph{Re}}  \left(#1\right)} 
\newcommand{\ime}[1]{\mbox{\emph{Im}}  \left(#1\right)}
\newcommand{\scal}[1]{\left\langle #1 \right\rangle}

\title[NLS Repulsive Inverse-power Potentials]
{On nonlinear Schr\"odinger equations with repulsive inverse-power potentials}
\author[V. D. Dinh]{Van Duong Dinh}
\address[V. D. Dinh]{Laboratoire Paul Painlev\'e UMR 8524, Universit\'e de Lille CNRS, 59655 Villeneuve d'Ascq Cedex, France
and 
Department of Mathematics, HCMC University of Pedagogy, 280 An Duong Vuong, Ho Chi Minh, Vietnam}
\email{contact@duongdinh.com}


\begin{document}
	
	\begin{abstract}
	In this paper, we consider the Cauchy problem for the nonlinear Schr\"odinger equations with repulsive inverse-power potentials
	\[
	i \partial_t u + \Delta u - c |x|^{-\sigma} u = \pm |u|^\alpha u, \quad c>0.
	\]
	We study the local and global well-posedness, finite time blow-up and scattering in the energy space $H^1$ for the equation. These results extend a recent work of Miao-Zhang-Zheng [{\bf Nonlinear Schr\"odinger equation with coulomb potential}, \url{arXiv:1809.06685}] to a general class of inverse-power potentials and higher dimensions.
	\end{abstract}

	\maketitle

	\section{Introduction}
	\label{S1}
	We consider the Cauchy problem for the nonlinear Schr\"odinger equations with repulsive inverse-power potentials
	\begin{equation} \label{NLS-rep}
		\left\{ 
		\begin{array}{rcl}
			i\partial_t u + \Delta u - c |x|^{-\sigma} u &=& \pm |u|^\alpha u, \quad (t,x) \in \R \times \R^d, \\
			u(0)&=& u_0,
		\end{array}
		\right.
	\end{equation}
	where $u: \mathbb{R} \times \mathbb{R}^d \rightarrow \mathbb{C}$, $u_0: \mathbb{R}^d \rightarrow \mathbb{C}$, $c>0$, $0<\sigma<\min\{2,d\}$ and $\alpha>0$. The plus and minus signs in front of the nonlinearity correspond to the defocusing and focusing cases respectively. 
	
	This paper is motivated by recent works of Mizutani \cite{Mizutani} and Miao-Zhang-Zheng \cite{MZZ} where the authors investigate the effect of slowly decaying potentials in linear and nonlinear Schr\"odinger equations. In \cite{Mizutani}, global-in-time Strichartz estimates for a class of slowly decaying potentials including the repulsive inverse-power potentials $c|x|^{-\sigma}, c>0$ and $0<\sigma<2$ was shown in dimensions $d\geq 3$. In \cite{MZZ}, the Cauchy problem including the global well-posedness, finite time blow-up and scattering in the energy space $H^1$ for the nonlinear Schr\"odinger equation with coulomb potential $c|x|^{-1}, c\in \R$ was studied in dimension 3. 
	
	The Schr\"odinger equations with inverse-power potentials have attracted a lot of interest in the past decades (see e.g. \cite{BDZ, BPST, CG, Dinh-inv, KMVZZ, KMVZ, KMVZZ-sob, LMM, OSY, ZZ, Zheng} for the inverse-square potential $\sigma=2$, \cite{BJ, ChG, HO, LL, Lions, MZZ} for the coulomb potential $\sigma=1$, and \cite{FO, GWY, Mizutani} for the slowly decaying potentials $0<\sigma<2$). 
	
	In this paper, we will study the Cauchy problem for \eqref{NLS-rep} in the energy space $H^1$. Before stating our results, let us recall some facts for the nonlinear Schr\"odinger equation without potential, i.e. $c=0$, namely
	\begin{equation} \label{NLS-intro}
	\left\{ 
	\begin{array}{rcl}
	i\partial_t u + \Delta u &=& \pm |u|^\alpha u, \quad (t,x) \in \R \times \R^d, \\
	u(0)&=& u_0.
	\end{array}
	\right.
	\end{equation}
	We first note that \eqref{NLS-intro} enjoys the following scaling invariance
	\[
	u_\lambda(t,x) := \lambda^{\frac{2}{\alpha}} u(\lambda^2 t, \lambda x), \quad \lambda>0.
	\]
	This scaling leaves the $\dot{H}^{\gamc}$-norm of initial data invariant, i.e. $\|u_\lambda(0)\|_{\dot{H}^{\gamc}} = \|u_0\|_{\dot{H}^{\gamc}}$, where
	\begin{align} \label{cri-exp}
	\gamc:= \frac{d}{2}-\frac{2}{\alpha}.
	\end{align}
	When $\gamc=0$ or $\alpha =\frac{4}{d}$, \eqref{NLS-intro} is called mass-critical. When $0<\gamc<1$ or $\frac{4}{d}<\alpha<\frac{4}{d-2}$ if $d\geq 3$ ($\frac{4}{d}<\alpha<\infty$ if $d=1,2$), \eqref{NLS-intro} is called intercritical, and when $\gamc=1$ or $\alpha=\frac{4}{d-2}$ and $d\geq 3$, \eqref{NLS-intro} is called energy-critical. 
	
	For sufficiently regular initial data, says e.g. $u_0 \in H^1$, the equation \eqref{NLS-intro} has the following conserved quantities
	\begin{align*}
	M(u(t)) &:= \int |u(t,x)|^2 dx = M(u_0), \\
	E_0(u(t))&:= \frac{1}{2} \int |\nabla u(t,x)|^2 dx \pm \frac{1}{\alpha+2} \int |u(t,x)|^{\alpha+2} dx = E_0(u_0).
	\end{align*}
	Let us briefly recall the global well-posedness in $H^1$ for \eqref{NLS-intro}. In the energy-subcritical case, i.e. $0<\alpha <\frac{4}{d-2}$ if $d\geq 3$ ($0<\alpha <\infty$ if $d=1,2$), it follows from the local theory that the time of existence depends only on the $H^1$-norm of initial data. Thus, by the conservation of mass, the local solutions can be extended globally in time if one has the uniform bound $\|\nabla u(t)\|_{L^2} \leq C$ for any $t$ in the existence time. In the defocusing case, this uniform bound follows immediately from the conservation of energy. While in the focusing case, one makes use of the sharp Gagliardo-Nirenberg inequality 
	\[
	\|f\|^{\alpha+2}_{L^{\alpha+2}} \leq C_{\text{GN}} \|\nabla f\|^{\frac{d\alpha}{2}}_{L^2} \|f\|^{\frac{4-(d-2)\alpha}{2}}_{L^2}, \quad f \in H^1,
	\]
	where the sharp constant $C_{\text{GN}}$ is attained by a function $Q$ which is the unique (up to symmetries) positive radial solution to the elliptic equation
	\begin{align} \label{ell-equ-Q}
	\Delta Q- Q +|Q|^\alpha Q=0
	\end{align}
	to obtain the uniform bound for
	\begin{itemize}
		\item $0<\alpha<\frac{4}{d}$;
		\item $\alpha = \frac{4}{d}$ and $\|u_0\|_{L^2} <\|Q\|_{L^2}$;
		\item $\frac{4}{d}<\alpha<\frac{4}{d-2}$ if $d \geq 3$ ($\frac{4}{d}<\alpha<\infty$ if $d=1,2$) and 
		\[
		E_0(u_0) M^{\betc}(u_0) < E_0(Q) M^{\betc}(Q), \quad \|\nabla u_0\|_{L^2} \|u_0\|^{\betc}_{L^2} <\|\nabla Q\|_{L^2} \|Q\|_{L^2}^{\betc},
		\]
		where
		\begin{align} \label{def-bet-c}
		\betc:= \frac{1-\gamc}{\gamc} = \frac{4-(d-2)\alpha}{d\alpha-4}.
		\end{align}
	\end{itemize}
	In the energy-critical case, i.e. $\alpha=\frac{4}{d-2}$ and $d\geq 3$, the local theory asserts that the time of existence depends not only on the $H^1$-norm of initial data but also on its profile. The global well-posedness is therefore more difficult. In the defocusing case, the global well-posedness and scattering for any data in $\dot{H}^1$ was shown in celebrated papers of Colliander-Keel-Staffilani-Takaoka-Tao \cite{CKSTT}, Ryckman-Visan \cite{RV} and Visan \cite{Visan}. In the focusing case, the global well-posedness and scattering was first proved by Kenig-Merle \cite{KM} in dimensions $3,4,5$ for radial initial data $u_0 \in \dot{H}^1$ satisfying
	\begin{align} \label{glo-con-intro}
	E_0(u_0) < E_0(W), \quad \|\nabla u_0\|_{L^2} <\|\nabla W\|_{L^2},
	\end{align}
	where 
	\begin{align} \label{def-W}
	W(x) = \left( 1+ \frac{|x|^2}{d(d-2)}\right)^{-\frac{d-2}{2}}
	\end{align}
	solves the elliptic equation 
	\begin{align} \label{ell-equ-W}
	\Delta W  + |W|^{\frac{4}{d-2}} W=0.
	\end{align}
	Later, Killip-Visan \cite{KV} extended this result to dimensions greater than or equal to 5 and for any initial data $u_0 \in \dot{H}^1$ (not necessary radial) satisfying \eqref{glo-con-intro}. Recently, Dodson \cite{Dodson} improved the result of \cite{KM} for non-radial initial data in $\dot{H}^1$ in the fourth dimensional case.
	
	We now turn our attention to \eqref{NLS-rep}. Due to the appearance of inverse-power potentials, the equation \eqref{NLS-rep} does not enjoy the scaling invariance. However, for initial data $u_0 \in H^1$, the equation \eqref{NLS-rep} still has the conservation of mass and energy 
	\begin{align} \label{mas-ene}
	\begin{aligned}
	M(u(t)) &:= \int |u(t,x)|^2 dx = M(u_0),  \\
	E(u(t))&:= \frac{1}{2} \int |\nabla u(t,x)|^2 dx +\frac{c}{2} \int |x|^{-\sigma} |u(t,x)|^2 dx \pm \frac{1}{\alpha+2} \int |u(t,x)|^{\alpha+2} dx = E(u_0). 
	\end{aligned}
	\end{align}
	
	In the energy-subcritical case, the local well-posedness (LWP) in $H^1$ for \eqref{NLS-rep} can be shown easily by using the energy method which does not use Strichartz estimates (see Proposition $\ref{prop-lwp-sub}$). This method allows us to show the existence of local solutions in any dimensions $d\geq 1$. However, we do not know whether or not the  local solutions belong to $L^p_{\text{loc}}((-T_*, T^*), W^{1,q})$ for any Schr\"odinger admissible pair $(p,q)$, where $(-T_*, T^*)$ is the maximal time interval. To ensure the local solutions satisfying this property, we make use of Strichartz estimates for the free Schr\"odinger operator $e^{it\Delta}$ and view the potential as a nonlinear perturbation term. Due to the appearance of singular potential $|x|^{-\sigma}$ which does not belong to any Lebesgue spaces, a good way is to use Strichartz estimates in Lorentz spaces. It leads to a restriction on the validity of $\sigma$ and $d$ (see Proposition $\ref{prop-lwp-sub-lor}$) which comes from Sobolev embeddings in Lorentz spaces (see Corollary $\ref{coro-sob-emb-lor}$). Another way to show the LWP in $H^1$ is to use Strichartz estimates for the Schr\"odinger operator $e^{-itH_c}$ (see after \eqref{klmn-con} for the meaning of $H_c$) and the equivalence between the usual Sobolev norms and the ones associated to $H_c$, namely
	\begin{align} \label{equ-sob-nor-intro}
	\left\|\scal{H_c} u \right\|_{L^q} \sim \|\scal{\nabla} u\|_{L^q}, \quad  1<q<\frac{2d}{\sigma}.
	\end{align}
	Due to the requirement of the Sobolev norms equivalence, we are not able to show the local solutions satisfying $L^p_{\text{loc}}((-T_*, T^*), W^{1,q})$ for any Schr\"odinger admissible pair $(p,q)$. As in the usual local theory, the above methods give the blow-up alternative, that is if the maximal time of existence is finite, then the kinetic energy $\|\nabla u(t)\|^2_{L^2}$ goes to infinity as time tends to the maximal value. This allows us to obtain global solutions by extending the local ones as long as we have the uniform bound $\sup_{t\in (-T_*,T^*)} \|\nabla u(t)\|_{L^2} \leq C$ for some constant $C>0$.
	
	In the energy-critical case, the energy method does not work, we thus rely mainly on Strichartz estimates. Using Strichartz estimates for $e^{-itH_c}$ and the Sobolev norms equivalence \eqref{equ-sob-nor-intro}, we show the existence of local $H^1$ solutions (see Proposition $\ref{prop-sma-dat-sca}$). However, the time of existence depends not only on the $H^1$-norm of initial data but also on its profile. This implies that even we have a uniform control on the kinetic energy, we cannot obtain global solutions simply by extending the local ones as in the energy-subcritical case. The interest of this method is that we are able to show the global well-posedness and scattering in $H^1$ for small initial data. Another interesting method is to use Strichartz estimates for $e^{it\Delta}$ and view the potential as a nonlinear energy-subcritical perturbation term. The pertubation argument of Zhang \cite{Zhang} allows us to show the ``good" local well-posedness for \eqref{NLS-rep} in the energy-critical case. Here the ``good" LWP means that the time of existence depends only on the $H^1$-norm of the initial data. This facts allows us to extend local solutions to global ones provided that the uniform bound on the kinetic energy holds. The idea of this perturbation argument is as follows. Since the energy-critical \eqref{NLS-rep}, i.e. $\alpha=\frac{4}{d-2}$ and $d\geq 3$, is invariant under the time translation, it suffices to show the well-posedness on the time interval $[0,T]$ for some small $T=T(\|u_0\|_{H^1})$ depending only on the $H^1$-norm of initial data. On the time interval $[0,T]$, we approximate \eqref{NLS-rep} by the energy-critical \eqref{NLS-intro}, namely
	\begin{align} \label{ene-cri-NLS}
	\left\{
	\begin{array}{rcl}
	i\partial_t v + \Delta v &=& \pm|v|^{\frac{4}{d-2}} v, \quad (t,x) \in \R \times \R^d,\\
	v(0)&=& u_0.
	\end{array}
	\right.
	\end{align}
	which is globally well-posed in the defocusing case (see \cite{CKSTT, RV, Visan}) for any initial data in $H^1$ and in the focusing case (see \cite{KM, KV, Dodson}) for initial data in $H^1$ satisfying
	\[
	E_0(u_0)<E_0(W), \quad \|\nabla u_0\|_{L^2} < \|\nabla W\|_{L^2},
	\]
	and an additional radial assumption when $d=3$. By choosing $T$ small enough depending only on $\|u_0\|_{H^1}$, we can show that the difference problem of $u-v$ with zero initial data is solvable and the solution stays small on $[0,T]$. We refer the reader to Section $\ref{S3}$ for more details on the local well-posedness results.
	
	Concerning the global well-posedness for \eqref{NLS-rep} in the energy space $H^1$, we have the following result. 
	\begin{theorem}[Global well-posedness] \label{theo-GWP}
		Let $c>0$ and $u_0 \in H^1$. Suppose that
		\begin{itemize}
			\item {in the defocusing case:}
			\begin{itemize}
				\item (Energy-subcritical case) $0<\sigma<\min\{2,d\}$ and $0<\alpha<\frac{4}{d-2}$ if $d\geq 3$ ($0<\alpha<\infty$ if $d=1,2$);
				\item (Energy-critical case) $0<\sigma<2$ if $d\geq 4$ ($0<\sigma<\frac{3}{2}$ if $d=3$) and $\alpha=\frac{4}{d-2}$;
			\end{itemize}
			\item {in the focusing case:}
			\begin{itemize}
				\item (Mass-subcritical case) $0<\sigma <\min\{2,d\}$ and $0<\alpha<\frac{4}{d}$;
				\item (Mass-critical case) $0<\sigma< \min\{2, d\}$, $\alpha=\frac{4}{d}$ and $\|u_0\|_{L^2}<\|Q\|_{L^2}$;
				\item (Intercritical case) $0<\sigma<\min\{2,d\}$, $\frac{4}{d}<\alpha<\frac{4}{d-2}$ if $d\geq 3$ ($\frac{4}{d}<\alpha<\infty$ if $d=1,2$) and
				\begin{align} \label{glo-con-int}
				E(u_0) M^{\betc}(u_0) < E_0(Q) M^{\betc}(Q), \quad \|\nabla u_0\|_{L^2} \|u_0\|^{\betc}_{L^2} <\|\nabla Q\|_{L^2} \|Q\|_{L^2}^{\betc};
				\end{align}
				\item (Energy-critical case) $0<\sigma <2$ if $d\geq 4$ ($0<\sigma<\frac{3}{2}$ if $d=3$), $\alpha=\frac{4}{d-2}$ and
				\begin{align} \label{glo-con-ene}
				E(u_0) <E_0(W), \quad \|\nabla u_0\|_{L^2} < \|\nabla W\|_{L^2},
				\end{align}
				and when $d=3$ we assume in addition that $u_0$ is radially symmetric.
			\end{itemize}
		\end{itemize}
		Then there exists a unique global solution to \eqref{NLS-rep}. Moreover, the global solution $u$ satisfies for $0<\sigma <2$ if $d\geq 4$ ($0<\sigma<\frac{3}{2}$ if $d=3$) and any compact interval $J\subset \R$,
		\begin{align} \label{glo-bou-est}
		\sup_{(p,q) \in S} \|u\|_{L^p(J, W^{1,q})} \leq C(\|u_0\|_{H^1}, |J|),
		\end{align}
		where $(p,q) \in S$ means that $(p,q)$ is a Schr\"odinger admissible pair.
	\end{theorem}
	
	The proof of Theorem $\ref{theo-GWP}$ is based on the ``good" local well-posedness for \eqref{NLS-rep} in $H^1$ in which the time of existence depends only on the $H^1$-norm of initial data and the uniform bound $\|\nabla u(t)\|_{L^2} \leq C$ for any $t$ in the existence time. In the energy-subcritical case, the ``good" local well-posedness coincides with the usual local well-posedness. In the energy-critical case, this ``good" local well-posedness is proved using the argument of Zhang \cite{Zhang} as mentioned above. The bound \eqref{glo-bou-est} follows from the local well-posedness by using Strichartz estimates for $e^{it\Delta}$ in Lorentz spaces (see Proposition $\ref{prop-lwp-sub-lor}$ and Proposition $\ref{prop-good-lwp}$).
	
	Although we mainly focus on the repulsive inverse-power potentials, we also have the following global well-posedness in the energy space for the attractive inverse-power potentials.
	\begin{proposition} \label{prop-gwp-att}
		Let $c<0$ and $u_0 \in H^1$. Suppose that
		\begin{itemize}
			\item in the defocusing case:
			\begin{itemize}
				\item (Energy-subcritical case) $0<\sigma<\min\{2,d\}$ and $0<\alpha<\frac{4}{d-2}$ if $d \geq 3$ ($0<\alpha<\infty$ if $d=1,2$);
				\item (Energy-critical case) $0<\sigma<2$ if $d\geq 4$  ($0<\sigma<\frac{3}{2}$ if $d=3$) and $\alpha=\frac{4}{d-2}$;
			\end{itemize}
			\item in the focusing case:
			\begin{itemize}
				\item (Mass-subcritical case) $0<\sigma<\min\{2,d\}$ and $0<\alpha<\frac{4}{d}$;
				\item (Mass-critical case) $0<\sigma<\min\{2,d\}$, $\alpha=\frac{4}{d}$ and $\|u_0\|_{L^2} <\|Q\|_{L^2}$.
			\end{itemize}
		\end{itemize}
		Then there exists a unique global solution to \eqref{NLS-rep}. Moreover, the global solution $u$ satisfies for $0<\sigma<2$ if $d\geq 4$ ($0<\sigma<\frac{3}{2}$ if $d=3$) and any compact interval $J\subset \R$,
		\[
		\sup_{(p,q)\in S} \|u\|_{L^p(J, W^{1,q})} \leq C(\|u_0\|_{H^1}, |J|),
		\]
		where $(p,q) \in S$ means that $(p,q)$ is a Schr\"odinger admissible pair.
	\end{proposition}
	
	As a complement for the global well-posedness given in Theorem $\ref{theo-GWP}$, we have the following finite time blow-up in the energy space $H^1$ for \eqref{NLS-rep} in the focusing case.
	
	\begin{theorem} [Blow-up] \label{theo-blo-up}
		Let $c>0$ and $u_0 \in H^1$. Suppose that
		\begin{itemize}
			\item (Mass-critical case) $0<\sigma<\min\{2,d\}$, $\alpha=\frac{4}{d}$, $|x|u_0 \in L^2$ with $d\geq 1$ or $u_0$ is radial with $d\geq 2$ and $E(u_0)<0$;
			\item (Intercritical case) $0<\sigma<\min\{2,d\}$, $\frac{4}{d}<\alpha<\frac{4}{d-2}$ if $d\geq 3$ ( $\frac{4}{d}<\alpha<\infty$ if $d=1,2$), $|x|u_0 \in L^2$ with $d\geq 1$ or $u_0$ is radial with $d\geq 2$ and $E(u_0)<0$ or if $E(u_0)\geq 0$, we assume that
			\begin{align} \label{blo-up-con-int}
			E(u_0) M^{\betc}(u_0) < E_0(Q) M^{\betc}(Q), \quad \|\nabla u_0\|_{L^2} \|u_0\|^{\betc}_{L^2} > \|\nabla Q\|_{L^2} \|Q\|^{\betc}_{L^2},
			\end{align}
			and in the case $u_0$ is radial we assume in addition that $\alpha \leq 4$;
			\item (Energy-critical case) $0<\sigma<2$ if $d\geq 4$ ($0<\sigma<\frac{3}{2}$ if $d=3$), $\alpha=\frac{4}{d-2}$, $|x|u_0 \in L^2$ or $u_0$ is radial and $E(u_0)<0$ or if $E(u_0)\geq 0$, we assume that
			\begin{align} \label{blo-up-con-ene}
			E(u_0)<E_0(W), \quad \|\nabla u_0\|_{L^2} > \|\nabla W\|_{L^2}.
			\end{align}
		\end{itemize}
		Then the corresponding solution to \eqref{NLS-rep} in the focusing case blows up in finite time.
	\end{theorem}

	The proof of Theorem $\ref{theo-blo-up}$ is based on the virial identity and localized virial estimates related to \eqref{NLS-rep} in the focusing case. This result extends the well-known finite time blow-up  of the focusing nonlinear Schr\"odinger equation without potential. The only different point is that we are not able to prove the finite time blow-up for the focusing 1D mass-critical \eqref{NLS-rep} due to the lack of scaling invariance. We refer the reader to Section $\ref{S5}$ and Section $\ref{S7}$ for more details.
	
	Our last result is the scattering in the energy space $H^1$ for \eqref{NLS-rep} in the defocusing case. To state this result, we first notice that for $d\geq 3$, the potential $c|x|^{-\sigma}$ with $c>0$ and $0<\sigma<2$ generates a symmetric quadratic form on $Q(-\Delta)=H^1$. This quadratic form satisfies for any $a>0$, there exists $b \in \R$ such that
	\begin{align} \label{klmn-con}
	\scal{\varphi,c|x|^{-\sigma} \varphi} \leq a \scal{\varphi, -\Delta \varphi} + b \scal{\varphi,\varphi}, \quad \forall \varphi \in D(-\Delta)=H^2. 
	\end{align}
	By the KLMN Theorem (see e.g. \cite[Theorem X.17]{RS}), there exists a unique self-adjoint extension of $-\Delta + c|x|^{-\sigma}$ in $L^2$, denoted by $H_c$, whose domain form is $Q(-\Delta)$ and its core is $C^\infty_0(\R^d)$. To see \eqref{klmn-con}, we recall the Hardy's inequality that for $d\geq 3$,
	\[
	\left(\frac{d-2}{2}\right)^2 \int |x|^{-2} |\varphi(x)|^2 dx \leq \int |\nabla \varphi(x)|^2 dx, \quad \forall \varphi \in H^1.
	\]
	Now given any $a>0$, we choose $R>0$ such that $c|x|^{-\sigma} \leq a \left(\frac{d-2}{2}\right)^2 |x|^{-2}$ for all $|x| \leq R$. This together with Hardy's inequality imply that for any $\varphi \in D(-\Delta) = H^2$,
	\begin{align*}
	\scal{\varphi,c|x|^{-\sigma} \varphi} =\int c|x|^{-\sigma} |\varphi(x)|^2 dx &\leq a \left(\frac{d-2}{2}\right)^2 \int_{|x|\leq R} |x|^{-2} |\varphi(x)|^2 dx + cR^{-\sigma} \int_{|x|>R} |\varphi(x)|^2 dx \\
	&\leq a \int |\nabla \varphi(x)|^2 dx + cR^{-\sigma} \int |\varphi(x)|^2 dx \\
	&= a \scal{\varphi,-\Delta \varphi} + cR^{-\sigma} \scal{\varphi,\varphi}.
	\end{align*}
	This shows \eqref{klmn-con} with $b=cR^{-\sigma}$.
	
	Recently, Mizutani \cite{Mizutani} proved global-in-time Strichartz estimates for a class of slowly decaying potentials including the repulsive inverse-power potentials $c|x|^{-\sigma}$ with $c>0$ and $0<\sigma<2$ in dimensions $d\geq 3$. These global estimates allow us to study the long time behavior of global solutions to \eqref{NLS-rep}. As a consequence of these global estimates and the Sobolev norms equivalence \eqref{equ-sob-nor-intro}, one can show easily the small data scattering for \eqref{NLS-rep} (see e.g. Proposition $\ref{prop-sma-dat-sca}$ for the energy-critical case). Note that the Sobolev norm equivalence \eqref{equ-sob-nor-intro} follows from the generalized Hardy's inequality (see e.g. \cite{ZZ}) and the Gaussian upper bound of the kernel of the heat operator $e^{-tH_c}$. We refer the reader to Section $\ref{S2}$ for more details.	
	
	For large data, we have the following asymptotic completeness (or energy scattering) for \eqref{NLS-rep} in the defocusing intercritical case.
	
	\begin{theorem}[Energy scattering] \label{theo-ene-sca}
		Let $d\geq 3, c>0$, $0<\sigma<2$ and $\frac{4}{d}<\alpha<\frac{4}{d-2}$.
		Let $u_0 \in H^1$ and $u$ be the corresponding global solution to \eqref{NLS-rep} in the defocusing case. Then there exists $u_0^\pm \in H^1$ such that 
		\[
		\lim_{t\rightarrow \pm \infty} \|u(t) - e^{-itH_c} u_0^\pm\|_{H^1} =0.
		\]
	\end{theorem}
	
	The proof of this result is based on global-in-time Strichartz estimates, the interaction Morawetz inequality 
	\begin{align} \label{int-mor-ine-intro}
	\||\nabla|^{-\frac{d-3}{2}} u\|_{L^4(J, L^4)} \leq C\|u\|^{\frac{3}{4}}_{L^\infty(J,L^2)} \|\nabla u\|^{\frac{1}{4}}_{L^\infty(J,L^2)}
	\end{align}
	and the Sobolev norm equivalence \eqref{equ-sob-nor-intro}. The interaction Morawetz inequality \eqref{int-mor-ine-intro} for \eqref{NLS-rep} in the defocusing case follows from the same argument for the defocusing \eqref{NLS-intro} as in \cite{CGT}. Unlike the nonlinear Schr\"odinger equation without potential, the equation \eqref{NLS-rep} is not invariant under the space translation. Consequencely, \eqref{NLS-rep} does not enjoy the momentum conservation law, and this leads to a non-positive term 
	\[
	\int \frac{(x-y) \cdot x}{|x-y| |x|} |x|^{-\sigma-1} |v(t,x)|^2 |u(t,y)|^2 dx dy
	\]
	in the interaction Morawetz action rate. Fortunately, we are able to use the classical Morawetz inequality to control this term. We refer the reader to Section $\ref{S3}$ and Section $\ref{S8}$ for more details.
	
	This paper is organized as follows. In Section $\ref{S2}$, we give some preliminaries including Strichartz estimates and the Sobolev norms equivalence. In Section $\ref{S3}$, we prove the local well-posedness in the energy space for \eqref{NLS-rep} in both energy-subcritical and energy-critical cases. In Section $\ref{S4}$, we prove the interaction Morawetz inequality for a general class of NLS with potentials including \eqref{NLS-rep} in the defocusing case. In Section $\ref{S5}$, we derive the virial identity and some localized virial estimates related to \eqref{NLS-rep} in the focusing case. Section $\ref{S6}$ is devoted to the proof of the global well-posedness given in Theorem $\ref{theo-GWP}$. The finite time blow-up given in Theorem $\ref{theo-blo-up}$ will be proved in Section $\ref{S7}$. Finally, we prove the energy scattering for \eqref{NLS-rep} in the defocusing intercritical case in Section $\ref{S8}$. 
	
	\section{Preliminaries}
	\label{S2}
	\subsection{Notations}
	For some non-negative quantities $X, Y$, we use the notation $X \lesssim Y$ to denote the estimate $X \leq CY$ for some constant $C>0$. We also use $X\sim Y$ if $X \lesssim Y \lesssim X$. 
	
	We use $L^q(\R^d)$ to denote the Banach space of measurable functions $f: \R^d \rightarrow \C$ whose norm
	\[
	\|f\|_{L^q}:= \left( \int_{\R^d} |f(x)|^q dx \right)^{\frac{1}{q}}
	\]
	is finite, with a usual modification when $q=\infty$. Let $J \subset \R$ be an interval and $1 \leq p,q <\infty$. We define the mixed norm
	\[
	\|u\|_{L^p(J,L^q)} := \left( \int_J \left(\int_{\R^d} |u(t,x)|^q dx\right)^{\frac{p}{q}} dt \right)^{\frac{1}{p}}
	\]
	with usual modifications when either $p$ or $q$ are infinity. 
	
	The Fourier and inverse Fourier transforms on $\R^d$ are defined respectively by
	\[
	\Fc(f)(\xi) := (2\pi)^{-\frac{d}{2}} \int_{\R^d} e^{-ix \cdot \xi} f(x) dx, \quad \Fc^{-1}(g) (x) := (2\pi)^{-\frac{d}{2}} \int_{\R^d} e^{ix\cdot \xi} g(\xi) d\xi
	\]
	We often use $\hat{f}$ instead of $\Fc(f)$. Let $\gamma \in \R$. We define the fractional differential operators $|\nabla|^\gamma$ and $\scal{\nabla}^\gamma$ to be
	\[
	\widehat{|\nabla|^\gamma f} (\xi):= |\xi|^\gamma \hat{f}(\xi), \quad \widehat{\scal{\nabla}^\gamma f} (\xi):= \scal{\xi}^\gamma \hat{f}(\xi),
	\]
	where $\scal{\xi} = \sqrt{1+|\xi|^2}$ is the Japanese bracket. The homogeneous and inhomogeneous Sobolev norms are defined respectively by
	\[
	\|f\|_{\dot{W}^{\gamma, q}} := \||\nabla|^\gamma f\|_{L^q}, \quad \|f\|_{W^{\gamma,q}} := \|\scal{\nabla}^\gamma f\|_{L^q}.
	\]
	When $q=2$, we use the notations $\dot{H}^\gamma, H^\gamma$ instead of $\dot{W}^{\gamma,2}$ and $W^{\gamma,2}$. 
	
	\subsection{Nonlinearity}
	Let $f(z):= |z|^\alpha z$ with $\alpha>0$. The complex derivatives of $f$ are
	\[
	\partial_z f(z) = \frac{\alpha+2}{2} |z|^\alpha, \quad \partial_{\overline{z}} f(z) = \frac{\alpha}{2} |z|^{\alpha-2} z^2. 
	\]
	We have the chain rule
	\begin{align} \label{chain-rule}
	\nabla f(u) = \partial_z f(u) \nabla u + \partial_{\overline{z}}f(u) \overline{\nabla u}. 
	\end{align}
	\begin{lemma} \label{lem-nonlinearity}
		It holds that
		\begin{align} \label{dif-est-1}
		|f(u+v) - f(u)| \lesssim (|u|^\alpha + |v|^\alpha) |v|, 
		\end{align}
		and
		\begin{align} \label{dif-est-2}
		|\nabla(f(u+v) - f(u))| \lesssim \left\{
		\renewcommand*{\arraystretch}{1.2}
		\begin{array}{ccc}
		|\nabla u| |v|^\alpha + |\nabla v| |u|^\alpha + |\nabla v| |v|^\alpha &\text{if}& 0<\alpha < 1, \\
		|\nabla u| |v| |u|^{\alpha-1} + |\nabla u| |v|^\alpha + |\nabla v| |u|^\alpha + |\nabla v| |v|^\alpha &\text{if}& \alpha \geq 1.
		\end{array}
		\right.
		\end{align}
	\end{lemma}
	
	\begin{proof}
		To see \eqref{dif-est-1}, we write
		\[
		f(u+v) - f(u) = \int_0^1 \partial_z f(u+\theta v) v + \partial_{\overline{z}}f(u+\theta v) \overline{v} d\theta.
		\]
		Since $\partial_z f(z), \partial_{\overline{z}} f(z) = O(|z|^\alpha)$, it follows that
		\[
		|f(u+v) -f(u)| \lesssim (|u+v|^\alpha + |u|^\alpha)|v|.
		\]
		If $\alpha \geq 1$, we simply bound $|u+v|^\alpha \lesssim |u|^\alpha + |v|^\alpha$. If $0<\alpha < 1$, we write $|u+v|^\alpha = |u+v|^\alpha - |u|^\alpha + |u|^\alpha$. Using the H\"older continuity, the difference is bounded (up to a constant) by $|v|^\alpha$. This shows \eqref{dif-est-1}.
		
		To see \eqref{dif-est-2}, we write
		\begin{align*}
		\nabla(f(u+v) - f(u)) &= \partial_z f(u+v) \nabla (u+v) + \partial_{\overline{z}} f(u+v) \overline{\nabla(u+v)} - \partial_zf(u) \nabla u - \partial_{\overline{z}}f(u) \overline{\nabla u} \\
		&=\nabla u (\partial_z f(u+v) - \partial_z f(u)) + \overline{\nabla u} (\partial_{\overline{z}} f(u+v) - \partial_{\overline{z}}f(u)) \\
		&\mathrel{\phantom{=}} + \nabla v \partial_z f(u+v) + \overline{\nabla v} \partial_{\overline{z}} f(u+v).
		\end{align*}
		This implies that
		\begin{align*}
		|\nabla(f(u+v)-f(u))| &\leq |\nabla u| \left(|\partial_z f(u+v)-\partial_z f(u)| +|\partial_{\overline{z}} f(u+v) -\partial_{\overline{z}} f(u)| \right) \\
		&\mathrel{\phantom{\leq}} + |\nabla v| \left(|\partial_z f(u+v)|+ |\partial_{\overline{z}} f(u+v)| \right).
		\end{align*}
		Since $\partial_z f(z), \partial_{\overline{z}} f(z) = O(|z|^\alpha)$, we see that
		\begin{align*}
		\left. 
		\renewcommand*{\arraystretch}{1.2}
		\begin{array}{r}
		\partial_zf(u+v) - \partial_z f(u) \\
		\partial_{\overline{z}} f(u+v) -\partial_{\overline{z}} f(u)
		\end{array}
		\right\} = \left\{ 
		\renewcommand*{\arraystretch}{1.2}
		\begin{array}{ccc}
		O(|v|^\alpha) &\text{if}& 0<\alpha<1,\\
		O\left(|v|(|u+v|^{\alpha-1} + |u|^{\alpha-1})\right) &\text{if}& \alpha\geq 1.
		\end{array}
		\right.
		\end{align*}
		In the case $0<\alpha<1$, we get
		\begin{align*}
		|\nabla(f(u+v)-f(u))| &\lesssim |\nabla u| |v|^\alpha + |\nabla v||u+v|^\alpha \\
		& =  |\nabla u| |v|^\alpha + |\nabla v| |u|^\alpha + |\nabla v| (|u+v|^\alpha - |u|^\alpha) \\
		&\lesssim |\nabla u| |v|^\alpha + |\nabla v| |u|^\alpha + |\nabla v| |v|^\alpha.
		\end{align*}
		In the case $\alpha \geq 1$, we have that
		\begin{align*}
		|\nabla(f(u+v)-f(u))| &\lesssim |\nabla u| |v|(|u+v|^{\alpha-1}+|u|^{\alpha-1}) + |\nabla v| |u+v|^\alpha \\
		&\lesssim |\nabla u| |v| (|u|^{\alpha-1}+|v|^{\alpha-1}) + |\nabla v| (|u|^\alpha +|v|^\alpha) \\
		&= |\nabla u| |v||u|^{\alpha-1} + |\nabla u| |v|^\alpha + |\nabla v| |u|^\alpha + |\nabla v| |v|^\alpha.
		\end{align*}
		The proof is complete.
	\end{proof}
	\subsection{Lorentz spaces}
	Let $f:\R^d \rightarrow \C$ be a measurable function. The distribution function of $f$ is defined by
	\[
	{\bf d}_f(\lambda):= |\{x \in \R^d \ : \ |f(x)| >\lambda \}|, \quad \lambda \in [0,\infty),
	\]
	where $|\cdot|$ is the Lebesgue measure on $\R^d$. The decreasing rearrangement of $f$ is given by
	\[
	f^*(t):= \inf \{ \lambda \geq 0 \ : \ {\bf d}_f(\lambda) \leq t \}, \quad t \in [0,\infty).
	\]
	Let $0<q <\infty$ and $0<r \leq \infty$. The Lorentz space $L^{q,r}(\R^d)$ is space of measurable functions $f: \R^d \rightarrow \C$ whose norm
	\[
	\|f\|_{L^{q,r}}:= \left\{
	\begin{array}{cl}
	\left( \mathlarger\int_0^\infty \left( t^{\frac{1}{q}} f^*(t)\right)^r \frac{dt}{t} \right)^{\frac{1}{r}} &\text{if } r<\infty,\\
	\sup_{t>0} t^{\frac{1}{q}} f^*(t) &\text{if } r=\infty
	\end{array}
	\right.
	\]
	is finite. Using the fact 
	\[
	\int_{\R^d} |f(x)|^q dx = \int_0^\infty (f^*(s))^q ds,
	\]
	we see that $L^{q,q} \equiv L^q$ for $0<q<\infty$ and by convention $L^{\infty, \infty} = L^\infty$. Moreover, for $0<q<\infty$ and $0<r_1<r_2 \leq \infty$, $L^{q,r_1}$ is a subspace of $L^{q,r_2}$. In particular, there exists $C=C(q,r_1,r_2)>0$ such that
	\begin{align} \label{lor-inc-ine}
	\|f\|_{L^{q,r_2}} \leq C \|f\|_{L^{q,r_1}}.
	\end{align}
	It is easy to see that the function $|x|^{-\frac{d}{q}}$ belongs to $L^{q,\infty}$ and $\||x|^{-\frac{d}{q}}\|_{L^{q,\infty}} = |B(0,1)|^{\frac{1}{q}}$ where $|B(0,1)|$ is the volume of the unit ball in $\R^d$, but it does not belong to any Lebesgue space. We have the following H\"older's inequalities in Lorentz spaces.
	\begin{lemma}[H\"older's inequality \cite{ONeil}]
		\begin{itemize}
			\item Let $1< q, q_1,q_2 <\infty$ and $1 \leq r, r_1, r_2 \leq \infty$ be such that 
			\[
			\frac{1}{q_1}+\frac{1}{q_2}=\frac{1}{q}, \quad  \frac{1}{r_1}+\frac{1}{r_2} \geq \frac{1}{r}.
			\]
			Then there exists $C=C(q,q_1,q_2, r, r_1,r_2)>0$ such that
			\begin{align} \label{lor-hol-ine}
			\|fg\|_{L^{q,r}} \leq C \|f\|_{L^{q_1,r_1}} \|g\|_{L^{q_2,r_2}},
			\end{align}
			for any $f \in L^{q_1,r_1}$ and $g\in L^{q_2,r_2}$.
			\item Let $1<q_1,q_2 <\infty$ and $1 \leq r_1, r_2 \leq \infty$ be such that 
			\[
			\frac{1}{q_1}+\frac{1}{q_2}=1, \quad \frac{1}{r_1}+\frac{1}{r_2} \geq 1.
			\]
			Then there exists $C=C(q_1,q_2, r_1,r_2)>0$ such that
			\begin{align} \label{lor-hol-ine-1}
			\|fg\|_{L^1} \leq C \|f\|_{L^{q_1,r_1}} \|g\|_{L^{q_2,r_2}},
			\end{align}
			for any $f \in L^{q_1,r_1}$ and $g\in L^{q_2,r_2}$.
		\end{itemize}
	\end{lemma}
	We also have the following convolution inequalities in Lorentz spaces.
	\begin{lemma} [Convolution inequality \cite{ONeil}]
		\begin{itemize}
			\item Let $1< q, q_1,q_2 <\infty$ and $1 \leq r, r_1, r_2 \leq \infty$ be such that 
			\[
			\frac{1}{q_1}+\frac{1}{q_2}=1+\frac{1}{q}, \quad  \frac{1}{r_1}+\frac{1}{r_2} \geq \frac{1}{r}.
			\]
			Then there exists $C=C(q,q_1,q_2, r, r_1,r_2)>0$ such that
			\begin{align} \label{lor-con-ine}
			\|f \ast g\|_{L^{q,r}} \leq C \|f\|_{L^{q_1,r_1}} \|g\|_{L^{q_2,r_2}},
			\end{align}
			for any $f \in L^{q_1,r_1}$ and $g\in L^{q_2,r_2}$.
			\item Let $1< q_1,q_2 <\infty$ and $1 \leq r_1, r_2 \leq \infty$ be such that 
			\[
			\frac{1}{q_1}+\frac{1}{q_2}=1, \quad  \frac{1}{r_1}+\frac{1}{r_2} \geq 1.
			\]
			Then there exists $C=C(q_1,q_2, r_1,r_2)>0$ such that
			\begin{align} \label{lor-con-ine-1}
			\|f \ast g\|_{L^\infty} \leq C \|f\|_{L^{q_1,r_1}} \|g\|_{L^{q_2,r_2}},
			\end{align}
			for any $f \in L^{q_1,r_1}$ and $g\in L^{q_2,r_2}$.
		\end{itemize}
	\end{lemma}
	As a direct consequence of convolution inequalities in Lorentz spaces and the fact $|x|^{-(d-\gamma)} \in L^{\frac{d}{d-\gamma},\infty}$, we have the following Hardy-Littlewood-Sobolev inequality in Lorentz spaces.
	\begin{corollary}[Hardy-Littlewood-Sobolev inequality \cite{ONeil}]
		Let $1<q<\infty$ and $1\leq r\leq\infty$ and $0<\gamma<\frac{d}{q}$. Then there exists $C=C(q,r,\gamma)>0$ such that
		\begin{align} \label{lor-har-lit-sob-ine}
		\|I_\gamma f\|_{L^{\frac{dq}{d-\gamma q},r}} \leq C \|f\|_{L^{q,r}},
		\end{align}
		for any $f \in L^{q,r}$, where $I_\gamma$ is the Riesz potential
		\[
		I_\gamma f (x) = |\nabla|^{-\gamma} f(x) = C(\gamma)  \int \frac{f(y)}{|x-y|^{d-\gamma}} dy. 
		\]
	\end{corollary}
	\begin{corollary}[Sobolev embedding] \label{coro-sob-emb-lor}
		Let $1<q<\infty$ and $1 \leq r \leq \infty$ and $0<\gamma <\frac{d}{q}$. Then there exists $C= C(q,r,\gamma)>0$ such that
		\begin{align} \label{lor-sob-emb}
		\|f\|_{L^{\frac{dq}{d-\gamma q},r}} \leq C \||\nabla|^\gamma f\|_{L^{q,r}},
		\end{align}
		for any $f \in \dot{W}^\gamma L^{q,r}$. 
	\end{corollary}
	Here $\dot{W}^\gamma L^{q,r}$ is the space of functions satisfying $|\nabla|^\gamma f \in L^{q,r}$. Similarly, we define 
	\begin{align} \label{sob-lor-spa}
	W^\gamma L^{q,r} \equiv \dot{W}^\gamma L^{q,r} \cap L^{q,r}.
	\end{align}
	
	Combining H\"older's inequality, Sobolev embedding in Lorentz spaces and the fact $|x|^{-\gamma} \in L^{\frac{d}{\gamma},\infty}$, we obtain the following Hardy's inequality in Lorentz spaces.
	\begin{corollary} [Hardy's inequality] 
		Let $1<q<\infty$, $1\leq r\leq \infty$ and $0<\gamma<\frac{d}{q}$. Then there exists $C= C(q,r,\gamma)>0$ such that
		\begin{align} \label{lor-har-ine}
		\||x|^{-\gamma} f\|_{L^{q,r}} \leq C \||\nabla|^\gamma f\|_{L^{q,r}},
		\end{align}
		for any $f\in \dot{W}^\gamma L^{q,r}$. 
	\end{corollary}
		
	\subsection{Strichartz estimates}
	\begin{definition} 
		A pair $(p,q)$ is said to be Schr\"odinger admissible, for short $(p,q) \in S$, if
		\[
		p, q\in [2,\infty], \quad (p,q,d) \ne (2,\infty, 2), \quad \frac{2}{p} + \frac{d}{q} =\frac{d}{2}.
		\]
	\end{definition}
	\begin{theorem}[Strichartz estimates \cite{KT, Planchon}] \label{theo-str-lor}
		Let $d\geq 1$. Then for any $(p,q)$ and $(a,b)$ Schr\"odinger admissible pairs, there exists $C>0$ such that
		\begin{align}
		\|e^{it\Delta} f\|_{L^p(\R, L^q)} &\leq C\|f\|_{L^2}, \label{hom-str} \\
		\left\| \int_{\R} e^{-is\Delta} F(s) ds \right\|_{L^2} &\leq C\|F\|_{L^{a'}(\R, L^{b'})}, \label{adj-hom-str} \\
		\left\| \int_{s<t} e^{i(t-s)\Delta} F(s) ds \right\|_{L^p(\R, L^q)} &\leq C\|F\|_{L^{a'}(\R, L^{b'})}. \label{inh-str}
		\end{align}
		Moreover, if $p, q, a, b<\infty$, then there exists $C>0$ such that
		\begin{align}
		\|e^{it\Delta} f\|_{L^{p,2}(\R, L^{q,2})} &\leq C \|f\|_{L^2}, \label{hom-str-lor} \\
		\left\| \int_{\R} e^{-is\Delta} F(s) ds \right\|_{L^2} &\leq C \|F\|_{L^{a',2}(\R, L^{b',2})}, \label{adj-hom-str-lor} \\
		\left\| \int_{s<t} e^{i(t-s)\Delta} F(s) ds \right\|_{L^{p,2}(\R, L^{q,2})} &\leq  C\|F\|_{L^{a',2}(\R, L^{b',2})}. \label{inh-str-lor}
		\end{align}
		Here $(a,a')$ and $(b,b')$ are H\"older conjugate pairs.
	\end{theorem}
	
	We also have the following global-in-time Strichartz estimates for $e^{-itH_c}$ which was proved recently by Mizutani \cite{Mizutani}. The proof employs several techniques from scattering theory such as the long time parametrix construction of Isozaki-Kitada type, propagation estimates and local decay estimates.
	
	\begin{theorem}[Global-in-time Strichartz estimates \cite{Mizutani}] \label{theo-glo-str}
		Let $d\geq 3$, $c>0$ and $0<\sigma<2$. Then for any Schr\"odinger admissible pairs $(p,q)$ and $(a,b)$, there exists $C>0$ such that
		\[
		\|e^{-itH_c} u_0\|_{L^p(\R, L^q)} \leq C \|u_0\|_{L^2},
		\]
		and
		\[
		\left\| \int_0^t e^{-i(t-s)H_c} F(s) ds\right\|_{L^p(\R, L^q)} \leq C \|F\|_{L^{a'}(\R, L^{b'})},
		\]
		for all $u_0 \in L^2$ and $F \in L^{a'}(\R, L^{b'}) \cap L^1_{\emph{loc}}(\R, L^2)$. 
	\end{theorem}
	
	\begin{remark}
		In \cite{MZZ}, Miao-Zhang-Zheng gave an example which shows the failure of global-in-time Strichartz estimates for $e^{-itH_c}$ with $H_c=-\Delta +c |x|^{-1}$ and $c<0$. More precisely, the function $u(t,x)= e^{ia^2 t} e^{a|x|}$ with $a= \frac{c}{d-1}$ solves the linear equation $i\partial_t u - H_c u =0$ with initial data $u_0(x)=e^{a|x|}$. It is easy to see that for $c<0$,
		\[
		\|u_0\|_{L^2} <\infty, \quad \|u\|_{L^p(\R, L^q)} = \infty.
		\]
	\end{remark}
	
	\subsection{Equivalence of Sobolev norms}
	In this paragraph, we show the equivalence between Sobolev norms defined by $H_c$ and the ones defined by the usual Laplacian operator $-\Delta$. To do so, we first define the homogeneous and inhomogeneous Sobolev spaces associated to $H_c$ as the closure of $C^\infty_0(\R^d)$ under the norms
	\[
	\|f\|_{\dot{W}^{\gamma,q}_c} := \left\|\left(\sqrt{H_c}\right)^\gamma f\right\|_{L^q}, \quad \|f\|_{W^{\gamma,q}_c}:= \left\|\scal{\sqrt{H_c}}^\gamma f\right\|_{L^q}
	\] 
	respectively. We abbreviate $\dot{H}^\gamma_c:=\dot{W}^{\gamma,2}_c$ and $H^\gamma_c:=W^{\gamma,2}_c$. Note that by definition, we have that
	\begin{align} \label{dot-H1-nor}
	\|f\|^2_{\dot{H}^1_c} = \int |\nabla f|^2 + c|x|^{-\sigma} |f|^2 dx.
	\end{align}
	We next recall some tools which are useful to show the Sobolev norms equivalence. The first tool is the generalized Hardy's inequality (see e.g. \cite[Lemma 2.6]{ZZ}).
	\begin{lemma} [Generalized Hardy's inequality \cite{ZZ}] \label{lem-gen-har-ine}
		Let $1<q<\infty$ and $0 < \gamma <\frac{d}{q}$. Then there exists $C=C(q,\gamma)>0$ such that 
		\begin{align} \label{har-ine}
		\||x|^{-\gamma} f\|_{L^q} \leq C \||\nabla|^\gamma f\|_{L^q},
		\end{align}
		for any $f \in \dot{W}^{\gamma,q}$.
	\end{lemma}
	Note that this inequality can be seen as a direct consequence of Hardy's inequality in Lorentz spaces \eqref{lor-har-ine} with $q=r$. Another useful tool is the heat kernel Gaussian upper bound. Let $K(t,x,y)$ be the kernel of the heat operator $e^{-tH_c}, t > 0$, i.e.
	\[
	e^{-tH_c} f(x) = \int K(t,x,y) f(y) dy, \quad f \in L^2.
	\]
	Since the potential $c|x|^{-\sigma}$ is non-negative, the semigroup $(e^{-tH_c})_{t\geq 0}$ is dominated by the free semigroup $(e^{t\Delta})_{t\geq 0}$ (see e.g. \cite[(7.6)]{Ouhabaz}). The following result follows immediately.
	\begin{lemma} [Gaussian upper bound] \label{lem-hea-ker-est}
		Let $d\geq 3$, $c>0$ and $0<\sigma<2$. Then the heat kernel of $e^{-tH_c}$ satisfies
		\begin{align} \label{hea-ker-upp-bou}
		0 \leq K(t,x,y) \leq \frac{1}{(4\pi t)^{\frac{d}{2}}} e^{-\frac{|x-y|^2}{4t}}, \quad \forall t>0, x, y \in \R^d.
		\end{align}
	\end{lemma}
	We are now in position to show the main result of this paragraph.
	\begin{proposition} [Equivalence of Sobolev norms] \label{prop-equ-sob}
		Let $d\geq 3$, $c>0$ and $0<\sigma<2$. Then for any $0 \leq \gamma \leq 2$ and $1<q<\frac{2d}{\gamma \sigma}$, it holds that
		\begin{align} \label{equ-sob-spa}
		\left\|\scal{\sqrt{H_c}}^\gamma f\right\|_{L^q} \simeq \|\scal{\nabla}^\gamma f\|_{L^q}.
		\end{align}
	\end{proposition}
	\begin{proof}
		The proof is based on the weak-type estimate of the imaginary powers $(1+H_c)^{iy}$ and the Stein-Weiss interpolation theorem (see e.g. \cite{DFVV} or \cite{MZZ}). For reader's convenience, we give some details.
		
		Let us consider the case $\gamma=2$. Thanks to the generalized Hardy's inequality \eqref{har-ine}, we have that for $1<q<\frac{d}{\sigma}$,
		\begin{align*}
		\left\|\scal{\sqrt{H_c}}^2 f \right\|_{L^q} = \|(1+H_c) f\|_{L^q}& \leq \|(1-\Delta) f\|_{L^q} + c \||x|^{-\sigma} f\|_{L^q} \\
		&\lesssim \|(1-\Delta) f \|_{L^q} + \||\nabla|^\sigma f\|_{L^q} \\
		&\lesssim \|(1-\Delta)f\|_{L^q} = \|\scal{\nabla}^2 f \|_{L^q}.
		\end{align*}
		For the inverse inequality, we take advantage of the Hardy's inequality related to $H_c$, namely
		\begin{align} \label{har-ine-Hc}
		\||x|^{-\sigma} f\|_{L^q} \lesssim \left\|\left(\sqrt{H_c} \right)^\sigma f \right\|_{L^q}, \quad 1<q<\frac{d}{\sigma}.
		\end{align}
		Thus
		\begin{align*}
		\|\scal{\nabla}^2 f\|_{L^q} = \|(1-\Delta)f\|_{L^q} &\leq \|(1+H_c)f\|_{L^q} + c \||x|^{-\sigma} f\|_{L^q} \\
		&\lesssim \|(1+H_c)f\|_{L^q} + \left\|\left(\sqrt{H_c} \right)^\sigma f \right\|_{L^q} \\
		&\lesssim \|(1+H_c) f\|_{L^q} = \left\|\scal{\sqrt{H_c}}^2 f \right\|_{L^q}.
		\end{align*}
		Let us now prove \eqref{har-ine-Hc}. By setting $\phi = \left(\sqrt{H_c}\right)^\sigma  f$, it suffices to show
		\begin{align} \label{har-ine-Hc-equ}
		\left\| |x|^{-\sigma} \left(\sqrt{H_c}\right)^{-\sigma} \phi\right\|_{L^q} \leq C \|\phi\|_{L^q}.
		\end{align}
		We have from the spectral theory that
		\[
		\left( \sqrt{H_c}\right)^{-\sigma} = \frac{1}{\Gamma\left(\frac{\sigma}{2}\right)} \int_0^\infty t^{\frac{\sigma}{2}-1} e^{-tH_c} dt.
		\]
		By the Gaussian upper bound \eqref{hea-ker-upp-bou}, 
		\[
		\left[ \left( \sqrt{H_c}\right)^{-\sigma}\right] (x,y) \leq \frac{1}{(4\pi)^{\frac{d}{2}}\Gamma\left(\frac{\sigma}{2}\right)} \int_0^\infty t^{\frac{\sigma}{2} - \frac{d}{2} -1} e^{-\frac{|x-y|^2}{4t}} dt.
		\]
		After a change of variables, we get
		\[
		\left[ \left( \sqrt{H_c}\right)^{-\sigma}\right] (x,y) \leq \frac{C}{\Gamma\left(\frac{\sigma}{2}\right)} |x-y|^{\sigma-d} \int_0^\infty t^{\frac{d}{2}-\frac{\sigma}{2}-1} e^{-t} dt  = C\frac{\Gamma\left(\frac{d}{2}-\frac{\sigma}{2}\right)}{\Gamma\left(\frac{\sigma}{2}\right)} \frac{1}{|x-y|^{d-\sigma}},
		\]
		for some constant $C>0$. Therefore
		\begin{align*}
		|x|^{-\sigma} \left(\sqrt{H_c}\right)^{-\sigma} \phi (x) &\leq C(d,\sigma) \int \frac{\phi(y)}{|x-y|^{d-\sigma} |x|^\sigma} dy \\
		&\lesssim \int_{|x-y| \leq 100 |x|} \frac{\phi(y)}{|x-y|^{d-\sigma} |x|^\sigma} dy + \int_{|x-y| \geq 100 |x|} \frac{\phi(y)}{|x-y|^{d-\sigma} |x|^\sigma} dy\\
		&=: A_1 \phi(x) + A_2 \phi(x).
		\end{align*}
		The proof of \eqref{har-ine-Hc-equ} is done if we show both $A_1$ and $A_2$ are strong $(q,q)$ type. This follows by the same lines as in \cite[Lemma 2.6]{ZZ}.
		
		We now consider the analytic family of operators
		\[
		T_z = (1+H_c)^z (1-\Delta)^{-z}, \quad z \in \C.
		\]
		By writing $z=x+iy$, we can decompose
		\[
		T_z = (1+H_c)^{iy} (1+H_c)^x (1-\Delta)^{-x} (1-\Delta)^{-iy}.
		\]
		Since the kernels of $e^{-t(1+H_c)}$ and $e^{-t(1-\Delta)}$ obey the Gaussian upper bound as in \eqref{hea-ker-upp-bou}, we have by Sikora-Wright \cite{SW} that
		\[
		\|(1+H_c)^{iy}\|_{L^1 \rightarrow L^{1,\infty}}, \|(1-\Delta)^{iy}\|_{L^1 \rightarrow L^{1,\infty}} \leq C(1+|y|)^{\frac{d}{2}}, \quad \forall y \in \R.
		\]
		Moreover, the operators $(1+H_c)^{iy}$ and $(1-\Delta)^{iy}$ are obviously bounded on $L^2$. By interpolation, we obtain that
		\[
		\|(1+H_c)^{iy}\|_{L^q \rightarrow L^q}, \|(1-\Delta)^{iy}\|_{L^q \rightarrow L^q} \leq C(1+|y|)^{\frac{d}{2}}, \quad \forall y \in \R, 1<q<\infty.
		\]
		This implies that
		\begin{align} \label{com-int-1}
		\|T_{iy}\|_{L^q \rightarrow L^q} \leq C (1+|y|)^d, \quad \forall y \in \R, 1<q<\infty.
		\end{align}
		On the other hand, it follows from the equivalence $\|(1+H_c)f\|_{L^q} \simeq \|(1-\Delta) f\|_{L^q}$ that 
		\begin{align} \label{com-int-2}
		\|T_{1+iy}\|_{L^q \rightarrow L^q} \leq C(1+|y|)^d \|(1+H_c) (1-\Delta)^{-1}\|_{L^q \rightarrow L^q} \leq C(1+|y|)^d, \quad \forall y \in \R, 1<q<\frac{d}{\sigma}.
		\end{align}
		Here the constant $C$ may vary from lines to lines. Applying the Stein-Weiss interpolation theorem, we obtain for $0\leq x \leq 1$ and $1<q<\frac{d}{x \sigma}$,
		\[
		\|T_x\|_{L^q \rightarrow L^q} \leq C.
		\]
		Equivalently, we prove that for $0\leq \gamma \leq 2$ and $1<q<\frac{2d}{\gamma \sigma}$,
		\[
		\left\|\scal{\sqrt{H_c}}^\gamma f\right\|_{L^q} \leq C \|\scal{\nabla}^\gamma f\|_{L^q}.
		\]
		The inverse inequality is treated similarly by considering $Q_z=(1-\Delta)^z (1+H_c)^{-z}$. The proof is complete.
	\end{proof}
	
	\subsection{Variational Analysis}
	In this subsection, we recall the sharp Gagliardo-Nirenberg inequality and the sharp Sobolev embedding which are useful for our purpose. 
	\begin{lemma} [Sharp Gagliardo-Nirenberg inequality \cite{Weinstein}] \label{lem-sha-gag-nir-ine}
		Let $d\geq 1$ and $0<\alpha <\frac{4}{d-2}$ if $d \geq 3$ ($0<\alpha<\infty$ if $d=1,2$). Then the Gagliardo-Nirenberg inequality 
		\[
		\|f\|^{\alpha+2}_{L^{\alpha+2}} \leq C_{\emph{GN}} \|\nabla f\|^{\frac{d\alpha}{2}}_{L^2} \|f\|^{\frac{4-(d-2)\alpha}{2}}_{L^2}, \quad f \in H^1
		\]
		holds true, and the sharp constant $C_{\emph{GN}}$ is attained by a function $Q$, i.e.
		\[
		C_{\emph{GN}} = \|Q\|^{\alpha+2}_{L^{\alpha+2}} \div \left[ \|\nabla Q\|^{\frac{d\alpha}{2}}_{L^2} \|Q\|_{L^2}^{\frac{4-(d-2)\alpha}{2}}\right],
		\]
		where $Q$ is the unique (up to symmetries) positive, radially symmetric, decreasing solution to \eqref{ell-equ-Q}.
	\end{lemma}
	
	We collect some properties of $Q$ as follows. It is well-known that $Q$ satisfies the following Pohozaev identities:
	\begin{align} \label{ide-Q}
	\|Q\|^2_{L^2} = \frac{4-(d-2)\alpha}{d\alpha} \|\nabla Q\|^2_{L^2} = \frac{4-(d-2)\alpha}{2(\alpha+2)} \|Q\|^{\alpha+2}_{L^{\alpha+2}}.
	\end{align}
	In the case $\alpha=\frac{4}{d}$, we see that
	\begin{align} \label{sha-con-gag-nir-mas}
	C_{\text{GN}} = \frac{d+2}{d} \left(\frac{1}{\|Q\|_{L^2}} \right)^{\frac{4}{d}}.
	\end{align}
	In the case $\frac{4}{d}<\alpha<\frac{4}{d-2}$ if $d\geq 3$ ($\frac{4}{d}<\alpha<\infty$ if $d=1,2$), we have that
	\begin{align} \label{sha-con-gag-nir-int}
	C_{\text{GN}} = \frac{2(\alpha+2)}{d\alpha} \left(\|\nabla Q\|_{L^2} \|Q\|_{L^2}^{\betc} \right)^{2-\frac{d\alpha}{2}},
	\end{align}
	where $\betc$ is as in \eqref{def-bet-c}. 
	
	\begin{lemma}[Sharp Sobolev embedding \cite{Aubin,Talenti}] \label{lem-sha-sob-emb}
		Let $d\geq 3$. Then the Sobolev embedding
		\[
		\|f\|^{\frac{2d}{d-2}}_{L^{\frac{2d}{d-2}}} \leq C_{\emph{SE}} \|\nabla f\|^{\frac{2d}{d-2}}_{L^2}, \quad f \in \dot{H}^1
		\]
		holds true, and the sharp constant $C_{\emph{SE}}$ is attained by a function $W$, i.e.
		\[
		C_{\emph{SE}} = \|W\|^{\frac{2d}{d-2}}_{L^{\frac{2d}{d-2}}} \div \|\nabla W\|^{\frac{2d}{d-2}}_{L^2},
		\]
		where $W$ is given in \eqref{def-W}.
	\end{lemma}
	
	It is well-known that $W$ satisfies the following identity
	\begin{align} \label{ide-W}
	\|\nabla W\|^2_{L^2} = \|W\|^{\frac{2d}{d-2}}_{L^{\frac{2d}{d-2}}}.
	\end{align}
	It follows that
	\begin{align} \label{sha-con-sob-emb}
	C_{\text{SE}} = \|\nabla W\|^{-\frac{4}{d-2}}_{L^2}.
	\end{align}
		
	\section{Local well-posedness}
	\label{S3}
	In this section, we prove the local wel-posedness in the energy space $H^1$ for \eqref{NLS-rep}. We consider separately the energy-subcritical and energy-critical cases.
	
	\subsection{Local well-posedness in the energy-subcritical case}
	As mentioned in the introduction, there are two methods to prove the local well-posedness for \eqref{NLS-rep}. One is the energy method which does not use Strichartz estimates and another one is the Kato method which uses Strichartz estimates. Let us start with the local well-posedness via the energy method.
	\subsubsection{LWP via the energy method}
	Consider the Cauchy problem
	\begin{align} \label{NLS}
	i\partial_t u + \Delta u  = g(u), \quad u(0) = u_0 \in H^1.
	\end{align}
	We first recall the following result due to Cazenave (see \cite[Theorem 3.3.5, Theorem 3.3.9 and Proposition 4.2.3]{Cazenave}).
	\begin{theorem} \label{theo-lwp-ene}
		Let $g= g_1 + \cdots + g_N$ be such that the following assumptions hold for each $j =1, \cdots, N$:
		\begin{itemize}
			\item[(A1)] $g_j \in C(H^1, H^{-1})$ and there exists $G_j\in C^1(H^1,\R)$ such that $g_j = G_j'$;
			\item[(A2)] there exist $r_j, \rho_j \in \left[2,\frac{2d}{d-2}\right)$ if $d \geq 2$ ($r_j, \rho_j \in [2,\infty]$ if $d=1$) such that $g \in C(H^1, L^{\rho'})$, and for any $M>0$, there exists $C(M)>0$ such that 
			\[
			\|g_j(u) - g_j(v)\|_{L^{\rho_j'}} \leq C(M) \|u-v\|_{L^{r_j}},
			\]
			for any $u,v \in H^1$ such that $\|u\|_{H^1} + \|v\|_{H^1} \leq M$;
			\item[(A3)]  for any $u \in H^1$, $\ime{g_j(u) \overline{u}} =0$ a.e. in $\R^d$.
		\end{itemize}
		Then for any $u_0 \in H^1$, there exist $T_*, T^* \in (0,\infty]$ and a unique solution 
		\[
		u \in C((-T_*, T^*), H^1) \cap C^1((-T_*, T^*), H^{-1})
		\]
		of \eqref{NLS}. The maximal times satisfy the blow-up alternative: if $T^*<\infty$ (resp. $T_*<\infty$), then $\lim_{t\uparrow T^*} \|u(t)\|_{H^1} = \infty$ (resp. $\lim_{t\downarrow -T_*} \|u(t)\|_{H^1} = \infty$). Moreover, there is convervation of mass and energy, i.e.
		\[
		M(u(t))= \int |u(t,x)|^2 dx = M(u_0), \quad E(u(t)) =\frac{1}{2} \int |\nabla u(t,x)|^2 dx + G_1(u) +\cdots + G_N(u)= E(u_0),
		\]
		for all $t\in (-T_*, T^*)$.
	\end{theorem}
	A direct consequence of Theorem $\ref{theo-lwp-ene}$ is the local well-posedness in $H^1$ for \eqref{NLS-rep} in the energy-subcritical case.
	\begin{proposition} \label{prop-lwp-sub}
		Let $d\geq 1, c>0, 0<\sigma<\min\{2,d\}$ and $0<\alpha<\frac{4}{d-2}$ if $d\geq 3$ ($0<\alpha<\infty$ if $d=1,2$). Then for any $u_0 \in H^1$, there exist $T_*, T^* \in (0,\infty]$ and a unique solution
		\[
		u \in C((-T_*,T^*), H^1) \cap C^1((-T_*,T^*), H^{-1})
		\]
		of \eqref{NLS-rep}. The maximal times of existence satisfy that $T^*<\infty$ (resp. $T_*<\infty$), then $\lim_{t\uparrow T^*} \|u(t)\|_{H^1} = \infty$ (resp. $\lim_{t\downarrow -T_*} \|u(t)\|_{H^1} = \infty$). Moreover, there is convervation of mass and energy, i.e. \eqref{mas-ene} holds 
		for all $t\in (-T_*, T^*)$.
	\end{proposition}
	\begin{proof}
		Since $0<\sigma<\min\{2,d\}$, the potential $c|x|^{-\sigma}$ belongs to $L^r(\R^d) + L^\infty(\R^d)$ for some $r>\max \left\{1,\frac{d}{2}\right\}$. The result follows from Theorem $\ref{theo-lwp-ene}$ using \cite[Example 3.2.11]{Cazenave}.
	\end{proof}
	
	\subsubsection{LWP via Strichartz estimates in Lorentz spaces}
	The local well-posedness given in Proposition $\ref{prop-lwp-sub}$ ensures the existence of local solutions to \eqref{NLS-rep} in the energy-subcritical case. However, we do not know whether or not the local solutions satisfy \eqref{glo-bou-est}. We will show this estimate by using Strichartz estimates in Lorentz spaces.
	\begin{proposition} \label{prop-lwp-sub-lor}
		Let 
		\begin{align} \label{res-sig-lor}
		\left\{
		\renewcommand*{\arraystretch}{1.2}
		\begin{array}{l l}
		0<\sigma<\frac{3}{2} &\text{if } d=3, \\
		0<\sigma<2 &\text{if } d\geq 4,
		\end{array}
		\right.
		\quad \text{and} \quad 0<\sigma <\frac{4}{d-2}. 
		\end{align}
		Then for any $u_0 \in H^1$, there exist $T_*, T^* \in (0,\infty]$ and a unique solution
		\[
		u \in C((-T_*,T^*), H^1) \cap L^m((-T_*,T^*), W^1L^{n,2}) \cap L^\kappa((-T_*,T^*), W^{1,\mu}),
		\]
		for some $(m,n), (\kappa, \mu) \in S$. Moreover, the following properties hold:
		\begin{itemize}
			\item If $T^*<\infty$ (resp. $T_*<\infty$), then $\lim_{t\uparrow T^*} \|u(t)\|_{H^1} = \infty$ (resp. $\lim_{t\downarrow -T_*} \|u(t)\|_{H^1} = \infty$);
			\item $u \in L^p_{\emph{loc}}((-T_*,T^*), W^{1,q})$ for any $(p,q) \in S$;
			\item There is conservation of mass and energy, i.e. \eqref{mas-ene} holds for all $t\in (-T_*,T^*)$.
		\end{itemize}
	\end{proposition}
	
	\begin{proof}
		We first show that under the assumption of $\sigma$ in \eqref{res-sig-lor}, there exist Schr\"odinger admissible pairs $(m,n)$ and $(a,b)$ such that for any finite time interval $J$,
		\begin{align} \label{pot-est}
		\|\scal{\nabla}(|x|^{-\sigma} u)\|_{L^{a'}(J, L^{b',2})} \leq C|J|^{\frac{2-\sigma}{2}} \|\scal{\nabla} u\|_{L^m(J,L^{n,2})}.
		\end{align}
		In fact, we first choose $(m,n) \in S$ with 
		\begin{align} \label{cho-mn}
		\left\{
		\renewcommand*{\arraystretch}{1.2}
		\begin{array}{ll}
		n \in \left[2, \frac{2d}{d-2\sigma}\right] &\text{if } 0<\sigma \leq 1, \\
		n \in \left[\frac{2d}{d+2 - 2\sigma}, \frac{2d}{d-2} \right] &\text{if } 1<\sigma<2.
		\end{array}
		\right.
		\end{align}
		We next choose $(a,b) \in S$ be such that
		\begin{align} \label{mn-ab-1}
		\frac{1}{b'} = \frac{\sigma}{d} + \frac{1}{n}.
		\end{align}
		Note that \eqref{cho-mn} implies that
		\begin{align} \label{cho-ab}
		\left\{
		\renewcommand*{\arraystretch}{1.2}
		\begin{array}{ll}
		b \in \left[2, \frac{2d}{d-2\sigma}\right] &\text{if } 0<\sigma \leq 1, \\
		b \in \left[\frac{2d}{d+2 - 2\sigma}, \frac{2d}{d-2} \right] &\text{if } 1<\sigma<2.
		\end{array}
		\right.
		\end{align}
		Since $(m,n), (a,b) \in S$, it follows from \eqref{mn-ab-1} that
		\begin{align} \label{mn-ab-2}
		\frac{1}{a'}-\frac{1}{m} = \frac{2-\sigma}{2}.
		\end{align}
		We now bound the lelf hand side of \eqref{pot-est} by
		\[
		\sum_{k=0}^1 \||x|^{-\sigma} |\nabla|^k u\|_{L^{a'}(J, L^{b',2})} + \|\nabla(|x|^{-\sigma}) u\|_{L^{a'}(J, L^{b',2})}.
		\]
		By H\"older's inequality \eqref{lor-hol-ine}, \eqref{mn-ab-1} and \eqref{mn-ab-2},
		\begin{align*}
		\||x|^{-\sigma} |\nabla|^k u\|_{L^{a'}(J, L^{b',2})} &\leq \||x|^{-\sigma}\|_{L^{\frac{d}{\sigma},\infty}} \||\nabla|^k u\|_{L^{a'}(J,L^{n,2})} \\
		&\lesssim |J|^{\frac{1}{a'}-\frac{1}{m}} \||\nabla|^k u\|_{L^m(J,L^{n,2})} \\
		&\lesssim |J|^{\frac{2-\sigma}{2}} \|\scal{\nabla} u\|_{L^m(J,L^{n,2})}.
		\end{align*}
		Next, by Sobolev embedding in Lorentz spaces \eqref{lor-sob-emb}, 
		\begin{align*}
		\|\nabla(|x|^{-\sigma}) u\|_{L^{a'}(J, L^{b',2})} &\leq \|\nabla(|x|^{-\sigma})\|_{L^{\frac{d}{\sigma+1},\infty}} \|u\|_{L^{a'}(J, L^{\overline{n},2})} \\
		&\lesssim |J|^{\frac{1}{a'}-\frac{1}{m}} \|u\|_{L^m(J,L^{\overline{n},2})} \\
		&\lesssim |J|^{\frac{2-\sigma}{2}} \||\nabla| u\|_{L^m(J,L^{n,2})} \\
		&\lesssim |J|^{\frac{2-\sigma}{2}} \|\scal{\nabla} u\|_{L^m(J,L^{n,2})}
		\end{align*}
		provided that
		\[
		\frac{1}{b'}= \frac{\sigma+1}{d} +\frac{1}{\overline{n}},\quad \frac{1}{\overline{n}} = \frac{1}{n}-\frac{1}{d}.
		\]
		The last condition requires $n<d$ which is satisfied under the assumption of $\sigma$ in \eqref{res-sig-lor}.
		
		Set
		\[
		\kappa=\frac{4(\alpha+2)}{\alpha(d-2)}, \quad \mu = \frac{d(\alpha+2)}{d+\alpha},
		\]
		and choose $(\zeta, \eta)$ so that
		\begin{align} \label{cho-zet-eta}
		\frac{1}{\kappa'} = \frac{1}{\kappa} + \frac{\alpha}{\zeta}, \quad \frac{1}{\mu'} = \frac{1}{\mu} + \frac{\alpha}{\eta}.
		\end{align}
		It is easy to check that $(\kappa, \mu) \in S$ and
		\begin{align} \label{pro-zet-eta}
		\frac{\alpha}{\zeta} -\frac{\alpha}{\kappa} = 1-\frac{\alpha(d-2)}{4}:\theta>0, \quad \frac{1}{\eta} =\frac{1}{\mu}-\frac{1}{d}.
		\end{align}
		The last condition allows us to use the Sobolev embedding $W^{1,\mu} \subset L^{\eta}$. 
		
		We now consider 
		\begin{multline*}
		X = \left\{
		u \in C(J, H^1) \cap L^m(J, W^1 L^{n,2}) \cap L^\kappa(J, W^{1,\mu}) \ : \right. \\
		\left. \|u\|_X:= \|u\|_{L^\infty(J, H^1)} + \|u\|_{L^m(J, W^1 L^{n,2})}  + \|u\|_{L^\kappa(J,W^{1,\mu})} \leq M
		\right\}
		\end{multline*}
		equipped with the distance
		\[
		d(u,v) = \|u-v\|_{L^\infty(J,L^2)} + \|u-v\|_{L^m(J,L^{n,2})} + \|u-v\|_{L^\kappa(J,L^\mu)},
		\]
		where $J= [0,T]$ with $T, M$ to be chosen later. Here we refer the reader to \eqref{sob-lor-spa} for the definition of $W^1L^{n,2}$.
		
		We will show that the functional
		\[
		\Phi(u(t)) = e^{it\Delta} u_0 - i\int_0^t e^{i(t-s)\Delta} \left( c|x|^{-\sigma} u(s) \pm |u(s)|^\alpha u(s) \right) ds
		\]
		is a contraction on $(X,d)$. Thanks to Strichartz estimates in Lorentz spaces given in Theorem $\ref{theo-str-lor}$ and the fact $L^{q,q} =L^q$, $L^{q,r_1} \subset L^{q,r_2}$ with $r_1 \leq r_2$, we see that
		\begin{align*}
		\|\Phi(u)\|_{L^\infty(J, H^1)} &\leq \|e^{it\Delta} \scal{\nabla} u_0\|_{L^\infty(J,L^2)}  + \left\|\int_0^t e^{i(t-s)\Delta} \scal{\nabla} (|x|^{-\sigma} u(s)) ds \right\|_{L^{\infty,2}(J, L^{2,2})} \\
		&\mathrel{\phantom{\leq \|e^{it\Delta} \scal{\nabla} u_0\|_{L^\infty(J,L^2)}  }}+ \left\| \int_0^t e^{i(t-s)\Delta} \scal{\nabla} \left( |u(s)|^\alpha u(s) \right) ds\right\|_{L^\infty(J,L^2)} \\
		&\lesssim \|u_0\|_{H^1} + \|\scal{\nabla} (|x|^{-\sigma} u)\|_{L^{a'}(J, L^{b',2})} + \|\scal{\nabla} (|u|^\alpha u) \|_{L^{\kappa'}(J, L^{\mu'})}.
		\end{align*}
		By \eqref{pot-est}, \eqref{cho-zet-eta} and \eqref{pro-zet-eta}, the fractional chain rule implies that
		\begin{align*}
		\|\Phi(u)\|_{L^\infty(J, H^1)} &\lesssim \|u_0\|_{H^1} + |J|^{\frac{2-\sigma}{2}} \|\scal{\nabla}u\|_{L^m(J,L^{n,2})} + \|u\|^\alpha_{L^\zeta(J,L^\eta)} \|\scal{\nabla} u\|_{L^\kappa(J,L^\mu)} \\
		&\lesssim \|u_0\|_{H^1} + |J|^{\frac{2-\sigma}{2}} \|u\|_{L^m(J,W^1L^{n,2})} + |J|^\theta \|u\|^\alpha_{L^\kappa(J,L^\eta)} \|u\|_{L^\kappa(J,W^{1,\mu})} \\
		&\lesssim \|u_0\|_{H^1} + |J|^{\frac{2-\sigma}{2}} \|u\|_{L^m(J,W^1L^{n,2})} + |J|^\theta \|u\|^{\alpha+1}_{L^\kappa(J,W^{1,\mu})}.
		\end{align*}
		Similarly, 
		\begin{align*}
		\|\Phi(u)\|_{L^m(J, W^1L^{n,2})} & \lesssim \|u_0\|_{H^1} + \|\scal{\nabla} (|x|^{-\sigma} u)\|_{L^{a'}(J, L^{b',2})} + \|\scal{\nabla} (|u|^\alpha u) \|_{L^{\kappa'}(J, L^{\mu'})} \\
		&\lesssim \|u_0\|_{H^1} + |J|^{\frac{2-\sigma}{2}} \|u\|_{L^m(J,W^1L^{n,2})} + |J|^\theta \|u\|^{\alpha+1}_{L^\kappa(J,W^{1,\mu})},
		\end{align*}
		and
		\[
		\|\Phi(u)\|_{L^\kappa(J, W^{1,\mu})}  \lesssim \|u_0\|_{H^1} + |J|^{\frac{2-\sigma}{2}} \|u\|_{L^m(J,W^1L^{n,2})} + |J|^\theta \|u\|^{\alpha+1}_{L^\kappa(J,W^{1,\mu})}.
		\]
		On the other hand,
		\begin{align*}
		\|\Phi(u) - \Phi(v)\|_{L^\infty(J, L^2)} &\lesssim \left\|\int_0^t e^{i(t-s)\Delta} |x|^{-\sigma} (u(s) - v(s))ds \right\|_{L^{\infty,2}(J, L^{2,2})} \\
		&\mathrel{\phantom{\lesssim}} + \left\|\int_0^t e^{i(t-s)\Delta} \left( |u(s)|^\alpha u(s) - |v(s)|^\alpha v(s) \right) ds \right\|_{L^\infty(J, L^2)}  \\
		&\lesssim \||x|^{-\sigma}(u-v)\|_{L^{a'}(J,L^{b',2})} + \| |u|^\alpha u - |v|^\alpha v \|_{L^{\kappa'}(J,L^{\mu'})} \\
		&\lesssim |J|^{\frac{2-\sigma}{2}} \|u-v\|_{L^m(J,L^{n,2})} + \left(\|u\|^\alpha_{L^\zeta(J,L^\eta)} + \|v\|^\alpha_{L^\zeta(J,L^\eta)} \right) \|u-v\|_{L^\kappa(J, L^\mu)} \\
		&\lesssim |J|^{\frac{2-\sigma}{2}} \|u-v\|_{L^m(J,L^{n,2})} + |J|^\theta \left(\|u\|^\alpha_{L^\kappa(J,W^{1,\mu})} + \|v\|^\alpha_{L^\kappa(J,W^{1,\mu})} \right) \|u-v\|_{L^\kappa(J, L^\mu)}.
		\end{align*}
		Similar estimates hold for $\|\Phi(u) - \Phi(v)\|_{L^m(J, L^{n,2})}$ and $\|\Phi(u) - \Phi(v)\|_{L^\kappa(J, L^\mu)}$. This implies that for $u,v \in X$, there exists $C>0$ independent of $T$ and $u_0 \in H^1$ such that
		\begin{align*}
		\|\Phi(u)\|_X &\leq C \|u_0\|_{H^1} + C T^{\frac{2-\sigma}{2}} M + C T^\theta M^{\alpha+1}, \\
		d(\Phi(u), \Phi(v)) &\leq \left(C T^{\frac{2-\sigma}{2}} + C T^\theta M^\alpha \right) d(u,v). 
		\end{align*}
		Taking $M=2 C\|u_0\|_{H^1}$ and choosing $T>0$ small enough so that 
		\[
		CT^{\frac{2-\sigma}{2}} + C T^\theta M^\alpha \leq \frac{1}{2},
		\]
		we see that $\Phi$ is a contraction on $(X,d)$. This shows the existence of local solutions for \eqref{NLS-rep} in the energy-subcritical case. The blow-up alternative follows from the fact that the time of existence depends only on $H^1$-norm of initial data. Now let $(p,q)\in S$ and $I$ be any compact interval of $(-T_*,T^*)$. We have that
		\begin{align*}
		\|u\|_{L^p(I, W^{1,q})} &\leq \|e^{it\Delta} \scal{\nabla} u_0\|_{L^p(I,L^q)} + \left\| \int_0^t e^{i(t-s)\Delta} \scal{\nabla} (|x|^{-\sigma} u(s)) ds \right\|_{L^{p,2}(I,L^{q,2})} \\
		&\mathrel{\phantom{\leq \|e^{it\Delta} \scal{\nabla} u_0\|_{L^p(I,L^q)} }}+ \left\| \int_0^t e^{i(t-s)\Delta} \scal{\nabla} \left( |u(s)|^\alpha u(s) \right) ds\right\|_{L^p(I,L^q)} \\
		&\lesssim \|u_0\|_{H^1} + \|\scal{\nabla} (|x|^{-\sigma} u)\|_{L^{a'}(I, L^{b',2})} + \|\scal{\nabla} (|u|^\alpha u) \|_{L^{\kappa'}(I, L^{\mu'})} \\
		&\lesssim \|u_0\|_{H^1} + |I|^{\frac{2-\sigma}{2}} \|u\|_{L^m(I,W^1L^{n,2})} + |I|^\theta \|u\|^{\alpha+1}_{L^\kappa(I,W^{1,\mu})} \\
		&\leq C(\|u_0\|_{H^1}, |I|)<\infty.
		\end{align*}
		Finally, the conservation of mass and energy follows from the standard argument (see e.g. \cite[Chapter 3]{Cazenave}). The proof is complete.
	\end{proof}
	
	\subsection{Local well-posedness in the energy-critical case}
	\subsubsection{LWP via Strichartz estimates for $e^{-itH_c}$.}
	In this paragraph, we show the local well-posedness in $H^1$ for \eqref{NLS-rep} in the energy-critical case by using Strichartz estimates for $e^{-itH_c}$. The advantage of this method is the global well-posedness and scattering for small data. However, as in the usual local theory, the time of existence depends not only on the $H^1$-norm of initial data but also on its profile. More precisely, have the following result.
	\begin{proposition} \label{prop-sma-dat-sca}
		Let $d\geq 3$, $0<\sigma<2$ and $\alpha =\frac{4}{d-2}$. Then for any $u_0 \in H^1$, there exist $T_*, T^* \in (0,\infty]$ and a unique solution
		\[
		u \in C((-T_*,T^*), H^1) \cap L^\nu((-T_*,T^*), W^{1,\rho}),
		\]
		for some $(\nu, \rho) \in S$. There is conservation of mass and energy, i.e. \eqref{mas-ene} holds for all $t\in (-T_*,T^*)$. Moreover, if $\|u_0\|_{H^1} \leq \varep$ for some $\varep>0$ small enough, then $T_*=T^* =\infty$ and the solution scatters in $H^1$, i.e. there exist $u_0^\pm \in H^1$ such that
		\[
		\lim_{t\rightarrow \pm \infty} \|u(t)- e^{-itH_c} u_0^\pm\|_{H^1} =0.
		\] 
	\end{proposition}
	\begin{proof}
	Set 
	\begin{align} \label{cho-nu-rho}
	\nu = \frac{2(d+2)}{d-2}, \quad \rho = \frac{2d(d+2)}{d^2+4}.
	\end{align}
	It is easy to check that that $(\nu,\rho) \in S$. Moreover, by Proposition $\ref{prop-equ-sob}$, we see that under our assumptions of $d$ and $\sigma$, $W^{1,\rho}_c \sim W^{1,\rho}$. 
	
	Consider 
	\[
	X= \left\{ u \in C(J, H^1) \cap L^\nu(J, W^{1,\rho}) \ : \ \|u\|_{L^\nu(J,W^{1,\rho})} \leq M \right\}
	\]
	equipped with the distance
	\[
	d(u,v) = \|u-v\|_{L^\nu(J,L^\rho)},
	\]
	where $J=[0,T]$ with $T, M>0$ to be chosen later. By Duhamel's formula, it suffices to show that the functional
	\[
	\Phi(u(t)) = e^{-itH_c} u_0 \mp i \int_0^t e^{-i(t-s)H_c} |u(s)|^{\frac{4}{d-2}} u(s)ds =: u_{\text{hom}}(t) + u_{\text{inh}}(t)
	\]
	is a contraction on $(X,d)$. Thanks to Strichartz estimates given in Theorem $\ref{theo-glo-str}$ and the fact $W^{1,\rho}_c \sim W^{1,\rho}$, we have that
	\begin{align*}
	\|u_{\text{hom}}\|_{L^\nu(J, W^{1,\rho})} \sim \|u_{\text{hom}}\|_{L^\nu(J,W^{1,\rho}_c)} &= \left\|e^{-itH_c} \scal{\sqrt{H_c}} u_0 \right\|_{L^\nu(J,L^\rho)} \\
	&\lesssim \|u_0\|_{H^1_c} \sim \|u_0\|_{H^1}. 
	\end{align*}
	Note that for $d\geq 3$ and $0<\sigma<2$, we also have that $H^1_c \sim H^1$. This shows that $\|u_{\text{hom}}\|_{L^\nu(J,W^{1,\rho})} \leq \varep$ for some $\varep>0$ small enough to be specified shortly provided that $T$ is small or $\|u_0\|_{H^1}$ is small. We next bound
	\begin{align*}
	\|u_{\text{inh}}\|_{L^\nu(J,W^{1,\rho})} \sim \|u_{\text{inh}}\|_{L^\nu(J,W^{1,\rho}_c)} & = \left\|\int_0^t e^{-i(t-s)H_c} \scal{\sqrt{H_c}} \left(|u(s)|^{\frac{4}{d-2}} u(s) \right) ds \right\|_{L^\nu(J,L^{\rho})} \\
	&\lesssim \left\| |u|^{\frac{4}{d-2}} u \right\|_{L^2(J, W^{1,\frac{2d}{d+2}}_c)} \\
	&\lesssim \|u\|^{\frac{4}{d-2}}_{L^\nu(J, L^\rho)} \|u\|_{L^\nu(J,W^{1,\rho})} \lesssim \|u\|^{\frac{d+2}{d-2}}_{L^\nu(J,W^{1,\rho})}.
	\end{align*}
	Here we have used the fact that $W^{1,\frac{2d}{d+2}}_c \sim W^{1,\frac{2d}{d+2}}$. Similarly,
	\begin{align*}
	\|\Phi(u)-\Phi(v)\|_{L^\nu(J,L^\rho)} &\lesssim \left\| |u|^{\frac{4}{d-2}}u - |v|^{\frac{4}{d-2}} v \right\|_{L^2(J,L^{\frac{2d}{d+2}})} \\
	&\lesssim \left(\|u\|^{\frac{4}{d-2}}_{L^\nu(J,L^\rho)} + \|v\|^{\frac{4}{d-2}}_{L^\nu(J,L^\rho)} \right) \|u-v\|_{L^\nu(J,L^\rho)} \\
	&\lesssim \left(\|u\|^{\frac{4}{d-2}}_{L^\nu(J,W^{1,\rho})} + \|v\|^{\frac{4}{d-2}}_{L^\nu(J,W^{1,\rho})} \right) \|u-v\|_{L^\nu(J,L^\rho)}.
	\end{align*}
	This implies that for $u,v\in X$, there exists $C>0$ independent of $T$ and $u_0 \in H^1$ such that
	\begin{align*}
	\|\Phi(u)\|_{L^\nu(J,W^{1,\rho})} &\leq \varep + C M^{\frac{d+2}{d-2}}, \\
	d(\Phi(u),\Phi(v)) &\leq C M^{\frac{4}{d-2}} d(u,v).
	\end{align*}
	If we choose $\varep>0$ and $M$ small so that
	\[
	CM^{\frac{4}{d-2}} \leq \frac{1}{2}, \quad \varep + \frac{M}{2} \leq M,
	\]
	then $\Phi$ is a contraction on $(X,d)$. This shows the existence of local solutions. The conservation of mass and energy follows from the standard argument (see e.g. \cite[Chapter 3]{Cazenave}).
	
	It remains to show the energy scattering for small data. Note that if $\|u_0\|_{H^1}$ is small, then we can take $T=+\infty$ or the solution exists globally in  time. Let $t_2>t_1>0$. We have
	\begin{align*}
	\|e^{it_2 H_c} u(t_2) - e^{it_1H_c} u(t_1)\|_{H^1} \sim \|e^{it_2 H_c} u(t_2) - e^{it_1H_c} u(t_1)\|_{H^1_c} &= \left\|\int_{t_1}^{t_2} e^{isH_c} \scal{\sqrt{H_c}} \left( |u(s)|^{\frac{4}{d-2}} u(s)\right)ds \right\|_{L^2} \\
	&\lesssim \left\| |u|^{\frac{4}{d-2}} u \right\|_{L^2([t_1,t_2], W^{1,\frac{2d}{d+2}}_c)} \\
	&\lesssim \|u\|^{\frac{d+2}{d-2}}_{L^\nu([t_1,t_2],W^{1,\rho})}.
	\end{align*}
	Since $\|u\|_{L^\nu([0,+\infty), W^{1,\rho})} \leq M$, see see that $\|u\|_{L^\nu([t_1,t_2],W^{1,\rho})} \rightarrow 0$ as $t_1, t_2 \rightarrow +\infty$. This shows that the limit $u_0^+ := \lim_{t\rightarrow +\infty} e^{itH_c} u(t)$ exists in $H^1$. Moreover, 
	\[
	u(t) - e^{-itH_c} u_0^+ = \mp i \int_t^{+\infty} e^{-i(t-s)H_c} |u(s)|^{\frac{4}{d-2}} u(s) ds.
	\]
	Arguing as above, we show that
	\[
	\|u(t)-e^{-itH_c} u_0^+\|_{H^1} \rightarrow 0 \text{ as } t\rightarrow +\infty.
	\]
	This completes the proof for positive times, the one for negative times is similar. 
	\end{proof}

	\subsubsection{``Good'' LWP}
	In this paragraph, we show the ``good" local well-posedness for \eqref{NLS-rep} in the energy-critical case. More precisely, we prove the following result.
	
	\begin{proposition} \label{prop-good-lwp}
		Let $d\geq 3$, $\sigma$ be as in \eqref{res-sig-lor} and $\alpha=\frac{4}{d-2}$.	Then for any $u_0 \in H^1$, there exist $T_*, T^* \in (0,\infty]$ and a unique solution
		\[
		u \in C((-T_*,T^*), H^1) \cap L^m((-T_*,T^*), W^1L^{n,2}) \cap L^\nu((-T_*,T^*), W^{1,\rho}) \cap L^r((-T_*,T^*), W^{1,r}),
		\]
		for some $(m,n), (\nu, \rho), (r,r) \in S$. Moreover, the following properties hold:
		\begin{itemize}
			\item If $T^*<\infty$ (resp. $T_*<\infty$), then $\lim_{t\uparrow T^*} \|u(t)\|_{H^1} = \infty$ (resp. $\lim_{t\downarrow -T_*} \|u(t)\|_{H^1} = \infty$);
			\item $u \in L^p_{\emph{loc}}((-T_*,T^*), W^{1,q})$ for any $(p,q) \in S$;
			\item There is conservation of mass and energy, i.e. \eqref{mas-ene} holds for all $t\in (-T_*,T^*)$.
		\end{itemize} 
	\end{proposition}
	
	\begin{proof}
	The proof is done by several steps.
	
	{\bf Step 1. Estimates on the global solution of \eqref{ene-cri-NLS}.} 
	It follows from \cite{CKSTT, RV, Visan} (for the defocusing case) and from \cite{KM, KV, Dodson} (for the focusing case) that \eqref{ene-cri-NLS} is globally well-posed in $H^1$ and the global solution satisfies
	\[
	\sup_{(p,q) \in S} \|v\|_{L^p(\R, \dot{W}^{1,q})} \leq C(\|u_0\|_{\dot{H}^1}), \quad \sup_{(p,q) \in S} \|v\|_{L^p(\R, L^q)} \leq C(\|u_0\|_{\dot{H}^1}) \|u_0\|_{L^2}.
	\]
	In particular, 
	\begin{align} \label{glo-bou}
	\sup_{(p,q) \in S} \|v\|_{L^p(\R, W^{1,q})} \leq C(\|u_0\|_{H^1}).
	\end{align} 
	We also have that
	\begin{align} \label{glo-bou-lor}
	\sup_{(p,q) \in S} \|v\|_{L^p(\R, W^1 L^{q,2})} \leq C(\|u_0\|_{H^1}).
	\end{align}
	To see this, we divide $\R$ into $N=N(\|u_0\|_{H^1},\delta)$ subintervals $J_k=[t_k,t_{k+1}]$ such that
	\[
	\|v\|_{L^\nu(J_k, W^{1,\rho})} \leq \delta, \quad k=1, \cdots, N,
	\]
	for some $\delta>0$ to be chosen later. By Strichartz estimates, $L^{q,q}\equiv L^q$ and $L^{q,r_1} \subset L^{q,r_2}$ with $r_1 \leq r_2$, we have that
	\begin{align*}
	\|v\|_{\Lc^1(J_k)} &\lesssim \|v(t_k)\|_{H^1} + \left\|\scal{\nabla} \left(|v|^{\frac{4}{d-2}} v\right) \right\|_{L^{\frac{2(d+2)}{d+4},2} (J_k, L^{\frac{2(d+2)}{d+4},2})} \\
	&\lesssim \|v(t_k)\|_{H^1} + \left\|\scal{\nabla} \left(|v|^{\frac{4}{d-2}} v\right) \right\|_{L^{\frac{2(d+2)}{d+4}} (J_k, L^{\frac{2(d+2)}{d+4}})} \\
	&\lesssim \|v(t_k)\|_{H^1} + \|v\|^{\frac{4}{d-2}}_{L^{\frac{2(d+2)}{d-2}}(J_k, L^{\frac{2(d+2)}{d-2}})} \|\scal{\nabla} v\|_{L^{\frac{2(d+2)}{d}}(J_k, L^{\frac{2(d+2)}{d}})}\\
	&\lesssim \|v(t_k)\|_{H^1} + \|v\|^{\frac{4}{d-2}}_{L^\nu(J_k,W^{1,\rho})} \|\scal{\nabla} v\|_{L^{\frac{2(d+2)}{d}}(J_k, L^{\frac{2(d+2)}{d},2})} \\
	&\lesssim \|v(t_k)\|_{H^1} + \delta^{\frac{4}{d-2}} \|v\|_{\Lc^1(J_k)},
	\end{align*}
	where $\|v\|_{\Lc^1(J_k)} := \sup_{(p,q) \in S} \|v\|_{L^p(J_k, W^1 L^{q,2})}$. Taking $\delta>0$ small enough, we obtain that
	\begin{align} \label{glo-bou-lor-pro}
	\|v\|_{\Lc^1(J_k)} \lesssim \|v(t_k)\|_{H^1}, \quad k=1, \cdots, N.
	\end{align}
	Since $\|v(t_k)\|_{H^1} \leq \|v\|_{L^\infty(J_k, H^1)} \leq C(\|u_0\|_{H^1})$, \eqref{glo-bou-lor} follows by adding \eqref{glo-bou-lor-pro} over all subintervals $J_k$.
	
	{\bf Step 2. Solving the difference equation.} Since the energy-critical \eqref{NLS-rep} is invariant under the time translation, it suffices to show the well-posedness on the time interval $[0,T]$ for some small $T=T(\|u_0\|_{H^1})$. Let $T>0$ be a small constant to be specified later, and $v$ be the unique global solution of \eqref{ene-cri-NLS}. To recover $u$ on the time interval $[0,T]$, it suffices to solve the difference equation of $w = u-v$ with zero initial data on $[0,T]$, namely
	\begin{align} \label{dif-equ}
	\left\{
	\begin{array}{rcl}
	i\partial_t w + \Delta w &=& c|x|^{-\sigma}(v+w) \pm |v+w|^{\frac{4}{d-2}} (v+w) \mp |v|^{\frac{4}{d-2}} v, \\
	w(0) &=&0.
	\end{array}
	\right.
	\end{align}
	
	Before solving \eqref{dif-equ}, let us introduce the following space
	\[
	X^0(J):= L^\nu(J, L^\rho) \cap L^r(J,L^r), \quad X^1(J)= \{ u \ : \ \scal{\nabla} u \in X^0(J)\},
	\]
	where $(\nu,\rho)$ be as in \eqref{cho-nu-rho} and $r=\frac{2(d+2)}{d}$. Note that $(\nu,\rho)$ and $(r,r)$ are Schr\"odinger admissible pairs and $r'=\frac{2(d+2)}{d+4}$. We claim that
	\begin{align} 
	\left\|\scal{\nabla} \left(|v+w|^{\frac{4}{d-2}} (v+w) - |v|^{\frac{4}{d-2}}v \right)\right\|_{L^{r'}(J, L^{r'})} &\lesssim \|v\|^{\frac{4}{d-2}}_{X^1(J)} \|w\|_{X^1(J)} + \|v\|_{X^1(J)} \|w\|^{\frac{4}{d-2}}_{X^1(J)} + \|w\|^{\frac{d+2}{d-2}}_{X^1(J)}, \label{non-est-1}\\
	\left\||v+w|^{\frac{4}{d-2}} (v+w) - |v+z|^{\frac{4}{d-2}}(v+z) \right\|_{L^{r'}(J, L^{r'})} &\lesssim \left(\|v\|^{\frac{4}{d-2}}_{X^1(J)} + \|w\|^{\frac{4}{d-2}}_{X^1(J)} + \|z\|^{\frac{4}{d-2}}_{X^1(J)} \right) \|w-z\|_{X^0(J)}. \label{non-est-2}
	\end{align}
	Indeed, the left hand side of \eqref{non-est-1} is bounded by
	\begin{align} \label{non-est-pro}
	\left\||v+w|^{\frac{4}{d-2}} (v+w) - |v|^{\frac{4}{d-2}}v \right\|_{L^{r'}(J, L^{r'})} + \left\|\nabla \left(|v+w|^{\frac{4}{d-2}} (v+w) - |v|^{\frac{4}{d-2}}v \right)\right\|_{L^{r'}(J, L^{r'})}
	\end{align}
	By \eqref{dif-est-1} and the same argument as in the proof of \eqref{glo-bou-lor}, the first term in \eqref{non-est-pro} is bounded (up to a constant) by
	\begin{align*}
	\left( \|v\|^{\frac{4}{d-2}}_{L^\nu(J, L^\nu)} + \|w\|^{\frac{4}{d-2}}_{L^\nu(J,L^\nu)}\right) \|w\|_{L^r(J,L^r)} &\lesssim \left( \|v\|^{\frac{4}{d-2}}_{L^\nu(J, W^{1,\rho})} + \|w\|^{\frac{4}{d-2}}_{L^\nu(J,W^{1,\rho})}\right) \|w\|_{L^r(J,L^r)} \\
	&\lesssim \|v\|^{\frac{4}{d-2}}_{X^1(J)} \|w\|_{X^1(J)} + \|w\|^{\frac{d+2}{d-2}}_{X^1(J)}.
	\end{align*}
	Similarly, by \eqref{dif-est-2}, the second term in \eqref{non-est-pro} with $d\geq 7$ is bounded (up to a constant) by
	\begin{align*}
	\|\nabla v\|_{L^r(J,L^r)} \|w\|^{\frac{4}{d-2}}_{L^\nu(J,L^\nu)} + \|\nabla w\|_{L^r(J,L^r)} \|v\|^{\frac{4}{d-2}}_{L^\nu(J,L^\nu)} + \|\nabla w\|_{L^r(J,L^r)} \|w\|^{\frac{4}{d-2}}_{L^\nu(J,L^\nu)}
	\end{align*}
	which is then bounded by
	\[
	\|v\|^{\frac{4}{d-2}}_{X^1(J)} \|w\|_{X^1(J)} + \|v\|_{X^1(J)} \|w\|^{\frac{4}{d-2}}_{X^1(J)} + \|w\|^{\frac{d+2}{d-2}}_{X^1(J)}.
	\]
	In the case $3\leq d\leq 6$, it is bounded (up to a constant) by
	\begin{align*}
	\|\nabla v\|_{L^r(J,L^r)} \|w\|_{L^\nu(J,L^\nu)} \|v\|^{\frac{6-d}{d-2}}_{L^\nu(J,L^\nu)} + \|\nabla v\|_{L^r(J,L^r)} \|w\|^{\frac{4}{d-2}}_{L^\nu(J,L^\nu)} &+ \|\nabla w\|_{L^r(J,L^r)} \|v\|^{\frac{4}{d-2}}_{L^\nu(J,L^\nu)} \\
	&+ \|\nabla w\|_{L^r(J,L^r)} \|w\|^{\frac{4}{d-2}}_{L^\nu(J,L^\nu)}
	\end{align*}
	which is again bounded by
	\[
	\|v\|^{\frac{4}{d-2}}_{X^1(J)} \|w\|_{X^1(J)} + \|v\|_{X^1(J)} \|w\|^{\frac{4}{d-2}}_{X^1(J)} + \|w\|^{\frac{d+2}{d-2}}_{X^1(J)}.
	\]
	This proves \eqref{non-est-1}. The estimate \eqref{non-est-2} is treated similarly and we omit the details.
	
	We are now able to solve \eqref{dif-equ}. Thanks to \eqref{glo-bou}, we can divide $\R$ into $N= N(\|u_0\|_{H^1}, \eta)$ subintervals $J_1, \cdots, J_N$ with
	\begin{align} \label{spl-R}
	\|v\|_{X^1(J_k)} \sim \eta, \quad k = 1, \cdots, N,
	\end{align}
	for some small constant $\eta>0$ to be specified later. We are only interested in those intervals $J_k$ that have non-empty intersection with $[0,T]$. By renumbering if necessary, we may assume that there exists $N'<N$ such that for any $k=1, \cdots, N'$, $J_k \cap [0,T] \ne \emptyset$. We then write
	\[
	[0,T] = \cup_{k=1}^{N'} J_k, \quad J_k = [t_k, t_{k+1}].
	\]
	We will solve \eqref{dif-equ} on each $J_k, k=1, \cdots, N'$ by induction arguments. More precisely, we show that for each $k =1, \cdots, N'$, \eqref{dif-equ} has a unique solution $w$ on $J_k$ satisfying
	\begin{align} \label{bou-w}
	\|w\|_{L^\infty(J_k, H^1)} + \|w\|_{L^\nu(J_k, W^1 L^{\rho,2})} + \|w\|_{X^1(J_k)} \leq (2C)^k T^{\frac{2-\sigma}{2}},
	\end{align}
	for some constant $C>0$ independent of $T$. 
	
	Let us start with $k=1$. Consider
	\begin{multline*}
	Y_1 = \left\{ w \in C(J_1, H^1) \cap L^\nu(J_1, W^1 L^{\rho,2}) \cap X^1(J_1) \ : \ \right.\\
	\left.\|w\|_{Y_1} := \|w\|_{L^\infty(J_1, H^1)} + \|w\|_{L^\nu(J_1, W^1 L^{\rho,2})} + \|w\|_{X^1(J_1)} \leq M_1 \right\}
	\end{multline*}
	equipped with the distance
	\[
	d_1(w,z) := \|w-z\|_{L^\infty(J_1, L^2)} + \|w-z\|_{L^\nu(J_1, L^{\rho,2})} + \|w-z\|_{X^0(J_1)},
	\]
	where $M_1 = 2C T^{\frac{2-\sigma}{2}}$. We will show that the functional
	\[
	\Phi_1(w(t)) = - i \int_0^t e^{i(t-s)\Delta} \left( c|x|^{-\sigma} (v+w)(s) \pm |v+w|^{\frac{4}{d-2}} (v+w)(s) \mp |v|^{\frac{4}{d-2}} v(s) \right) ds
	\]
	is a contraction on $(Y_1, d_1)$. By Strichartz estimates, \eqref{pot-est}, \eqref{non-est-1} and \eqref{glo-bou-lor}, we have that
	\begin{align*}
	\|\Phi_1(w)\|_{Y_1} &\lesssim \|\scal{\nabla}\left(|x|^{-\sigma}(v+w)\right)\|_{L^{a'}(J_1, L^{b',2})} + \left\|\scal{\nabla} \left(|v+w|^{\frac{4}{d-2}} (v+w) - |v|^{\frac{4}{d-2}} v\right)\right\|_{L^{r'}(J_1, L^{r'})} \\
	&\lesssim |J_1|^{\frac{2-\sigma}{2}} \|\scal{\nabla}(v+w)\|_{L^\nu(J_1, L^{\rho,2})} + \|v\|^{\frac{4}{d-2}}_{X^1(J_1)} \|w\|_{X^1(J_1)} + \|v\|_{X^1(J_1)} \|w\|^{\frac{4}{d-2}}_{X^1(J_1)} + \|w\|^{\frac{d+2}{d-2}}_{X^1(J_1)} \\
	&\lesssim T^{\frac{2-\sigma}{2}} + T^{\frac{2-\sigma}{2}} \|w\|_{L^\nu(J_1, \scal{\nabla}^{-1} L^{\rho,2})} + \|v\|^{\frac{4}{d-2}}_{X^1(J_1)} \|w\|_{X^1(J_1)} + \|v\|_{X^1(J_1)} \|w\|^{\frac{4}{d-2}}_{X^1(J_1)} + \|w\|^{\frac{d+2}{d-2}}_{X^1(J_1)}.
	\end{align*}
	Similarly,
	\begin{align*}
	d_1(\Phi_1(w),\Phi_1(z)) &\lesssim \||x|^{-\sigma}(w-z)\|_{L^{a'}(J_1, L^{b',2})} + \left\| |v+w|^{\frac{4}{d-2}}(v+w) - |v+z|^{\frac{4}{d-2}} (v+z)\right\|_{L^{r'}(J_1, L^{r'})} \\
	&\lesssim |J_1|^{\frac{2-\sigma}{2}} \|w-z\|_{L^\nu(J_1, L^{\rho,2})} + \left(\|v\|^{\frac{4}{d-2}}_{X^1(J_1)} + \|w\|^{\frac{4}{d-2}}_{X^1(J_1)} + \|z\|^{\frac{4}{d-2}}_{X^1(J_1)}  \right) \|w-z\|_{X^0(J_1)} \\
	&\lesssim \left( T^{\frac{2-\sigma}{2}} + \|v\|^{\frac{4}{d-2}}_{X^1(J_1)} + \|w\|^{\frac{4}{d-2}}_{X^1(J_1)} + \|z\|^{\frac{4}{d-2}}_{X^1(J_1)} \right) d_1(w,z).
	\end{align*}
	Using \eqref{spl-R}, we see that for any $w, z \in Y_1$, there exists $C>0$ independent of $T$ such that
	\begin{align*}
	\|\Phi_1(w)\|_{Y_1} &\leq C T^{\frac{2-\sigma}{2}} + C T^{\frac{2-\sigma}{2}} M_1 + C \eta^{\frac{4}{d-2}} M_1 + \eta M_1^{\frac{4}{d-2}} + M_1^{\frac{d+2}{d-2}}, \\
	d_1(\Phi_1(w), \Phi_1(z)) &\leq \left( CT^{\frac{2-\sigma}{2}} + C\eta^{\frac{4}{d-2}} + CM_1^{\frac{4}{d-2}}\right) d_1(w,z).
	\end{align*}
	Since $M_1= 2C T^{\frac{2-\sigma}{2}}$, we first choose $\eta>0$ small enough such that $C\eta^{\frac{4}{d-2}} M_1 \leq \frac{1}{2} CT^{\frac{2-\sigma}{2}}$. We then choose $T>0$ small enough depending on $\eta$ so that $CT^{\frac{2-\sigma}{2}} M_1 + \eta M_1^{\frac{4}{d-2}} + M_1^{\frac{d+2}{d-2}} \leq \frac{1}{2} CT^{\frac{2-\sigma}{2}}$. By decreasing the values of $\eta$ and $T$ if necessary, we can make $CT^{\frac{2-\sigma}{2}} + C\eta^{\frac{4}{d-2}} + CM_1^{\frac{4}{d-2}} \leq \frac{1}{2}$. With these choice of $\eta$ and $T$, we see that $\Phi_1$ is a contraction on $(Y_1, d_1)$. This also implies \eqref{bou-w} for $k=1$. 
	
	We now assume that \eqref{dif-equ} has been solved on $J_{k-1}$ and $w$ satisfies \eqref{bou-w} up to $k-1$. We will show that \eqref{dif-equ} has a unique solution on $J_k$ satisfying \eqref{bou-w}. It suffices to show the functional
	\[
	\Phi_k(w(t)) = e^{i(t-t_k) \Delta} w(t_k) - i \int_{t_k}^t e^{i(t-s)\Delta} \left(c|x|^{-\sigma}(v+w)(s) \pm |v+w|^{\frac{4}{d-2}} (v+w)(s) \mp |v|^{\frac{4}{d-2}} v(s) \right)ds
	\]
	is a contraction on $(Y_k,d_k)$, where $Y_k$ and $d_k$ are defined as for $Y_1$ and $d_1$ with $J_k, M_k:=(2C)^k T^{\frac{2-\sigma}{2}}$ in place of $J_1, M_1$. Estimating as above, we get
	\begin{align*}
	\|\Phi_k(w)\|_{Y_k} &\lesssim \|w(t_k)\|_{H^1} + \|\scal{\nabla}\left(|x|^{-\sigma}(v+w)\right)\|_{L^{a'}(J_k, L^{b',2})} + \left\|\scal{\nabla} \left(|v+w|^{\frac{4}{d-2}} (v+w) - |v|^{\frac{4}{d-2}} v\right)\right\|_{L^{r'}(J_k, L^{r'})} \\
	&\lesssim \|w(t_k)\|_{H^1} + T^{\frac{2-\sigma}{2}} + T^{\frac{2-\sigma}{2}} \|w\|_{L^\nu(J_k, \scal{\nabla}^{-1} L^{\rho,2})} + \|v\|^{\frac{4}{d-2}}_{X^1(J_k)} \|w\|_{X^1(J_k)} \\
	&\mathrel{\phantom{\lesssim \|w(t_k)\|_{H^1} + T^{\frac{2-\sigma}{2}} + T^{\frac{2-\sigma}{2}} \|w\|_{L^\nu(J_k, \scal{\nabla}^{-1} L^{\rho,2})}}} + \|v\|_{X^1(J_k)} \|w\|^{\frac{4}{d-2}}_{X^1(J_k)} + \|w\|^{\frac{d+2}{d-2}}_{X^1(J_k)},
	\end{align*}
	and
	\begin{align*}
	d_k(\Phi_1(w),\Phi_1(z)) &\lesssim \||x|^{-\sigma}(w-z)\|_{L^{a'}(J_k, L^{b',2})} + \left\| |v+w|^{\frac{4}{d-2}}(v+w) - |v+z|^{\frac{4}{d-2}} (v+z)\right\|_{L^{r'}(J_k, L^{r'})} \\
	&\lesssim \left( T^{\frac{2-\sigma}{2}} + \|v\|^{\frac{4}{d-2}}_{X^1(J_k)} + \|w\|^{\frac{4}{d-2}}_{X^1(J_k)} + \|z\|^{\frac{4}{d-2}}_{X^1(J_k)} \right) d_k(w,z).
	\end{align*}
	This implies that for any $w, z \in Y_k$, there exists $C>0$ independent of $T$ such that
	\begin{align*}
	\|\Phi_k(w)\|_{Y_k} &\leq C\|w(t_k)\|_{H^1} + C T^{\frac{2-\sigma}{2}} + C T^{\frac{2-\sigma}{2}} M_k + C \eta^{\frac{4}{d-2}} M_k + \eta M_k^{\frac{4}{d-2}} + M_k^{\frac{d+2}{d-2}}, \\
	d_k(\Phi_k(w), \Phi_k(z)) &\leq \left( CT^{\frac{2-\sigma}{2}} + C\eta^{\frac{4}{d-2}} + CM_k^{\frac{4}{d-2}}\right) d_k(w,z).
	\end{align*}
	By the induction hypothesis, we see that $C\|w(t_k)\|_{H^1} \leq C\|w\|_{L^\infty(J_{k-1}, H^1)} \leq C(2C)^{k-1} T^{\frac{2-\sigma}{2}} = \frac{1}{2} (2C)^k T^{\frac{2-\sigma}{2}}$. By choosing $\eta$ and $T$ small enough, we show that $\Phi_k$ is a contraction on $(Y_k, d_k)$. Of course $T$ will depend on $k$, however, since $k \leq N' \leq N(\|u\|_{H^1}, \eta)$, we can choose $T$ to be a small constant depending only on $\|u_0\|_{H^1}$ and $\eta$. Therefore, we get a unique solution of \eqref{dif-equ} on $[0,T]$ satisfying
	\[
	\|w\|_{X^1([0,T]} \leq \sum_{k=1}^{N'} \|w\|_{X^1(J_k)} \leq \sum_{k=1}^{N'} (2C)^k T^{\frac{2-\sigma}{2}} \leq C (2C)^N T^{\frac{2-\sigma}{2}} \leq C(\|u_0\|_{H^1}).
	\]
	{\bf Step 3. Conclusion.}
	Since on $[0,T]$, $u=v+w$, we get a unique solution of \eqref{NLS-rep} on $[0,T]$ such that
	\[
	\|u\|_{X^1([0,T])} \leq \|v\|_{X^1([0,T]} + \|w\|_{X^1([0,T])} \leq C(\|u_0\|_{H^1}).
	\]
	By the same argument as in the proof of \eqref{glo-bou-lor}, we also have that
	\[
	\sup_{(p,q)\in S} \|u\|_{L^p([0,T], W^{1,q})} \leq C(\|u_0\|_{H^1}).
	\]
	The proof is complete.
	\end{proof}

	\section{Interaction Morawetz inequality}
	\label{S4}
	In this section, we establish the interaction Morawetz inequality for a class of nonlinear Schr\"odinger equations including \eqref{NLS-rep} in the defocusing case. Given a real valued function $a$, we define the Morawetz action by
	\begin{align} \label{mor-act}
	\Mcal_a(t):= 2 \int \nabla a \cdot \imt{\overline{u}(t) \nabla u(t)} dx.
	\end{align}
	Let us start with the following lemma (see e.g. \cite[Lemma 5.3]{TVZ}). 
	\begin{lemma} [\cite{TVZ}] Let $u$ be a (sufficiently smooth and decaying) solution to 
		\[
		i \partial_t u + \Delta u = N(u).
		\]
		Then it holds that
		\begin{align} \label{mor-act-rat}
		\begin{aligned}
		\frac{d}{dt} \Mcal_a(t)= - \int \Delta^2 a |u(t)|^2 dx &+ 4 \sum_{jk} \int \partial^2_{jk} a \ree{\partial_j \overline{u}(t) \partial_k u(t)} dx \\
		&+ 2 \int \nabla a \cdot \{N(u), u\}_p (t) dx,
		\end{aligned}
		\end{align}
		where $\{f,g\}_p:= \ree{f\nabla \overline{g} - g \nabla \overline{f}}$ is the momentum bracket. 
	\end{lemma}
	
	\begin{corollary} \label{coro-mor-act-rat}
		Let $V, W: \R^d \rightarrow \R$. If $u$ is a (sufficiently smooth and decaying) solution to 
		\begin{align} \label{NLS-VW}
		i\partial_t u + \Delta u = V u + W|u|^\alpha u,
		\end{align}
		then it holds that
		\begin{align} \label{app-mor-act-rat}
		\begin{aligned}
		\frac{d}{dt} \Mcal_a(t) &= - \int \Delta^2 a |u(t)|^2 dx + 4 \sum_{jk} \int \partial^2_{jk} a \ree{\partial_j \overline{u}(t) \partial_k u(t)} dx \\
		&\mathrel{\phantom{=}} - 2\int \nabla a \cdot \nabla V |u(t)|^2 dx + \frac{2\alpha}{\alpha+2} \int \Delta a W|u(t)|^{\alpha+2} dx - \frac{4}{\alpha+2} \int \nabla a \cdot \nabla W |u(t)|^{\alpha+2} dx.
		\end{aligned}
		\end{align}
	\end{corollary}
	\begin{proof}
		The result follows immediately from \eqref{mor-act-rat} and the fact that 
		\[
		\{Vu,u\}_p = \ret{ Vu \nabla \overline{u} - u \nabla(V \overline{u})}  = -\nabla V |u|^2,
		\]
		and
		\[
		\{W|u|^\alpha u, u\}_p = -\frac{\alpha}{\alpha+2} \nabla (W|u|^{\alpha+2}) - \frac{2}{\alpha+2} \nabla W |u|^{\alpha+2}.
		\]
	\end{proof}
	\begin{lemma} \label{lem-cla-mor-ine}
		Let $V, W: \R^d \rightarrow \R$ be radial functions satisfying $W\geq 0$ and $\partial_r V, \partial_r W \leq 0$. If $u: J \times \R^d \rightarrow \C$ is a (sufficiently smooth and decaying) solution to \eqref{NLS-VW}, then it holds that for $d\geq 3$,
		\[
		-2\int_J \int_{\R^d} \partial_r V |u(t)|^2 dx dt - \frac{4}{\alpha+2} \int_J \int_{\R^d} \partial_r W |u(t)|^{\alpha+2} dx dt \leq \sup_{t\in J} |\Mcal_{|x|}(t)| \leq C\|u\|_{L^\infty(J, L^2)} \|\nabla u\|_{L^\infty(J,L^2)}.
		\]
	\end{lemma}
	
	\begin{proof}
		By \eqref{app-mor-act-rat}, we have for a radial function $a$ that
		\begin{align*}
		\frac{d}{dt} \Mcal_a(t) &= - \int \Delta^2 a |u(t)|^2 dx + 4\sum_{jk} \int \partial^2_{jk} a \ret{\partial_j \overline{u}(t) \partial_k u(t)} dx \\
		&\mathrel{\phantom{=}} - 2 \int \partial_r a \partial_r V |u(t)|^2 dx + \frac{2\alpha}{\alpha+2} \int \Delta a W |u(t)|^{\alpha+2} dx - \frac{4}{\alpha+2} \int \partial_r a \partial_rW |u(t)|^{\alpha+2} dx.
		\end{align*}
		Applying the above identity to $a(x) = |x|$ with the fact
		\[
		\partial_r a = 1, \quad \Delta a = \frac{d-1}{|x|}, \quad \partial^2_{jk} a = \frac{\delta_{jk}}{|x|} - \frac{x_j x_k}{|x|^3}, \quad -\Delta^2 a = \left\{
		\begin{array}{cl}
		8\pi \delta_0 &\text{if } d=3, \\
		\frac{(d-1)(d-3)}{|x|^3} &\text{if } d\geq 4,
		\end{array}
		\right.
		\]
		and dropping positive terms, we obtain that
		\[
		\frac{d}{dt} \Mcal_{|x|}(t) \geq -2 \int \partial_r V |u(t)|^2 dx - \frac{4}{\alpha+2} \int \partial_r W |u(t)|^{\alpha+2} dx.
		\]
		Note that
		\begin{align} \label{pos-hes}
		\sum_{jk} \partial^2_{jk} a \ret{\partial_j \overline{u} \partial_k u} = \frac{1}{|x|} \left(|\nabla u|^2 - \left|\frac{x\cdot \nabla u}{|x|} \right|^2 \right) = \frac{1}{|x|} \left| \nabla u - \frac{x}{|x|} \left(\frac{x}{|x|} \cdot \nabla u \right)\right|^2 \geq 0.
		\end{align}
		Taking integration over a time interval $J$, the result follows by H\"older's inequality.
	\end{proof}

	\begin{remark} \label{rem-VW}
		It is easy to see that $V(x) = c|x|^{-\sigma}$ and $W(x) = |x|^{-b}$ with $c\sigma \geq 0$ and $b\geq 0$ satisfy the assumptions of Lemma $\ref{lem-cla-mor-ine}$.
	\end{remark}

	Now let $u,v$ be solution to 
	\[
	\left\{
	\begin{array}{lcl}
	i\partial_t u + \Delta_x u &=& N(u), \quad (t,x) \in \R \times \R^m, \\
	i\partial_t v + \Delta_y v &=& N(v), \quad (t,y) \in \R \times \R^n.
	\end{array}
	\right.
	\]
	Denote
	\[
	w(t,z)= (u\otimes v)(t,z) := u(t,x) v(t,y).
	\]
	It is obvious that $w$ solves 
	\begin{align} \label{equ-w}
	i \partial_t w + \Delta_z w = N(w),
	\end{align}
	where $\Delta_z:= \Delta_x + \Delta_y$ and $N(w) = N(u) v + N(v) u$. 
	
	Given a real-valued function $A$ on $\R^m \times \R^n$, we define the interaction Morawetz action 
	\[
	\Mcal_A^{\otimes 2} (t):= 2 \int \nabla_z A \cdot \imt{\overline{w}(t) \nabla_z w(t)} dz.
	\]
	A direct computation shows the following result (see e.g. \cite{CGT}).
	\begin{lemma} [\cite{CGT}] 
		Let $w$ be a (sufficiently smooth and decaying) solution to \eqref{equ-w}. Then it holds that
		\begin{align*}
		\frac{d}{dt} \Mcal_A^{\otimes 2}(t) &= - \int (\Delta^2_x A + \Delta^2_y A) |u(t)|^2 |v(t)|^2 dz \\
		&\mathrel{\phantom{=}} + 4 \sum_{jk} \int \partial^2_{jk} A \ree{\partial_j \overline{u}(t) \partial_k u(t)} |v(t)|^2dz + 4 \sum_{jk} \partial^2_{jk} A \ree{\partial_j \overline{v}(t) \partial_k v(t)} |u(t)|^2 dz \\
		&\mathrel{\phantom{=}} + 2 \int \nabla_x A \cdot \{N(u), u\}_p(t) |v(t)|^2 dz + 2\int \nabla_y A \cdot \{N(v), v\}_p(t) |u(t)|^2 dz.
		\end{align*}
	\end{lemma}

	\begin{corollary} \label{coro-int-mor-act-rat}
		Let $V, W:\R^d \rightarrow \R$ and $u$ be a (sufficiently smooth and decaying) solution to \eqref{NLS-VW}. Set 
		\[
		w(t,z) := u(t,x) u(t,y), \quad x,y \in \R^d.
		\] 
		Then it holds that
		\begin{align*}
		\frac{d}{dt} \Mcal_A^{\otimes 2}(t) = -2\int \Delta^2_x A (x,y) |u(t,x)|^2 |u(t,y)|^2 dxdy &+ 8 \sum_{jk} \int \partial^2_{jk}A(x,y) \ree{\partial_j \overline{u}(t,x) \partial_k u(t,x) } |u(t,y)|^2 dxdy \\
		&-4\int \nabla_x A(x,y) \cdot \nabla_x V(x)|u(t,x)|^2 |u(t,y)|^2 dxdy \\
		&+\frac{4\alpha}{\alpha+2} \int \Delta_x A(x,y) W(x) |u(t,x)|^{\alpha+2} |u(t,y)|^2 dxdy \\
		&-\frac{8}{\alpha+2} \int \nabla_x A(x,y) \cdot \nabla_x W(x) |u(t,x)|^{\alpha+2} |u(t,y)|^2 dx dy.
		\end{align*}
	\end{corollary}

	\begin{proposition} \label{prop-int-mor-ine}
		Let $V, W:\R^d \rightarrow \R$ be radial functions satisfying $W\geq 0$ and $\partial_r V, \partial_r W \leq 0$. If $u: J\times \R^d \rightarrow \C$ is a (sufficiently smooth and decaying) solution to \eqref{NLS-VW}, then it holds that
		\begin{itemize}
			\item for $d=3$,
			\[
			\int_J \int_{\R^3} |u(t,x)|^4 dx dt \leq C \|u\|^3_{L^\infty(J,L^2)} \|\nabla u\|_{L^\infty(J,L^2)};
			\]
			\item for $d\geq 4$,
			\[
			\int_J \int_{\R^d \times \R^d} \frac{|u(t,x)|^2 |u(t,y)|^2}{|x-y|^3} dx dy dt \leq C \|u\|^3_{L^\infty(J, L^2)} \|\nabla u\|_{L^\infty(J, L^2)}.
			\]
		\end{itemize}
		In particular, for $d\geq 3$,
		\[
		\| |\nabla|^{-\frac{d-3}{4}} u \|_{L^4(J,L^4)} \leq C\|u\|^{\frac{3}{4}}_{L^\infty(J,L^2)} \|\nabla u\|^{\frac{1}{4}}_{L^\infty(J,L^2)}.
		\]
	\end{proposition}

	\begin{proof}
		Let $A(x,y) = |x-y|$. A direct computation shows that
		\[
		\nabla_x A = \frac{x-y}{|x-y|}, \quad \Delta_x A = \frac{d-1}{|x-y|}, \quad \partial^2_{jk} A = \frac{\delta_{jk}}{|x-y|} - \frac{(x_j-y_j)(x_k-y_k)}{|x-y|^3},
		\]
		and
		\[
		-\Delta^2_x A= \left\{
		\begin{array}{cl}
		8\pi \delta_{x=y} &\text{if } d=3, \\
		\frac{(d-1)(d-3)}{|x-y|^3} &\text{if } d\geq 4.
		\end{array}
		\right.
		\]
		As in \eqref{pos-hes}, we see that
		\begin{align*}
		\sum_{jk} \partial^2_{jk} A \ret{\partial_j \overline{u} \partial_k u} = \frac{1}{|x-y|} \left| \nabla u - \frac{x-y}{|x-y|} \left(\frac{x-y}{|x-y|} \cdot \nabla u \right) \right|^2 \geq 0.
		\end{align*}
		Applying Corollary $\ref{coro-int-mor-act-rat}$ to $A(x,y) = |x-y|$ and dropping positive terms, we get 
		\begin{itemize}
			\item for $d=3$,
			\begin{align*}
			\frac{d}{dt} \Mcal^{\otimes 2}_{|x-y|}(t) \geq 16 \pi \int |u(t,x)|^4 dx &- 4 \int \frac{(x-y)\cdot x}{|x-y| |x|} \partial_r V |u(t,x)|^2 |u(t,y)|^2 dx dy \\
			& - \frac{8}{\alpha+2} \int \frac{(x-y) \cdot x}{|x-y||x|} \partial_r W |u(t,x)|^{\alpha+2} |u(t,y)|^2 dx dy.
			\end{align*}
			\item for $d\geq 4$,
			\begin{align*}
			\frac{d}{dt} \Mcal^{\otimes 2}_{|x-y|}(t) &\geq  2(d-1)(d-3) \int \frac{|u(t,x)|^2|u(t,y)|^2}{|x-y|^3} dxdy \\
			&\mathrel{\phantom{\geq}}- 4 \int \frac{(x-y) \cdot x}{|x-y| |x|} \partial_r V |u(t,x)|^2 |u(t,y)|^2 dx dy \\
			&\mathrel{\phantom{\geq}} - \frac{8}{\alpha+2} \int \frac{(x-y) \cdot x}{|x-y||x|} \partial_r W |u(t,x)|^{\alpha+2} |u(t,y)|^2 dx dy.
			\end{align*}
		\end{itemize}
		This implies that
		\begin{align*}
		\left. 
		\begin{array}{c}
		\mathlarger{\int}_{\R^3} |u(t,x)|^4 dx \\
		\mathlarger{\int}_{\R^d \times \R^d} \frac{|u(t,x)|^2|u(t,y)|^2}{|x-y|^3} dxdy
		\end{array}
		\right\} \lesssim \frac{d}{dt} \Mcal^{\otimes 2}_{|x-y|}(t) &- \int_{\R^d \times \R^d} \partial_r V |u(t,x)|^2 |u(t,y)|^2 dxdy \\
		&- \int_{\R^d \times \R^d} \partial_r W |u(t,x)|^{\alpha+2} |u(t,y)|^2 dx dy.
		\end{align*}
		Taking integration over a time interval $J$, we obtain
		\begin{align*}
		\left. 
		\begin{array}{c}
		\mathlarger{\int}_J \mathlarger{\int}_{\R^3} |u(t,x)|^4 dx dt\\
		\mathlarger{\int}_J \mathlarger{\int}_{\R^d \times \R^d} \frac{|u(t,x)|^2|u(t,y)|^2}{|x-y|^3} dxdy dt
		\end{array}
		\right\} \lesssim \sup_{t\in J} |\Mcal^{\otimes 2}_{|x-y|}(t)| &- \left(\int_J\int_{\R^d} \partial_r V |u(t,x)|^2 dxdt \right. \\
		&\mathrel{\phantom{-}}\left. + \int_J\int_{\R^d} \partial_r W |u(t,x)|^{\alpha+2}  dx dt\right) \|u\|^2_{L^\infty(J, L^2)}.
		\end{align*}
		We have from Lemma $\ref{lem-cla-mor-ine}$ that 
		\[
		-\int_J\int_{\R^d} \partial_r V |u(t,x)|^2 dxdt - \int_J\int_{\R^d} \partial_r W |u(t,x)|^{\alpha+2}  dx dt \lesssim \sup_{t \in J} |\Mcal_{|x|}(t)|.
		\]
		It follows that
		\begin{align*}
		\left. 
		\begin{array}{c}
		\mathlarger{\int}_J \mathlarger{\int}_{\R^3} |u(t,x)|^4 dx dt\\
		\mathlarger{\int}_J \mathlarger{\int}_{\R^d \times \R^d} \frac{|u(t,x)|^2|u(t,y)|^2}{|x-y|^3} dxdy dt
		\end{array}
		\right\} &\lesssim \sup_{t\in J} |\Mcal^{\otimes 2}_{|x-y|}(t)| + \sup_{t\in J} |\Mcal_{|x|}(t)| \|u\|^2_{L^\infty(J, L^2)} \\
		&\lesssim \|u\|^3_{L^\infty(J, L^2)} \|\nabla u\|_{L^\infty(J, L^2)}.
		\end{align*}
		For $d\geq 4$, we can write
		\[
		\int_J \int_{\R^d \times \R^d} \frac{|u(t,x)|^2 |u(t,y)|^2}{|x-y|^3} dx dy dt = \int_J \int_{\R^d} |u(t,x)|^2 \left( |u|^2 \ast \frac{1}{|\cdot|^3} \right) (t,x) dx dt.
		\]
		Recall that 
		\[
		|\nabla|^{-(d-3)} f(x) = C(d)\int_{\R^d} \frac{f(y)}{|x-y|^3} dy,
		\]
		for some constant $C(d)$ depending only on $d$. By Plancherel's theorem, we write
		\begin{align*}
		\int_J \int_{\R^d\times \R^d} \frac{|u(t,x)|^2 |u(t,y)|^2}{|x-y|^3} dx dy dt &= C^{-1}(d)\int_J \int_{\R^d} \widehat{|u|^2}(\xi) |\xi|^{-(d-3)} \widehat{|u|^2} (\xi) d\xi dt \\
		&= C^{-1}(d)\int_J \int_{\R^d} \left||\nabla|^{-\frac{d-3}{2}} (|u(t,x)|^2) \right|^2 dx dt.
		\end{align*}
		The result follows by using the fact (see \cite[Lemma 5.6]{TVZ}) that
		\[
		\||\nabla|^{-\frac{d-3}{4}} u \|_{L^4(J, L^4)}^4 \lesssim \||\nabla|^{-\frac{d-3}{2}}(|u|^2)\|^2_{L^2(J, L^2)}.
		\]
		The proof is complete.
	\end{proof}
	\section{Virial estimates}
	\label{S5}
	In this section, we derive some virial estimates related to \eqref{NLS-rep} in the focusing case which are useful to study the finite time blow-up. Given a real valued funtion $a$, we define the virial potential by
	\[
	\Vc_a(t):= \int a |u(t)|^2 dx.
	\]
	\begin{lemma}
		Let $V, W: \R^d \rightarrow \C$ and $u_0 \in H^1$ be such that $|x| u_0 \in L^2$. Let $u: J\times \R^d \rightarrow \C$  be the corresponding solution to \eqref{NLS-VW} with initial data $u(0)=u_0$. Then the function $t\mapsto |\cdot| u(t,\cdot)$ belongs to $C(J, L^2)$. Moreover, the function $t\mapsto \|xu(t)\|^2_{L^2}$ is in $C^2(J)$, and for any $t\in J$,
		\[
		\frac{d}{dt} \|xu(t)\|^2_{L^2} = 4 \int x \cdot \ime{\overline{u}(t) \nabla u(t)} dx,
		\]
		and
		\[
		\frac{d^2}{dt^2} \|xu(t)\|^2_{L^2}  = 8 \int |\nabla u(t)|^2 dx - 4 \int x \cdot \nabla V |u(t)|^2 dx + \frac{4d\alpha}{\alpha+2} \int W |u(t)|^{\alpha+2} dx - \frac{8}{\alpha+2} \int x \cdot \nabla W |u(t)|^{\alpha+2} dx.
		\]
	\end{lemma}

	\begin{proof}
		The first claim follows from a standard approximation argument (see e.g. \cite[Proposition 6.5.1]{Cazenave}). Observe that $\frac{d}{dt} \|xu(t)\|^2_{L^2} = \Mcal_{|x|^2}(t)$ (see \eqref{mor-act}). The second derivative of $\|xu(t)\|^2_{L^2}$ follows from Corollary $\ref{coro-mor-act-rat}$ with $a(x) = |x|^2$. 
	\end{proof}

	\begin{corollary} \label{coro-vir-ide}
		Let $d\geq 1, c>0$, $0<\sigma<\min\{2,d\}$ and $\alpha>0$. Let $u_0 \in H^1$ be such that $|x|u_0 \in L^2$, and $u: J \times \R^d \rightarrow \C$ the corresponding solution to the focusing \eqref{NLS-rep}. Then the function $t \mapsto |\cdot| u(t,\cdot)$ belongs to $C(J, L^2)$. Moreover, the function $t\mapsto \|xu(t)\|^2_{L^2}$ is in $C^2(J)$, and for any $t\in J$,
		\begin{align} \label{vir-ide}
		\begin{aligned}
		\frac{d^2}{dt^2} \|xu(t)\|^2_{L^2} &= 8 \|\nabla u(t)\|^2_{L^2} + 4c\sigma \||x|^{-\sigma} |u(t)|^2\|_{L^1} -\frac{4d\alpha}{\alpha+2} \|u(t)\|^{\alpha+2}_{L^{\alpha+2}} \\
		&= 16 E(u(t)) - 4c(2-\sigma) \||x|^{-\sigma} |u(t)|^2\|_{L^1} - \frac{4(d\alpha-4)}{\alpha+2} \|u(t)\|^{\alpha+2}_{L^{\alpha+2}} \\
		&= 4d\alpha E(u(t)) - 2(d\alpha-4) \|\nabla u(t)\|^2_{L^2} -2c (d\alpha -2\sigma) \||x|^{-\sigma} |u(t)|^2\|_{L^1}.
		\end{aligned}
		\end{align}
	\end{corollary}

	We next derive some localized virial estimates which are useful to show the blow-up for \eqref{NLS-rep} in the focusing case with radial initial data. This is done by the same spirit of \cite{Dinh-blo, Dinh-inv}.  Let $\chi$ be a function defined on $[0,\infty)$ and satisfy
	\begin{align} \label{def-chi}
	\chi(r) = \left\{
	\begin{array} {cl}
	r^2 & \text{if } 0 \leq r \leq 1, \\
	\text{const.} &\text{if } r\geq 2, 
	\end{array}
	\right.
	\quad \text{and} \quad \chi''(r) \leq 2 \text{ for } r\geq 0.
	\end{align}
	Given $R>0$, we define the radial function
	\begin{align} \label{def-var-R}
	\varphi_R(x) = \varphi_R(r) := R^2 \chi(r/R), \quad  r= |x|.
	\end{align}
	By definition, we see that
	\begin{align} \label{pro-var-R}
	2-\varphi''_R(r) \geq 0, \quad 2- \frac{\varphi_R'(r)}{r} \geq 0, \quad 2d - \Delta \varphi_R(x) \geq 0, \quad \forall r \geq 0, x \in \R^d.
	\end{align}
	\begin{lemma} \label{lem-vir-est-1}
		Let $d\geq 2$, $c>0$, $0<\sigma<2$ and $0<\alpha \leq 4$. Let $u: J \times \R^d \rightarrow \C$ be a radial solution to \eqref{NLS-rep} in the focusing case. Then for any $\varep>0$ and any $t \in J$,
		\begin{align} \label{vir-est-1}
		\begin{aligned}
		\frac{d^2}{dt^2} \Vc_{\varphi_R}(t) &\leq 8 \|\nabla u(t)\|^2_{L^2} + 4c \sigma \||x|^{-\sigma} |u(t)|^2\|_{L^1} - \frac{4d\alpha}{\alpha+2} \|u(t)\|^{\alpha+2}_{L^{\alpha+2}} \\
		&\mathrel{\phantom{\leq}} + \left\{
		\renewcommand*{\arraystretch}{1.3}
		\begin{array}{cl}
		O \left( R^{-2} + R^{-2(d-1)} \|\nabla u(t)\|^2_{L^2} \right) &\text{if } \alpha=4, \\
		O \left(R^{-2} + \varep^{-\frac{\alpha}{4-\alpha}} R^{-\frac{2(d-1)\alpha}{4-\alpha}} + \varep \|\nabla u(t)\|^2_{L^2} \right) &\text{if } 0<\alpha<4.
		\end{array}
		\right.
		\end{aligned}
		\end{align}
		Here the implicit constant depends only on $d$ and $\alpha$.
	\end{lemma}
		
	\begin{proof}
		Applying Corollary $\ref{coro-mor-act-rat}$ with $V=c|x|^{-\sigma}$ and $W=-1$, we get that
		\begin{align*}
		\frac{d^2}{dt^2} \Vc_{\varphi_R}(t) &= -\int \Delta^2 \varphi_R |u(t)|^2 dx + 4 \sum_{jk} \int \partial^2_{jk} \varphi_R \ret{\partial_j\overline{u}(t) \partial_k u(t)} dx \\
		&\mathrel{\phantom{=}} - 2c\int \nabla \varphi_R \cdot \nabla (|x|^{-\sigma}) |u(t)|^2 dx - \frac{2\alpha}{\alpha+2} \int \Delta \varphi_R |u(t)|^{\alpha+2} dx.
		\end{align*}
		Using the fact that
		\[
		\partial_j = \frac{x_j}{r} \partial_r, \quad \partial^2_{jk} = \left(\frac{\delta_{jk}}{r} - \frac{x_jx_k}{r^3}\right) \partial_r + \frac{x_j x_k}{r^2} \partial^2_r,
		\]
		we have that
		\[
		\sum_{jk} \partial^2_{jk} \varphi_R \partial_j \overline{u} \partial_k u = \varphi''_R |\partial_r u|^2 = \varphi''_R |\nabla u|^2, \quad \nabla \varphi_R \cdot \nabla (|x|^{-\sigma}) = -\sigma \frac{\varphi'_R}{r} |x|^{-\sigma}.
		\]
		We then write
		\begin{align} \label{vir-est-1-pro}
		\begin{aligned}
		\frac{d^2}{dt^2} \Vc_{\varphi_R}(t) &= 8 \|\nabla u(t)\|^2_{L^2} + 4c\sigma\||x|^{-\sigma} |u(t)|^2\|_{L^1} -\frac{4d\alpha}{\alpha+2} \|u(t)\|^{\alpha+2}_{L^{\alpha+2}} \\
		&\mathrel{\phantom{=}} -\int \Delta^2 \varphi_R |u(t)|^2 dx - 4 \int (2-\varphi''_R) |\nabla u(t)|^2 dx \\
		&\mathrel{\phantom{=}} - 2c\sigma \int \left( 2- \frac{\varphi'_R}{r}\right)|x|^{-\sigma} |u(t)|^2 dx + \frac{2\alpha}{\alpha+2} \int(2d - \Delta \varphi_R) |u(t)|^{\alpha+2} dx.
		\end{aligned}
		\end{align}
		Thanks to \eqref{pro-var-R} and the fact $|2d-\Delta\varphi_R| \lesssim 1$, $\text{supp}(2d-\Delta \varphi_R) \subset \{|x| >R\}$, $|\Delta^2\varphi| \lesssim R^{-2}$, the conservation of mass implies that
		\begin{align*}
		\frac{d^2}{dt^2} \Vc_{\varphi_R}(t) \leq 8 \|\nabla u(t)\|^2_{L^2} + 4c\sigma\||x|^{-\sigma} |u(t)|^2\|_{L^1} & -\frac{4d\alpha}{\alpha+2} \|u(t)\|^{\alpha+2}_{L^{\alpha+2}} \\
		& + O \left( R^{-2} + \int_{|x|>R} |u(t)|^{\alpha+2} dx\right).		\end{align*}
		We next recall the following radial Sobolev embedding (see e.g. \cite{CO,Strauss}): for $d\geq 2$, there exists $C=C(d)>0$ such that for any radial function $f \in H^1$,
		\begin{align} \label{rad-sob-emb}
		\sup_{x\ne 0} |x|^{\frac{d-1}{2}} |f(x)| \leq C\|f\|^{\frac{1}{2}}_{L^2} \|\nabla f\|_{L^2}^{\frac{1}{2}}.
		\end{align}
		Using \eqref{rad-sob-emb} and the conservation of mass, we estimate
		\begin{align*}
		\int_{|x|>R} |u(t)|^{\alpha+2} dx &\leq \left(\sup_{|x|>R} |u(t,x)|^\alpha \right)\|u(t)\|^2_{L^2} \\
		&\leq R^{-\frac{(d-1)\alpha}{2}} \left(\sup_{|x|>R} |x|^{\frac{d-1}{2}} |u(t,x)| \right)^\alpha \|u(t)\|^2_{L^2} \\
		&\lesssim R^{-\frac{(d-1)\alpha}{2}} \|\nabla u(t)\|^{\frac{\alpha}{2}}_{L^2} \|u(t)\|^{\frac{\alpha}{2}+2}_{L^2} \\
		&\lesssim R^{-\frac{(d-1)\alpha}{2}} \|\nabla u(t)\|^{\frac{\alpha}{2}}_{L^2}.
		\end{align*}
		When $\alpha=4$, we are done. When $0<\alpha<4$, we use the Young inequality to get
		\[
		R^{-\frac{(d-1)\alpha}{2}} \|\nabla u(t)\|^{\frac{\alpha}{2}}_{L^2} \lesssim \varep \|\nabla u(t)\|^2_{L^2} + \varep^{-\frac{\alpha}{4-\alpha}} R^{-\frac{2(d-1)\alpha}{4-\alpha}}. 
		\]
		The proof is complete.
	\end{proof}
	
	We also have the following refined version of Lemma $\ref{lem-vir-est-1}$ in the mass-critical case $\alpha=\frac{4}{d}$.
	\begin{lemma} \label{lem-vir-est-2}
	Let $d\geq 2$, $c>0$, $0<\sigma<2$ and $\alpha=\frac{4}{d}$. Let $u: J \times \R^d \rightarrow \C$ be a radial solution to \eqref{NLS-rep} in the focusing case. Then there exists $C=C(d)>0$ such that for any $\varep>0$ and any $t\in J$, 
	\begin{align} \label{vir-est-2}
	\begin{aligned}
	\frac{d^2}{dt^2} \Vc_{\varphi_R}(t) \leq 16 E(u(t)) &- 4 \int \left(\psi_{1,R} - C \varep \psi_{2,R}^{\frac{d}{2}} \right) |\nabla u(t)|^2 dx \\
	&+ O\left( R^{-2} + \varep R^{-2} + \varep^{-\frac{1}{d-1}} R^{-2}\right),
	\end{aligned}
	\end{align}
	where
	\begin{align} \label{def-psi-12-R}
	\psi_{1,R} = 2-\varphi''_R, \quad \psi_{2,R} = 2d-\Delta \varphi_R.
	\end{align}
	Here the implicit constant depends only on $d$.
	\end{lemma}
	
	\begin{proof}
		Using \eqref{vir-ide} and \eqref{vir-est-1-pro} with $\alpha=\frac{4}{d}$, we have that
		\begin{align*}
		\frac{d^2}{dt^2} \Vc_{\varphi_R}(t) = 16E(u(t)) &- 4c(2-\sigma)\||x|^{-\sigma}|u(t)|^2\|_{L^1} \\
		&-\int \Delta^2\varphi_R |u(t)|^2 dx - 4 \int(2-\varphi''_R)|\nabla u(t)|^2 dx \\
		&-2c\sigma\int \left(2-\frac{\varphi'_R}{r}\right) |x|^{-\sigma} |u(t)|^2 dx + \frac{4}{d+2} \int (2d-\Delta \varphi_R) |u(t)|^{\frac{4}{d}+2} dx.
		\end{align*}
		Since $0<\sigma<2$, $2-\frac{\varphi'_R}{r}\geq 0$ and $|\Delta^2\varphi_R| \lesssim R^{-2}$, the conservation of mass implies that
		\[
		\frac{d^2}{dt^2} \Vc_{\varphi_R}(t) \leq 16E(u(t)) - 4 \int \psi_{1,R} |\nabla u(t)|^2 dx +\frac{4}{d+2} \int \psi_{2,R} |u(t)|^{\frac{4}{d}+2} dx + O(R^{-2}),
		\]
		where $\psi_{1,R}$ and $\psi_{2,R}$ are as in \eqref{def-psi-12-R}. Thanks to the radial Sobolev embedding \eqref{rad-sob-emb}, the conservation of mass and the fact $|\psi_{2,R}| \lesssim 1$, $\text{supp}(\psi_{2,R}) \subset \{|x|>R\}$, we estimate
		\begin{align*}
		\int \psi_{2,R} |u(t)|^{\frac{4}{d}+2} dx &= \int_{|x|>R} \left|\psi_{2,R}^{\frac{d}{4}} u(t)\right|^{\frac{4}{d}} |u(t)|^2 dx \\
		&\leq \left(\sup_{|x|>R} \left|\psi_{2,R}^{\frac{d}{4}} u(t,x)\right|^{\frac{4}{d}}\right) \|u(t)\|^2_{L^2} \\
		&\leq R^{-\frac{2(d-1)}{d}} \left( \sup_{|x|>R} |x|^{\frac{d-1}{2}} \left| \psi_{2,R}^{\frac{d}{4}} u(t,x)\right|\right)^{\frac{4}{d}} \|u(t)\|^2_{L^2} \\
		&\lesssim R^{-\frac{2(d-1)}{d}} \left\|\nabla \left( \psi_{2,R}^{\frac{d}{4}} u(t) \right)\right\|_{L^2}^{\frac{2}{d}} \left\|\psi_{2,R}^{\frac{d}{4}} u(t)\right\|_{L^2}^{\frac{2}{d}} \|u(t)\|^2_{L^2} \\
		&\lesssim R^{-\frac{2(d-1)}{d}} \left\|\nabla \left( \psi_{2,R}^{\frac{d}{4}} u(t) \right)\right\|_{L^2}^{\frac{2}{d}}.
		\end{align*}
		By the Young inequality, we get for any $\varep>0$,
		\[
		R^{-\frac{2(d-1)}{d}} \left\|\nabla \left( \psi_{2,R}^{\frac{d}{4}} u(t) \right)\right\|_{L^2}^{\frac{2}{d}}  \lesssim \varep \left\| \nabla \left( \psi_{2,R}^{\frac{d}{4}} u(t) \right) \right\|^2_{L^2} + \varep^{-\frac{1}{d-1}} R^{-2}.
		\]
		By the definition of $\varphi_R$, it is easy to see that $\left| \nabla \left(\psi_{2,R}^{\frac{d}{4}} \right)\right| \lesssim R^{-1}$. The conservation of mass then implies that
		\[
		\left\| \nabla \left( \psi_{2,R}^{\frac{d}{4}} u(t) \right) \right\|^2_{L^2} \lesssim R^{-2} + \left\|\psi_{2,R}^{\frac{d}{4}} \nabla u(t) \right\|^2_{L^2}
		\]
		Collecting the above estimates, we prove \eqref{vir-est-2}.
	\end{proof}

	\section{Global well-posedness}
	\label{S6}
	In this section, we prove the global well-posedness in the energy space for \eqref{NLS-rep}. 
	
	\noindent {\bf Proof of Theorem $\ref{theo-GWP}$.}
		Thanks to the local well-posedness in the energy-subcritical case, the ``good" local well-posedness in the energy-critical case and the conservation of mass, the result follows if we can show that there exists $C>0$ independent of $t$ such that $\|\nabla u(t)\|_{L^2} \leq C$ for any $t$ in the existence time. 
		
		In the defocusing case, we have from the conservation of energy that
		\[
		\|\nabla u(t)\|_{L^2} \leq \sqrt{2 E(u_0)}
		\] 
		for any $t$ in the existence time. 
		
		In the focusing case, we consider several subcases.
		
		{\bf Subcase 1: Mass-subcritical case.} By the Gagliardo-Nirenberg inequality, we have that
		\begin{align*}
		E(u(t)) &= \frac{1}{2} \|\nabla u(t)\|^2_{L^2} + \frac{c}{2}\||x|^{-\sigma} |u(t)|^2 \|_{L^1} - \frac{1}{\alpha+2} \|u(t)\|^{\alpha+2}_{L^{\alpha+2}} \\
		&\geq \frac{1}{2} \|\nabla u(t)\|^2_{L^2} - \frac{C_{\text{GN}}}{\alpha+2} \|\nabla u(t)\|^{\frac{d\alpha}{2}}_{L^2} \|u(t)\|^{\frac{4-(d-2)\alpha}{2}}_{L^2}.
		\end{align*}
		Since $\frac{d\alpha}{2}<2$, we use the Young inequality and the conservation of mass to get that for any $\varep>0$,
		\begin{align} \label{gag-nir-ine-app}
		\frac{C_{\text{GN}}}{\alpha+2} \|\nabla u(t)\|^{\frac{d\alpha}{2}}_{L^2} \|u(t)\|^{\frac{4-(d-2)\alpha}{2}}_{L^2} \leq \varep \|\nabla u(t)\|^2_{L^2} + C(\varep, \|u_0\|_{L^2}).
		\end{align}
		The conservation of energy then implies that
		\[
		\left(\frac{1}{2}-\varep \right) \|\nabla u(t)\|^2_{L^2} \leq E(u_0) + C(\varep, \|u_0\|_{L^2}).
		\]
		Taking $0<\varep<\frac{1}{2}$, we obtain the uniform bound on $\|\nabla u(t)\|_{L^2}$.
		
		{\bf Subcase 2: Mass-critical case.} By the sharp Gagliardo-Nirenberg inequality with \eqref{sha-con-gag-nir-mas} and the conservation of mass and energy, we have that
		\begin{align*}
		E(u_0)=E(u(t)) &= \frac{1}{2} \|\nabla u(t)\|^2_{L^2}+\frac{c}{2} \||x|^{-\sigma} |u(t)|^2 \|_{L^1} -\frac{d}{2d+4} \|u(t)\|^{\frac{4}{d}+2}_{L^{\frac{4}{d}+2}} \\
		&\geq \frac{1}{2} \|\nabla u(t)\|^2_{L^2} - \frac{1}{2} \left(\frac{\|u(t)\|_{L^2}}{\|Q\|_{L^2}}\right)^{\frac{4}{d}} \|\nabla u(t)\|^2_{L^2} \\
		&= \frac{1}{2} \left(1-\left(\frac{\|u_0\|_{L^2}}{\|Q\|_{L^2}}\right)^{\frac{4}{d}}\right) \|\nabla u(t)\|^2_{L^2}.
		\end{align*}
		Since $\|u_0\|_{L^2} <\|Q\|_{L^2}$, we again get the uniform bound on $\|\nabla u(t)\|_{L^2}$.
		
		{\bf Subcase 3: Intercritical case.} It follows from the sharp Gagliardo-Nirenberg inequality that
		\begin{align} \label{gwp-int-pro}
		\begin{aligned}
		E(u(t)) M^{\betc}(u(t)) &= \frac{1}{2} \|\nabla u(t)\|_{L^2}^2 \|u(t)\|^{2\betc}_{L^2} + \frac{c}{2} \||x|^{-\sigma} |u(t)|^2 \|_{L^1} \|u(t)\|^{2\betc}_{L^2} - \frac{1}{\alpha+2} \|u(t)\|^{\alpha+2}_{L^{\alpha+2}} \|u(t)\|^{2\betc}_{L^2} \\
		&\geq \frac{1}{2} \left(\|\nabla u(t)\|_{L^2} \|u(t)\|^{\betc}_{L^2} \right)^2 - \frac{C_{\text{GN}}}{\alpha+2} \|\nabla u(t)\|^{\frac{d\alpha}{2}}_{L^2} \|u(t)\|_{L^2}^{\frac{4-(d-2)\alpha}{2}+2\betc} \\
		&= f\left(\|\nabla u(t)\|_{L^2} \|u(t)\|^{\betc}_{L^2}\right),
		\end{aligned}
		\end{align}
		where $f(x):= \frac{1}{2} x^2 - \frac{C_{\text{GN}}}{\alpha+2} x^{\frac{d\alpha}{2}}$. Using \eqref{ide-Q} and \eqref{sha-con-gag-nir-int}, it is easy to check that
		\begin{align} \label{ide-Q-app}
		f\left(\|\nabla Q\|_{L^2} \|Q\|^{\betc}_{L^2}\right) = \frac{d\alpha-4}{2d\alpha} \left(\|\nabla Q\|_{L^2} \|Q\|_{L^2}^{\betc}\right)^2 = E_0(Q) M^{\betc}(Q).
		\end{align}
		By \eqref{gwp-int-pro}, the conservation of energy and mass and the first condition in \eqref{glo-con-int}, we get that
		\[
		f\left(\|\nabla u(t)\|_{L^2} \|u(t)\|^{\betc}_{L^2} \right) \leq E(u_0) M^{\betc}(u_0) < E_0(Q) M^{\betc}(Q) = f\left(\|\nabla Q\|_{L^2} \|Q\|^{\betc}_{L^2}\right)
		\]
		for any $t$ in the existence time. Thanks to the second condition in \eqref{glo-con-int}, the continuity argument implies that
		\[
		\|\nabla u(t)\|_{L^2} \|u(t)\|^{\betc}_{L^2} < \|\nabla Q\|_{L^2} \|Q\|^{\betc}_{L^2}
		\]
		for any $t$ in the existence time. This gives the uniform bound on $\|\nabla u(t)\|_{L^2}$. 
		
		{\bf Subcase 4: Energy-critical case.} By the sharp Sobolev embedding, we see that
		\begin{align} \label{gwp-ene-pro}
		\begin{aligned}
		E(u(t)) &= \frac{1}{2} \|\nabla u(t)\|^2_{L^2} +\frac{c}{2} \||x|^{-\sigma} |u(t)|^2 \|_{L^1} - \frac{d-2}{2d} \|u(t)\|^{\frac{2d}{d-2}}_{L^{\frac{2d}{d-2}}} \\
		&\geq \frac{1}{2} \|\nabla u(t)\|^2_{L^2} - \frac{(d-2)C_{\text{SE}}}{2d} \|\nabla u(t)\|^{\frac{2d}{d-2}}_{L^2} \\
		&= g(\|\nabla u(t)\|_{L^2}),
		\end{aligned}
		\end{align}
		where $g(y)=\frac{1}{2} y^2 - \frac{(d-2)C_{\text{SE}}}{2d} y^{\frac{2d}{d-2}}$. It follows from \eqref{ide-W} and \eqref{sha-con-sob-emb} that
		\begin{align} \label{ide-W-app}
		g(\|\nabla W\|_{L^2}) = \frac{1}{d} \|\nabla W\|^2_{L^2} = E_0(W).
		\end{align}
		Thanks to \eqref{gwp-ene-pro}, the conservation of energy and the first condition in \eqref{glo-con-ene}, we get that
		\[
		g(\|\nabla u(t)\|_{L^2}) \leq E(u_0) <E_0(W) = g(\|\nabla W\|_{L^2})
		\]
		for any $t$ in the existence time. The continuity argument together with the second condition in \eqref{glo-con-ene} imply that
		\[
		\|\nabla u(t)\|_{L^2}< \|\nabla W\|_{L^2}
		\]
		for any $t$ in the existence time. The proof is complete.
	\hfill $\Box$

	\noindent {\bf Proof of Proposition $\ref{prop-gwp-att}$.}
	As in the proof of Theorem $\ref{theo-GWP}$, it suffices to show the uniform bound of the kinetic energy. Recall that we consider the case $c<0$ here.
	
	In the defocusing case, we have that
	\begin{align*}
	E(u(t)) &= \frac{1}{2} \|\nabla u(t)\|^2_{L^2} - \frac{|c|}{2} \||x|^{-\sigma} |u(t)|^2 \|_{L^1} +\frac{1}{\alpha+2} \|u(t)\|^{\alpha+2}_{L^{\alpha+2}} \\
	&\geq \frac{1}{2} \|\nabla u(t)\|^2_{L^2} - \frac{|c|}{2} \||x|^{-\sigma} |u(t)|^2 \|_{L^1}.
	\end{align*}
	By Hardy's inequality,
	\[
	\||x|^{-\sigma} |u(t)|^2 \|_{L^1}=\||x|^{-\frac{\sigma}{2}} u(t)\|^2_{L^2} \leq C \||\nabla|^{\frac{\sigma}{2}} u(t)\|^2_{L^2} \leq C \|\nabla u(t)\|^{\sigma}_{L^2} \|u(t)\|^{2-\sigma}_{L^2}.
	\]
	Since $0<\sigma<2$, we apply the Young's inequality $ab \leq \frac{a^p}{\varep p} + \frac{\varep^{\frac{q}{p}} b^q}{q}$ with $\varep>0$, $a,b\geq 0$ and $1<p,q<\infty$ satisfying $\frac{1}{p}+\frac{1}{q} =1$ to get
	\begin{align} \label{har-ine-app}
	\||x|^{-\sigma} |u(t)|^2 \|_{L^1} \leq \frac{1}{2|c|} \|\nabla u(t)\|^2_{L^2} + C(\sigma,|c|) \|u(t)\|^2_{L^2}.
	\end{align}
	This implies that
	\[
	E(u(t)) \geq \frac{1}{4} \|\nabla u(t)\|^2_{L^2} - C(\sigma, |c|) \|u(t)\|^2_{L^2}.
	\]
	Thanks to the conservation of mass and energy, we obtain
	\[
	\|\nabla u(t)\|^2_{L^2} \leq 4 E(u_0) + C(\sigma, |c|) M(u_0),
	\]
	for any $t$ in the existence time. Note that the constant $C(\sigma, |c|)$ may change from lines to lines. 
	
	In the focusing case, we consider two cases. 
	
	{\bf Subcase 1: Mass-subcritical case.} By \eqref{gag-nir-ine-app} and \eqref{har-ine-app}, we have that
	\begin{align*}
	E(u_0)=E(u(t)) &=\frac{1}{2} \|\nabla u(t)\|^2_{L^2} - \frac{|c|}{2} \||x|^{-\sigma} |u(t)|^2\|_{L^1} - \frac{1}{\alpha+2} \|u(t)\|^{\alpha+2}_{L^{\alpha+2}} \\
	&\leq \left( \frac{1}{4} - \varep \right) \|\nabla u(t)\|^2_{L^2} - C(\sigma, |c|) M(u_0) - C(\varep, M(u_0)).
	\end{align*}
	Taking $0<\varep<\frac{1}{4}$, we see that 
	\[
	\|\nabla u(t)\|^2_{L^2} \leq C(\sigma, |c|, \varep, E(u_0), M(u_0)),
	\]
	for any $t$ in the existence time.
	
	{\bf Subcase 2: Mass-critical case.} In this case, instead of \eqref{har-ine-app}, we use the following estimate
	\[
	\||x|^{-\sigma} |u(t)|^2\|_{L^1} \leq \frac{\varep}{|c|} \|\nabla u(t)\|^2_{L^2}  + C(\sigma, |c|, \varep) \|u(t)\|^2_{L^2}
	\]
	which is valid for any $\varep>0$. Using this inequality and the sharp Gagliardo-Nirenberg inequality, the conservation of mass and energy imply that
	\begin{align*}
	E(u_0) = E(u(t)) &= \frac{1}{2} \|\nabla u(t)\|^2_{L^2} - \frac{|c|}{2}\||x|^{-\sigma} |u(t)|^2 \|_{L^1} - \frac{d}{2d+4}\|u(t)\|^{\frac{4}{d}+2}_{L^{\frac{4}{d}+2}} \\
	&\geq \frac{1}{2} \|\nabla u(t)\|^2_{L^2} - \frac{\varep}{2} \|\nabla u(t)\|^2_{L^2}  - C(\sigma, |c|, \varep) \|u(t)\|^2_{L^2} - \frac{1}{2} \left( \frac{\|u(t)\|_{L^2}}{\|Q\|_{L^2}} \right)^{\frac{4}{d}} \|\nabla u(t)\|^2_{L^2} \\
	&=\frac{1}{2} \left( 1-\varep - \left( \frac{\|u_0\|_{L^2}}{\|Q\|_{L^2}} \right)^{\frac{4}{d}} \right) \|\nabla u(t)\|^2_{L^2} - C(\sigma, |c|, \varep) M(u_0).
	\end{align*}
	Since $\|u_0\|_{L^2} <\|Q\|_{L^2}$, if we take $0<\varep <1-\left( \frac{\|u_0\|_{L^2}}{\|Q\|_{L^2}} \right)^{\frac{4}{d}}$, then we get the uniform bound on $\|\nabla u(t)\|_{L^2}$. The proof is complete.
	\hfill $\Box$
	
	\section{Blow-up}
	\label{S7}
	In this section, we prove the finite time blow-up for \eqref{NLS-rep} in the focusing case.
	
	\noindent {\bf Proof of Theorem $\ref{theo-blo-up}$.}
		We will consider separately the mass-critical, intercritical and energy-critical cases.
		
		{\bf (1) Mass-critical case.} 
		
		{\bf Subcase 1: $d\geq 1$ and $|x|u_0 \in L^2$.} In this case, we assume that $E(u_0)<0$. Applying the virial identity \eqref{vir-ide} with $\alpha=\frac{4}{d}$ and recalling that $0<\sigma<\min\{2,d\}$, we get
		\[
		\frac{d^2}{dt^2} \|xu(t)\|^2_{L^2} = 16 E(u(t)) - 4c(2-\sigma) \||x|^{-\sigma} |u(t)|^2\|_{L^1} \leq 16 E(u_0) <0
		\]
		for any $t$ in the existence time. The standard argument of Glassey \cite{Glassey} implies that the solution blows up in finite time.
		
		{\bf Subcase 2: $d\geq 2$ and $u_0$ is radial.} We assume again that $E(u_0)<0$. By Lemma $\ref{lem-vir-est-2}$, we have that for any $\varep>0$ and any $t$ in the existence time
		\begin{align*}
		\frac{d^2}{dt^2} \Vc_{\varphi_R}(t) \leq 16E(u_0) &- 4 \int \left(\psi_{1,R} - C \varep \psi_{2,R}^{\frac{d}{2}} \right) |\nabla u(t)|^2 dx \\
		&+ O \left( R^{-2} + \varep R^{-2} + \varep^{-\frac{1}{d-1}} R^{-2} \right),
		\end{align*}
		where
		\[
		\psi_{1,R}=2-\varphi''_R, \quad \psi_{2,R} =2d - \Delta \varphi_R.
		\]
		Assume at the moment that there exists a suitable radial function $\varphi_R$ defined by \eqref{def-var-R} so that
		\begin{align} \label{pos}
		\psi_{1,R}- C\varep \psi_{2,R}^{\frac{d}{2}} \geq 0, \quad \forall r\geq 0
		\end{align}
		for a sufficiently small $\varep>0$. We then choose $R>0$ sufficiently large depending on $\varep$ such that
		\[
		\frac{d^2}{dt^2} \Vc_{\varphi_R}(t) \leq 8 E(u_0) <0
		\]
		for any $t$ in the existence time. By Glassey's argument, the solution blows up in finite time. We now show \eqref{pos}. To this end, we define 
		\[
		\zeta(r):= \left\{
		\begin{array}{ccc}
		2r &\text{if}& 0 \leq r \leq 1, \\
		2[r-(r-1)^3] &\text{if}& 1<r \leq 1+1/\sqrt{3}, \\
		\text{smooth and } \zeta' <0 &\text{if}& 1 + 1/\sqrt{3} <r<2, \\
		0 &\text{if}& r\geq 2,
		\end{array}
		\right.
		\]
		and
		\[
		\chi(r) := \int_0^r \zeta(s) ds.
		\]
		It is obvious that $\chi$ defined above satisfies \eqref{def-chi}. We thus then define $\varphi_R$ as in \eqref{def-var-R}. We will show that \eqref{pos} is fulfilled with this choice of $\varphi_R$. In fact, we have 
		\[
		\psi_{2,R} = 2d-\Delta \varphi_R = 2-\varphi''_R + (d-1)\left(2-\frac{\varphi'_R}{r}\right).
		\]
		
		When $0\leq r \leq R$, \eqref{pos} is obvious since $\psi_{1,R}=\psi_{2,R} =0$.
		
		When $R<r \leq (1+1/\sqrt{3})R$, we have $\psi_{1,R} = 6(r/R-1)^2$ and 
		\[
		\psi_{2,R} = 6(r/R-1)^2 \left(1+ \frac{(d-1) (r/R-1)}{3r/R}  \right) < 6(r/R-1)^2 \left( 1+ \frac{d-1}{3\sqrt{3}}\right).
		\]
		Since $0<r/R-1<1/\sqrt{3}$, we can choose $\varep>0$ sufficiently small so that \eqref{pos} is satisfied.
		
		When $r>(1+1/\sqrt{3}) R$, we see that $\zeta'(r/R) \leq 0$, so $\psi_{1,R}(r) = 2-\varphi''_R (r) \geq 2$. On the other hand, $\psi_{2,R}(r) \leq C$ for some constant $C>0$. Therefore taking $\varep>0$ sufficiently small, we get \eqref{pos}.
		
		{\bf (2) Intercritical case.}
		
		{\bf Subcase 1: $d\geq 1$, $|x|u_0\in L^2$ and $E(u_0)<0$.} Applying \eqref{vir-ide} and using the fact $0<\sigma<\min\{2,d\}$, $d\alpha>4>2\sigma$, we have that
		\[
		\frac{d^2}{dt^2} \|xu(t)\|^2_{L^2} = 4d\alpha E(u(t)) - 2(d\alpha-4) \|\nabla u(t)\|^2_{L^2} - 2c(d\alpha-2\sigma) \||x|^{-\sigma} |u(t)|^2 \|_{L^1} \leq 4d\alpha E(u_0) <0
		\]
		for any $t$ in the existence time. This implies that the solution blows up in finite time.
		
		{\bf Subcase 2: $d\geq 2$, $u_0$ is radial and $E(u_0)<0$.} It follows from Lemma $\ref{lem-vir-est-1}$ that for any $\varep>0$,
		\begin{align*}
		\frac{d^2}{dt^2} \Vc_{\varphi_R}(t) &\leq 4d\alpha E(u(t)) - 2(d\alpha-4) \|\nabla u(t)\|^2_{L^2} - 2c(d\alpha-2\sigma) \||x|^{-\sigma} |u(t)|^2 \|_{L^1} \\
		&\mathrel{\phantom{\leq}} + \left\{
		\renewcommand*{\arraystretch}{1.3}
		\begin{array}{cl}
		O \left( R^{-2} + R^{-2(d-1)} \|\nabla u(t)\|^2_{L^2} \right) &\text{if } \alpha=4 \\
		O \left(R^{-2} + \varep^{-\frac{\alpha}{4-\alpha}} R^{-\frac{2(d-1)\alpha}{4-\alpha}} + \varep \|\nabla u(t)\|^2_{L^2} \right) &\text{if } 0<\alpha<4
		\end{array}
		\right. \\
		& \leq 4d\alpha E(u_0) - 2(d\alpha-4) \|\nabla u(t)\|^2_{L^2} \\
		&\mathrel{\phantom{\leq}} + \left\{
		\renewcommand*{\arraystretch}{1.3}
		\begin{array}{cl}
		O \left( R^{-2} + R^{-2(d-1)} \|\nabla u(t)\|^2_{L^2} \right) &\text{if } \alpha=4 \\
		O \left(R^{-2} + \varep^{-\frac{\alpha}{4-\alpha}} R^{-\frac{2(d-1)\alpha}{4-\alpha}} + \varep \|\nabla u(t)\|^2_{L^2} \right) &\text{if } 0<\alpha<4
		\end{array}
		\right. 
		\end{align*}
		for any $t$ in the existence time. Since $d\alpha>4$, we take $R>0$ sufficiently large if $\alpha=4$ and $\varep>0$ sufficiently small and $R>0$ sufficiently large depending on $\varep$ if $0<\alpha<4$ to obtain that
		\[
		\frac{d^2}{dt^2} \Vc_{\varphi_R}(t) \leq 2d\alpha E(u_0) <0
		\]
		for any $t$ in the existence time. This shows that the solution must blow up in finite time. 
		
		{\bf Subcase 3: $d\geq 1$, $|x|u_0 \in L^2$ and $E(u_0) \geq 0$.} In this case, we assume that \eqref{blo-up-con-int} holds. We claim that under the assumption \eqref{blo-up-con-int}, there exists $\delta>0$ such that 
		\begin{align} \label{neg}
		4d\alpha E(u(t)) - 2(d\alpha-4) \|\nabla u(t)\|^2_{L^2} \leq -\delta
		\end{align}
		for any $t$ in the existence time. Indeed, by \eqref{gwp-int-pro}, the conservation of energy and mass and the first condition in \eqref{blo-up-con-int}, we have that
		\[
		f\left(\|\nabla u(t)\|_{L^2} \|u(t)\|^{\betc}_{L^2} \right) \leq E(u_0) M^{\betc}(u_0) < E_0(Q) M^{\betc}(Q) = f\left(\|\nabla Q\|_{L^2} \|Q\|^{\betc}_{L^2} \right)
		\]
		for any $t$ in the existence time. The continuity argument together with the second condition in \eqref{blo-up-con-int} imply that
		\begin{align} \label{neg-pro-1}
		\|\nabla u(t)\|_{L^2} \|u(t)\|^{\betc}_{L^2} > \|\nabla Q\|_{L^2} \|Q\|^{\betc}_{L^2}
		\end{align}
		for any $t$ in the existence time. Since $E(u_0) M^{\betc}(u_0)<E_0(Q)M^{\betc}(Q)$, we pick $\rho>0$ small enough so that
		\begin{align} \label{neg-pro-2}
		E(u_0) M^{\betc}(u_0) \leq (1-\rho) E_0(Q) M^{\betc}(Q).
		\end{align}
		Denote the left hand side of \eqref{neg} by $K(u(t))$. Multiplying $K(u(t))$ with the conserved quantity $M^{\betc}(u(t))$ and using \eqref{ide-Q-app}, \eqref{neg-pro-1} and \eqref{neg-pro-2}, we obtain
		\begin{align*}
		K(u(t)) M^{\betc}(u(t)) &= 4d\alpha E(u(t)) M^{\betc}(u(t)) - 2(d\alpha-4) \left( \|\nabla u(t)\|_{L^2} \|u(t)\|^{\betc}_{L^2} \right)^2 \\
		&= 4d\alpha E(u_0) M^{\betc}(u_0) - 2(d\alpha-4) \left( \|\nabla u(t)\|_{L^2} \|u(t)\|^{\betc}_{L^2} \right)^2 \\
		&\leq 4d\alpha(1-\rho) E_0(Q) M^{\betc}(Q) - 2(d\alpha-4) \left( \|\nabla Q\|_{L^2} \|Q\|^{\betc}_{L^2} \right)^2 \\
		&= - 2(d\alpha-4) \rho \left( \|\nabla Q\|_{L^2} \|Q\|^{\betc}_{L^2} \right)^2
		\end{align*}
		for any $t$ in the existence time. This shows \eqref{neg} with 
		\[
		\delta = 2(d\alpha-4)\rho \|\nabla Q\|^2_{L^2} \left(\frac{M(Q)}{M(u_0)} \right)^{\betc}>0.
		\]
		We now apply \eqref{vir-ide} with the fact $0<\sigma<\min\{2,d\}$, $d\alpha>4>2\sigma$ and \eqref{neg} to get 
		\begin{align*}
		\frac{d^2}{dt^2} \|xu(t)\|^2_{L^2} &= 4d\alpha E(u(t)) - 2(d\alpha-4) \|\nabla u(t)\|^2_{L^2} - 2c (d\alpha -2\sigma) \||x|^{-\sigma} |u(t)|^2\|_{L^1} \\
		&\leq 4d\alpha E(u(t)) - 2(d\alpha-4) \|\nabla u(t)\|^2_{L^2} \leq -\delta <0
		\end{align*}
		for any $t$ in the existence time. This shows that the solution blows up in finite time.
		
		{\bf Subcase 4: $d\geq 2$, $u_0$ is radial and $E(u_0)\geq 0$.} We again assume \eqref{blo-up-con-int} in this case. Under the assumption \eqref{blo-up-con-int}, it follows that for $\varep>0$ small enough, there exists $\delta(\varep)>0$ such that
		\begin{align} \label{neg-rad}
		4d\alpha E(u(t)) - 2(d\alpha -4) \|\nabla u(t)\|^2_{L^2} + \varep \|\nabla u(t)\|^2_{L^2} \leq -\delta(\varep)
		\end{align} 
		for any $t$ in the existence time. This is proved by the same lines as in the proof of \eqref{neg}, just take $0<\varep<2(d\alpha-4)\rho$ and we get 
		\[
		\delta(\varep) = [2(d\alpha-4)\rho-\varep] \|\nabla Q\|^2_{L^2} \left(\frac{M(Q)}{M(u_0)} \right)^{\betc}>0.
		\]		
		Next by Lemma $\ref{lem-vir-est-1}$, we have that for any $\varep>0$,
		\begin{align*}
		\frac{d^2}{dt^2} \Vc_{\varphi_R}(t) &\leq 4d\alpha E(u(t)) - 2(d\alpha-4) \|\nabla u(t)\|^2_{L^2} - 2c(d\alpha-2\sigma) \||x|^{-\sigma} |u(t)|^2 \|_{L^1} \\
		&\mathrel{\phantom{\leq}} + \left\{
		\renewcommand*{\arraystretch}{1.3}
		\begin{array}{cl}
		O \left( R^{-2} + R^{-2(d-1)} \|\nabla u(t)\|^2_{L^2} \right) &\text{if } \alpha=4 \\
		O \left(R^{-2} + \varep^{-\frac{\alpha}{4-\alpha}} R^{-\frac{2(d-1)\alpha}{4-\alpha}} + \varep \|\nabla u(t)\|^2_{L^2} \right) &\text{if } 0<\alpha<4
		\end{array}
		\right. \\
		& \leq 4d\alpha E(u(t)) - 2(d\alpha-4) \|\nabla u(t)\|^2_{L^2} \\
		&\mathrel{\phantom{\leq}} + \left\{
		\renewcommand*{\arraystretch}{1.3}
		\begin{array}{cl}
		O \left( R^{-2} + R^{-2(d-1)} \|\nabla u(t)\|^2_{L^2} \right) &\text{if } \alpha=4 \\
		O \left(R^{-2} + \varep^{-\frac{\alpha}{4-\alpha}} R^{-\frac{2(d-1)\alpha}{4-\alpha}} + \varep \|\nabla u(t)\|^2_{L^2} \right) &\text{if } 0<\alpha<4
		\end{array}
		\right. 
		\end{align*}
		for any $t$ in the existence time. Taking $R>0$ sufficiently large when $\alpha=4$, and $\varep>0$ sufficiently small and $R>0$ sufficiently large depending on $\varep$ when $0<\alpha<4$, we obtain from \eqref{neg-rad} that
		\[
		\frac{d^2}{dt^2}\Vc_{\varphi_R}(t) \leq -\frac{\delta(\varep)}{2} <0
		\]
		for any $t$ in the existence time. This again implies that the solution blows up in finite time.
		
		{\bf (3) Energy-critical case.}
		The case $E(u_0)<0$ is similar to the intercritical case. We thus only consider the case $E(u_0)\geq 0$. Note that \eqref{blo-up-con-ene} is assumed in this case. We claim that for $\varep>0$ small enough, there exists $\delta(\varep)>0$ such that
		\begin{align} \label{neg-ene}
		\frac{16d}{d-2} E(u(t)) -\frac{16}{d-2} \|\nabla u(t)\|^2_{L^2} + \varep \|\nabla u(t)\|^2_{L^2} \leq -\delta(\varep)
		\end{align}
		for any $t$ in the existence time. It follows from \eqref{gwp-ene-pro}, the conservation of energy and the first condition in \eqref{blo-up-con-ene} that
		\[
		g(\|\nabla u(t)\|_{L^2}) \leq E(u_0) <E_0(W) = g(\|\nabla W\|_{L^2}).
		\]
		Thanks to the second condition in \eqref{blo-up-con-ene}, we get 
		\begin{align} \label{neg-ene-pro-1}
		\|\nabla u(t)\|_{L^2} > \|\nabla W\|_{L^2}
		\end{align}
		for any $t$ in the existence time. Using again the first condition in \eqref{blo-up-con-ene}, we pick $\rho>0$ small enough so that
		\begin{align} \label{neg-ene-pro-2}
		E(u_0) \leq (1-\rho) E_0(W).
		\end{align}
		By the conservation of energy, \eqref{ide-W-app}, \eqref{neg-ene-pro-1} and \eqref{neg-ene-pro-2}, we estimate the left hand side of \eqref{neg-ene} which is denoted by $K(u(t))$ as
		\begin{align*}
		K(u(t)) &= \frac{16d}{d-2} E(u(t)) - \left(\frac{16}{d-2} - \varep \right) \|\nabla u(t)\|^2_{L^2} \\
		&= \frac{16d}{d-2} E(u_0) - \left(\frac{16}{d-2} - \varep \right) \|\nabla u(t)\|^2_{L^2} \\
		&\leq \frac{16d}{d-2} (1-\rho)E_0(W) - \left(\frac{16}{d-2} - \varep \right) \|\nabla W\|^2_{L^2} \\
		&= -\left(\frac{16 \rho}{d-2} - \varep\right) \|\nabla W\|^2_{L^2}
		\end{align*}
		for any $t$ in the existence time. Taking $0<\varep<\frac{16\rho}{d-2}$, we obtain \eqref{neg-ene} with 
		\[
		\delta(\varep) = \left(\frac{16 \rho}{d-2} - \varep\right) \|\nabla W\|^2_{L^2}>0.
		\]
		
		{\bf Subcase 1: $d\geq 3$, $|x| u_0 \in L^2$ and $E(u_0)\geq 0$.} Applying the virial identity \eqref{vir-ide} with $\alpha=\frac{4}{d-2}$ and using \eqref{neg-ene} with $\varep=0$, there exists $\delta>0$ such that
		\begin{align*}
		\frac{d^2}{dt^2} \|xu(t)\|^2_{L^2} &= \frac{16d}{d-2} E(u(t)) - \frac{16}{d-2} \|\nabla u(t)\|^2_{L^2} - 2c \left( \frac{4d}{d-2} -2\sigma\right) \||x|^{-\sigma} |u(t)|^2 \|_{L^1} \\
		&\leq \frac{16d}{d-2} E(u(t)) - \frac{16}{d-2} \|\nabla u(t)\|^2_{L^2} \leq -\delta<0
		\end{align*}
		for any $t$ in the existence time. This implies that the solution must blow up in finite time.
		
		{\bf Subcase 2: $d\geq 3$, $u_0$ is radial and $E(u_0)\geq 0$.} Applying Lemma $\ref{lem-vir-est-1}$ with $\alpha=\frac{4}{d-2}$, we have that for any $\varep>0$,
		\begin{align*}
		\frac{d^2}{dt^2} \Vc_{\varphi_R}(t) &\leq \frac{16d}{d-2} E(u(t)) - \frac{16}{d-2} \|\nabla u(t)\|^2_{L^2} - 2c \left( \frac{4d}{d-2} - 2\sigma\right) \||x|^{-\sigma} |u(t)|^2 \|_{L^1} \\
		&\mathrel{\phantom{\leq}} + \left\{
		\renewcommand*{\arraystretch}{1.3}
		\begin{array}{cl}
		O \left( R^{-2} + R^{-4} \|\nabla u(t)\|^2_{L^2} \right) &\text{if } d=3 \\
		O \left(R^{-2} + \varep^{-\frac{1}{d-3}} R^{-\frac{2(d-1)}{d-3}} + \varep \|\nabla u(t)\|^2_{L^2} \right) &\text{if } d\geq 4
		\end{array}
		\right. \\
		& \leq \frac{16d}{d-2} E(u(t)) - \frac{16}{d-2} \|\nabla u(t)\|^2_{L^2}  \\
		&\mathrel{\phantom{\leq}} + \left\{
		\renewcommand*{\arraystretch}{1.3}
		\begin{array}{cl}
		O \left( R^{-2} + R^{-4} \|\nabla u(t)\|^2_{L^2} \right) &\text{if } d=3 \\
		O \left(R^{-2} + \varep^{-\frac{1}{d-3}} R^{-\frac{2(d-1)}{d-3}} + \varep \|\nabla u(t)\|^2_{L^2} \right) &\text{if } d\geq 4
		\end{array}
		\right. 
		\end{align*}
		for any $t$ in the existence time. Choosing $R>0$ sufficiently large when $d=3$, and $\varep>0$ sufficiently small and $R>0$ sufficiently large depending on $\varep$ when $d\geq 4$, it follows from \eqref{neg-ene} that
		\[
		\frac{d^2}{dt^2} \Vc_{\varphi_R}(t) \leq -\frac{\delta(\varep)}{2}<0
		\]
		for any $t$ in the existence time. The solution thus blows up in finite time. 
	\hfill $\Box$
	
	\section{Scattering}
	\label{S8}
	The main purpose of this section is to prove the energy scattering for \eqref{NLS-rep} in the defocusing case. Let us start with the following result which is a consequence of the interaction Morawetz inequality.
	
	\begin{proposition}
		Let $d\geq 3$, $c>0$, $0<\sigma<2$. If $u$ is the global $H^1$ solution to the defocusing \eqref{NLS-rep}, then it holds that
		\begin{align} \label{glo-int-mor-bou}
		\|u\|_{L^{d+1}(\R, L^{\frac{2(d+1)}{d-1}})} \leq C(E,M) <\infty.
		\end{align}
	\end{proposition}

	\begin{proof}
	Applying Proposition $\ref{prop-int-mor-ine}$ with $V(x)=c|x|^{-\sigma}, c>0, 0<\sigma<2$ and $W(x)=1$, we get the following interaction Morawetz inequality for the defocusing \eqref{NLS-rep} in dimensions $d\geq 3$
	\[
	\| |\nabla|^{-\frac{d-3}{4}} u \|_{L^4(\R,L^4)} \leq C\|u\|^{\frac{3}{4}}_{L^\infty(\R,L^2)} \|\nabla u\|^{\frac{1}{4}}_{L^\infty(\R,L^2)}.
	\]
	This implies that
	\begin{align*}
	\|u\|_{L^{d+1}(\R, L^{\frac{2(d+1)}{d-1}})} &\lesssim \||\nabla|^{-\frac{d-3}{4}} u\|_{L^4(\R,L^4)}^{\frac{4}{d+1}} \|\nabla u\|_{L^\infty(\R,L^2)}^{\frac{d-3}{d+1}} \\
	&\lesssim \left(\|u\|_{L^\infty(\R,L^2)}^{\frac{3}{4}} \|\nabla u\|_{L^\infty(\R,L^2)}^{\frac{1}{4}} \right)^{\frac{4}{d+1}} \|\nabla u\|_{L^\infty(\R,L^2)}^{\frac{d-3}{d+1}} \\
	&\lesssim \|u\|_{L^\infty(\R,L^2)}^{\frac{3}{d+1}} \|\nabla u\|_{L^\infty(\R,L^2)}^{\frac{d-2}{d+1}}.
	\end{align*}
	By the conservation of mass and energy, we obtain the global bound \eqref{glo-int-mor-bou} for the defocusing \eqref{NLS-rep} in dimensions $d\geq 3$.
	\end{proof}

	Due to the equivalence between Sobolev norms $\|\cdot\|_{W^{1,q}_c}$ and $\|\cdot\|_{W^{1,q}}$, to show energy scattering, we need to define 
	\[
	\|u\|_{S^1(J)}:= \sup_{(p,q) \in S, ~ 2\leq q<\frac{2d}{\sigma}} \|\scal{\nabla} u\|_{L^p(J,L^q)}.
	\]
	Let us start with the following nonlinear estimates.
	
	\begin{lemma} \label{lem-non-est-1}
		Let 
		\[
		\begin{cases}
		0<\sigma<2 &\text{if } d\geq 4, \\
		0<\sigma \leq 1 &\text{if } d=3,
		\end{cases}
		\quad \text{and} \quad \frac{4}{d}<\alpha <\frac{4}{d-2}.
		\]
		Then there exists $\varep>0$ small enough such that for any time interval $J$,
		\[
		\|\scal{\nabla} (|u|^\alpha u)\|_{L^2(J,L^{\frac{2d}{d+2}})} \lesssim \|u\|_{S^1(J)} \|u\|^{\alpha\theta_1}_{L^{d+1}(J,L^{\frac{2(d+1)}{d-1}})} \|u\|^{\alpha(1-\theta_1)\theta_2}_{L^\infty(J,L^2)} \|\nabla u\|^{\alpha(1-\theta_1)(1-\theta_2)}_{L^\infty(J,L^2)},
		\]
		for some $0<\theta_1 = \theta_1(\varep), \theta_2 = \theta_2(\varep)<1$.
	\end{lemma}
		
	\begin{proof}
		By the fractional chain rule, we estimate
		\begin{align} \label{pro}
		\|\scal{\nabla}(|u|^\alpha u)\|_{L^2(J,L^{\frac{2d}{d+2}})} \lesssim \|\scal{\nabla} u\|_{L^{2+\varep}(J, L^{\frac{2d(2+\varep)}{d(2+\varep)-4}})} \|u\|^\alpha_{L^{\frac{2\alpha(2+\varep)}{\varep}}(J, L^{\frac{d\alpha(2+\varep)}{4+\varep}})}.
		\end{align}
		We next bound
		\[
		\|u\|_{L^{\frac{2\alpha(2+\varep)}{\varep}}(J, L^{\frac{d\alpha(2+\varep)}{4+\varep}})} \leq \|u\|^{\theta_1}_{L^{d+1}(J,L^{\frac{2(d+1)}{d-1}})} \|u\|^{1-\theta_1}_{L^\infty(J,L^n)}
		\]
		provided that $\theta_1 = \frac{\varep(d+1)}{2\alpha(2+\varep)}$ and $n\geq 1$ satisfies
		\[
		\frac{4+\varep}{d\alpha(2+\varep)} = \frac{(d-1)\theta_1}{2(d+1)} + \frac{1-\theta_1}{n}.
		\]
		We continue to bound
		\[
		\|u\|_{L^\infty(J, L^n)} \leq \|u\|^{\theta_2}_{L^\infty(J,L^2)} \|u\|^{1-\theta_2}_{L^\infty(J,L^{\frac{2d}{d-2}})} \lesssim \|u\|^{\theta_2}_{L^\infty(J,L^2)} \|\nabla u\|^{1-\theta_2}_{L^\infty(J,L^2)}
		\]
		provided that $\frac{1}{n} = \frac{\theta_2}{2} + \frac{(d-2)(1-\theta_2)}{2d}$. We thus obtain 
		\[
		\|u\|^\alpha_{L^{\frac{2\alpha(2+\varep)}{\varep}}(J,L^{\frac{d\alpha(2+\varep)}{4+\varep}})} \lesssim \|u\|^{\frac{\varep(d+1)}{2(2+\varep)}}_{L^{d+1}(J,L^{\frac{2(d+1)}{d-1}})} \|u\|^{a(\varep)}_{L^\infty(J,L^2)} \|\nabla u\|^{b(\varep)}_{L^\infty(J,L^2)},
		\]
		where
		\begin{align*}
		a(\varep) &= \alpha(1-\theta_1) \theta_2 =  \frac{d+2}{2} - \frac{2(d-2)(\alpha+1)+ \varep(d+1+ (d-2)\alpha)}{2(2+\varep)}, \\
		b(\varep) &= \alpha(1-\theta_1)(1-\theta_2) = -\frac{d+2}{2} + \frac{2(d\alpha+d-2)+d\alpha\varep}{2(2+\varep)}.
		\end{align*}
		The above estimates are valid provided that $a(\varep)>0$ and $b(\varep)>0$. Since $\varep \mapsto a(\varep)$ and $\varep \mapsto b(\varep)$ are decreasing, by taking $\varep>0$ small enough, it suffices to show the limits as $\varep \rightarrow 0$ are positive. Note that
		\[
		\lim_{\varep \rightarrow 0} a(\varep) = \frac{d+2}{2} - \frac{(d-2)(\alpha+1)}{2}, \quad \lim_{\varep \rightarrow 0} b(\varep) = -\frac{d+2}{2} + \frac{d\alpha+d-2}{2}
		\]
		are positive due to the fact $\frac{4}{d}<\alpha<\frac{4}{d-2}$. Finally, to ensure the first factor in the right hand side of \eqref{pro} is bounded by $\|u\|_{S^1(J)}$, we need $\frac{2d(2+\varep)}{d(2+\varep)-4} <\frac{2d}{\sigma}$ for $\varep>0$ small enough. This gives the restriction on $\sigma$ and the proof is complete.
	\end{proof}
	
	\begin{lemma} \label{lem-non-est-2}
		Let $d=3$, $1<\sigma<2$, $\frac{4}{3}<\alpha<4$ and set $\varep_0:=\frac{2\sigma-2}{3-\sigma}$.	Then there exists $\varep>\varep_0$ small enough such that for any time interval $J$,
		\[
		\|\scal{\nabla}(|u|^\alpha u)\|_{L^2(J,L^{\frac{6}{5}})} \lesssim \|u\|_{S^1(J)}^{1+\alpha(1-\theta_2)} \|u\|_{L^4(J, L^4)}^{\alpha\theta_1 \theta_2} \|u\|_{L^\infty(J, H^1)}^{\alpha(1-\theta_1)\theta_2},
		\]
		for some $0<\theta_1=\theta_1(\varep) <1$ and $0<\theta_2=\theta_2(\varep) \leq 1$. 
	\end{lemma}
	\begin{proof}
		We estimate
		\[
		\|\scal{\nabla}(|u|^\alpha u)\|_{L^2(J,L^{\frac{6}{5}})} \lesssim \|\scal{\nabla} u\|_{L^{2+\varep}(J, L^{\frac{6(2+\varep)}{3\varep+2}})} \|u\|^\alpha_{L^{\frac{2\alpha(2+\varep)}{\varep}}(J, L^{\frac{3\alpha(2+\varep)}{4+\varep}})}.
		\]
		We first need $\frac{6(2+\varep)}{3\varep +2} <\frac{6}{\sigma}$ to ensure the first factor in the right hand side is bounded by $\|u\|_{S^1(J)}$, and it requires $\varep>\varep_0:= \frac{2\sigma-2}{3-\sigma}$. We now denote
		\[
		(p,q) := \left(\frac{2\alpha(2+\varep)}{\varep}, \frac{3\alpha(2+\varep)}{4+\varep}\right).
		\]
		It is easy to check that $(p,q) \in \Lambda_{\gamc}$, i.e.
		\[
		\frac{2}{p}+\frac{3}{q}=\frac{3}{2}-\gamc, \quad \gamc= \frac{3}{2}-\frac{2}{\alpha}.
		\]
		Since we are considering $\alpha \in \left(\frac{4}{3}, 4\right)$, we see that $\gamc \in (0,1)$. 
		\begin{center}
			\begin{figure}[ht] \label{fig:adm-pai-3d}
				\begin{tikzpicture} [scale=8]
				\draw [->] (-0.05,0) -- (0.6,0);
				\draw [->] (0,-0.05) -- (0, 0.6);
				\draw (-0.03, -0.03) node {$0$};
				\draw (1/6, 0) node[below] {$\frac{1}{6}$};
				\draw (1/4, 0) node[below] {$\frac{1}{4}$};
				\draw (11/36, 0) node[below] {$\frac{1}{q}$};
				\draw (5/12, 0) node[below] {$\frac{5}{12}$};
				\draw (1/2, 0) node[below] {$\frac{1}{2}$};
				\draw (0, 1/6) node[left] {$\frac{1}{p}$};
				\draw (0, 1/4) node[left] {$\frac{1}{4}$};
				\draw (0, 1/2) node[left] {$\frac{1}{2}$};
				\draw [thick] (1/6,1/2) -- (1/2,0);
				\draw (1/3, 1/4) node[right] {$\Lambda_0$};
				\draw [thick, blue] (1/12,1/2) -- (5/12,0);
				\draw[blue] (1/3, 1/8) node[right] {$\Lambda_{\frac{1}{4}}$};
				\draw [thick, red] (0,1/4) -- (1/6,0);
				\draw[red] (1/8, 1/16) node[right] {$\Lambda_1$};
				\draw [thick, dashed] (0, 1/4) -- (1/4, 1/4)-- (1/4,0);
				\draw [thick, dashed] (0, 1/6) -- (11/36, 1/6)-- (11/36, 0);
				\end{tikzpicture}
				\caption{Admissible pairs in 3D}
			\end{figure}
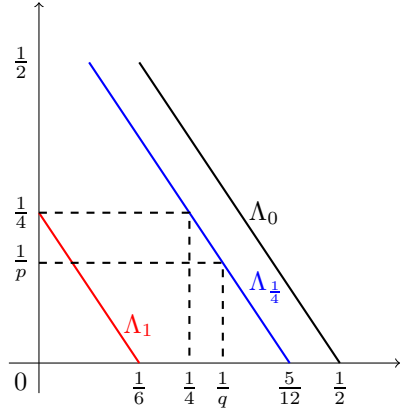
		\end{center}
		\noindent {\bf Case 1: $\gamc= \frac{1}{4}$ or $\alpha =\frac{8}{5}$.} We use H\"older's inequality to have
		\[
		\|u\|_{L^p(J,L^q)} \leq \|u\|_{L^4(J,L^4)}^{\theta_1} \|u\|^{1-\theta_1}_{L^\infty(J, L^{\frac{12}{5}})} \lesssim \|u\|_{L^4(J,L^4)}^{\theta_1} \|u\|^{1-\theta_1}_{L^\infty(J, H^1)}. 
		\]
		To make the above estimates valid, we need to check that $\theta_1 \in (0,1)$. Note that $\theta_1=\frac{4}{p} = \frac{5\varep}{4(2+\varep)}$. By choosing $\varep>\varep_0$ small enough, the condition $\theta_1 \in (0,1)$ implies that $\varep_0<8$ which is satisfied for $1<\sigma<2$. \\
		\noindent {\bf Case 2: $\gamc \in \left(0, \frac{1}{4}\right)$ or $\frac{4}{3} <\alpha <\frac{8}{5}$.} In this case, there exist $q_1, q_2$ such that $q_1<q<q_2$ and
		\begin{align} \label{pro-1}
		(p,q_1) \in \Lambda_0, \quad (p,q_2) \in \Lambda_{\frac{1}{4}}.
		\end{align}
		We thus obtain for some $\theta_1, \theta_2 \in (0,1)$ that
		\begin{align} \label{pro-2}
		\begin{aligned}
		\|u\|_{L^p(J,L^q)} &\leq \|u\|^{1-\theta_2}_{L^p(J,L^{q_1})} \|u\|^{\theta_2}_{L^p(J,L^{q_2})} \\
		&\leq \|u\|^{1-\theta_2}_{L^p(J,L^{q_1})} \left(\|u\|^{\theta_1}_{L^4(J,L^4)} \|u\|^{1-\theta_1}_{L^\infty(J,L^{\frac{12}{5}})} \right)^{\theta_2} \\
		&\lesssim \|u\|^{1-\theta_2}_{S^1(J)} \|u\|^{\theta_1\theta_2}_{L^4(J,L^4)} \|u\|^{(1-\theta_1)\theta_2}_{L^\infty(J,H^1)}.
		\end{aligned}
		\end{align}
		The above estimates are valid provided that 
		\begin{align} \label{con}
		\theta_1,\theta_2 \in (0,1), \quad 2\leq q_1 <\frac{6}{\sigma}.
		\end{align}
		Let us check \eqref{con}. By \eqref{pro-1} and \eqref{pro-2}, we see that $\theta_2 = 4\gamc \in (0,1)$, $\theta_1 = \frac{4}{p} = \frac{2\varep}{\alpha(2+\varep)}$ and $q_1 = \frac{6\alpha(2+\varep)}{6\alpha+(3\alpha-2)\varep}$. Taking $\varep>\varep_0$ small enough, the condition $\theta_1\in (0,1)$ implies that $\alpha>\frac{2\varep_0}{2+\varep_0}$. Since $\alpha \in \left(\frac{4}{3}, \frac{8}{5}\right)$, we need
		\[
		\frac{4}{3} \geq \frac{2\varep_0}{2+\varep_0} \text{ or } \varep_0 \leq 4
		\]
		which is again satisfied for $1<\sigma<2$. The condition $q_1 \geq 2$ is easy to verify. By taking $\varep>\varep_0$ small enough, the condition $q_1<\frac{6}{\sigma}$ implies that
		\[
		6\alpha - 2\alpha \sigma + (3\alpha-2- \alpha \sigma) \varep_0 >0 \text{ or } 3-\sigma + \left( \frac{3}{4} +\frac{\gamc}{2} -\frac{\sigma}{2} \right) \varep_0 >0.
		\]
		If $1<\sigma \leq \frac{3}{2} +\gamc$, then the above inequality is satisfied for any $\varep_0>0$. If $\sigma >\frac{3}{2} +\gamc$, then $\varep_0<\frac{2(3-\sigma)}{\sigma-\frac{3}{2}-\gamc}$. Since $\gamc>0$, if we take 
		\[
		\varep_0 \leq \frac{2(3-\sigma)}{\sigma-\frac{3}{2}}
		\]
		which is also satisfied for $\frac{3}{2} + \gamc <\sigma<2$, then the above inequality holds true. Therefore, the requirement \eqref{con} is verified for $1<\sigma<2$.\\
		\noindent {\bf Case 3: $\gamc \in \left(\frac{1}{4}, 1\right)$ or $\frac{8}{5} <\alpha<4$.} There exist $q_1, q_2$ such that $q_1<q<q_2$ and
		\[
		(p,q_1) \in \Lambda_{\frac{1}{4}}, \quad (p,q_2) \in \Lambda_1.
		\] 
		By H\"older's inequality and Sobolev embedding, we obtain for some $\theta_1, \theta_2 \in (0,1)$ that
		\begin{align*}
		\|u\|_{L^p(J,L^q)} &\leq \|u\|^{\theta_2}_{L^p(J,L^{q_1})} \|u\|^{1-\theta_2}_{L^p(J,L^{q_2})} \\
		&\lesssim \left(\|u\|^{\theta_1}_{L^4(J,L^4)} \|u\|^{1-\theta_1}_{L^\infty(J,L^{\frac{12}{5}})} \right)^{\theta_2} \|\scal{\nabla} u\|^{1-\theta_2}_{L^p(J,L^{q_3})} \\
		&\lesssim \|u\|^{\theta_1\theta_2}_{L^4(J,L^4)} \|u\|^{(1-\theta_1)\theta_2}_{L^\infty(J,H^1)} \|u\|^{1-\theta_2}_{S^1(J)},
		\end{align*}
		where $(p,q_3) \in \Lambda_0$ with $\frac{1}{q_2}=\frac{1}{q_3}-\frac{1}{3}$. The above estimates hold true provided that
		\[
		\theta_1, \theta_2 \in (0,1), \quad 2\leq q_3 <\frac{6}{\sigma}.
		\]
		We see that $\theta_2 = \frac{4}{3}(1-\gamc) \in (0,1)$, $\theta_1 = \frac{4}{p} = \frac{2\varep}{\alpha(2+\varep)}$ and $q_3 = \frac{6\alpha(2+\varep)}{6\alpha+(3\alpha-2)\varep}$. Arguing as in Case 2, we see that the above conditions are satisfied for $1<\sigma<2$. The proof is complete.
	\end{proof}
	
	We are now able to prove the energy scattering for the defocusing \eqref{NLS-rep} given in Theorem $\ref{theo-ene-sca}$.
	
	\noindent {\bf Proof of Theorem $\ref{theo-ene-sca}$.}
		We first show that the global Morawetz bound \eqref{glo-int-mor-bou} implies the global Strichartz bound
		\begin{align} \label{glo-str-bou}
		\|u\|_{S^1(\R)} \leq C(E,M) <\infty.
		\end{align}
		To see this, we decompose $\R$ into a finite number of disjoint intervals $J_k=[t_k, t_{k+1}], k=1, \cdots, N$ so that
		\begin{align} \label{spl-del}
		\|u\|_{L^{d+1}(J_k, L^{\frac{2(d+1)}{d-1}})} \leq \delta, \quad k=1, \cdots, N,
		\end{align}
		for some small constant $\delta>0$ to be chosen later. By Strichartz estimates given in Theorem $\ref{theo-glo-str}$ and the equivalence $\|\cdot\|_{W^{1,q}_c} \sim \|\cdot\|_{W^{1,q}}$, we have that
		\begin{align*}
		\|u\|_{S^1(J_k)} \lesssim \|u(t_k)\|_{H^1} + \|\scal{\nabla} (|u|^\alpha u)\|_{L^2(J_k, L^{\frac{2d}{d+2}})}. 
		\end{align*}
		For $0<\sigma<2$ if $d\geq 4$ ($0<\sigma \leq 1$ if $d=3$), we learn from Lemma $\ref{lem-non-est-1}$ that
		\[
		\|\scal{\nabla} (|u|^\alpha u)\|_{L^2(J_k, L^{\frac{2d}{d+2}})} \lesssim \|u\|_{S^1(J_k)} \|u\|^{\alpha\theta_1}_{L^{d+1}(J_k,L^{\frac{2(d+1)}{d-1}})} \|u\|_{L^\infty(J_k, H^1)}^{\alpha(1-\theta_1)},
		\]
		for some $0<\theta_1<1$. It follows from \eqref{spl-del} and the conservation of mass and energy that
		\begin{align} \label{ene-sca-pro-1}
		\|u\|_{S^1(J_k)} \lesssim \|u(t_k)\|_{H^1} + \|u\|_{S^1(J_k)} \delta^{\alpha\theta_1}. 
		\end{align}
		For $d=3$ and $1<\sigma <2$, we have from Lemma $\ref{lem-non-est-2}$ that
		\[
		\|\scal{\nabla} (|u|^\alpha u) \|_{L^2(J_k, L^{\frac{6}{5}})} \lesssim \|u\|^{1+\alpha(1-\theta_2)}_{S^1(J_k)} \|u\|^{\alpha \theta_1\theta_2}_{L^4(J_k,L^4)} \|u\|^{\alpha(1-\theta_1)\theta_2}_{L^\infty(J_k,H^1)},
		\]
		for some $0<\theta_1<1$ and $0<\theta_2 \leq 1$. Thus
		\begin{align} \label{ene-sca-pro-2}
		\|u\|_{S^1(J_k)} \lesssim \|u(t_k)\|_{H^1} + \|u\|^{1+\alpha(1-\theta_2)}_{S^1(J_k)} \delta^{\alpha\theta_1\theta_2}.
		\end{align}
		Taking $\delta>0$ small enough, we get from \eqref{ene-sca-pro-1} and \eqref{ene-sca-pro-2} that
		\[
		\|u\|_{S^1(J_k)} \lesssim \|u(t_k)\|_{H^1} \leq C(E,M)<\infty, \quad k =1, \cdots, N.
		\]
		By summing over all intervals $J_k, k=1, \cdots, N$, we obtain \eqref{glo-str-bou}. 
		
		We now show the scattering property. By the time reversal symmetry, it suffices to treat positive times. By Duhamel's formula, we have that
		\[
		e^{itH_c} u(t) = u_0 - i \int_0^t e^{isH_c} |u(s)|^\alpha u(s) ds.
		\]
		Let $0<t_1<t_2$. By Strichartz estimates and the Sobolev norms equivalence,
		\[
		\|e^{it_2H_c} u(t_2) - e^{it_1H_c} u(t_1)\|_{H^1} = \left\|-i\int_{t_1}^{t_2} e^{isH_c} |u(s)|^\alpha u(s) ds \right\|_{H^1} \lesssim \|\scal{\nabla}(|u|^\alpha u)\|_{L^2([t_1,t_2], L^{\frac{2d}{d+2}})}. 
		\]
		For $0<\sigma<2$ if $d\geq 4$ ($0<\sigma \leq 1$ if $d=3$), we have from Lemma $\ref{lem-non-est-1}$ that
		\[
		\|\scal{\nabla} (|u|^\alpha u)\|_{L^2([t_1,t_2], L^{\frac{2d}{d+2}})} \lesssim \|u\|_{S^1([t_1,t_2])} \|u\|^{\alpha\theta_1}_{L^{d+1}([t_1,t_2],L^{\frac{2(d+1)}{d-1}})} \|u\|_{L^\infty([t_1,t_2], H^1)}^{\alpha(1-\theta_1)},
		\]
		for some $0<\theta_1<1$. For $d=3$ and $1<\sigma<2$, Lemma $\ref{lem-non-est-2}$ implies that
		\[
		\|\scal{\nabla} (|u|^\alpha u) \|_{L^2([t_1,t_2], L^{\frac{6}{5}})} \lesssim \|u\|^{1+\alpha(1-\theta_2)}_{S^1([t_1,t_2])} \|u\|^{\alpha \theta_1\theta_2}_{L^4([t_1,t_2],L^4)} \|u\|^{\alpha(1-\theta_1)\theta_2}_{L^\infty([t_1,t_2],H^1)},
		\]
		for some $0<\theta_1<1$ and $0<\theta_2 \leq 1$. Thus, by \eqref{glo-int-mor-bou}, \eqref{glo-str-bou} and the conservation of mass and energy, we see that under our assumptions on $\sigma$ and $\alpha$, 
		\[
		\|e^{it_2H_c} u(t_2) - e^{it_1H_c} u(t_1)\|_{H^1} \rightarrow 0 \text{ as } t_1, t_2 \rightarrow +\infty.
		\]
		Hence the limit
		\[
		u_0^+:= \lim_{t\rightarrow +\infty} e^{itH_c} u(t) = u_0 - i \int_0^{+\infty} e^{isH_c} |u(s)|^\alpha u(s) ds
		\]
		exists in $H^1$. Moreover,
		\[
		u(t) - e^{-itH_c} u(t) = i \int_t^{+\infty} e^{-i(t-s)H_c} |u(s)|^\alpha u(s) ds.
		\]
		Minicing the above estimates, we prove as well that
		\[
		\|u(t)-e^{-itH_c} u(t)\|_{H^1} \rightarrow 0 \text{ as } t \rightarrow +\infty.
		\]
		The proof is complete.
	\hfill $\Box$
	
	\section*{Acknowledgement}
	This work was supported in part by the Labex CEMPI (ANR-11-LABX-0007-01). The author would like to express his deep gratitude to his wife - Uyen Cong for her encouragement and support. He would like to thank Prof. Changxing Miao for fruitful discussions on the energy scattering. He also would like to thank the reviewer for his/her helpful comments and suggestions.

\end{document}